\documentclass[final,3p,11pt]{elsarticle}
\usepackage[T1]{fontenc}
\usepackage{lmodern}
\usepackage{amsmath,amssymb,amsthm,mathtools,bm,mathrsfs}
\usepackage{graphicx,booktabs,array,multirow,ragged2e,placeins,needspace,float}
\usepackage{microtype}
\usepackage{flafter}
\usepackage[hidelinks]{hyperref}
\usepackage{url}
\biboptions{sort&compress}
\newcolumntype{L}[1]{>{\RaggedRight\arraybackslash}p{#1}}
\newcolumntype{R}[1]{>{\RaggedLeft\arraybackslash}p{#1}}
\newcolumntype{C}[1]{>{\Centering\arraybackslash}p{#1}}
\theoremstyle{plain}
\newtheorem{mytheorem}{Theorem}[section]
\newtheorem{mylemma}[mytheorem]{Lemma}
\newtheorem{myproposition}[mytheorem]{Proposition}
\newtheorem{mycorollary}[mytheorem]{Corollary}
\theoremstyle{definition}

\theoremstyle{remark}

\numberwithin{equation}{section}
\newcommand{\Mh}{M_h}
\newcommand{\dd}{\mathrm d}

\newcommand{\paperTitle}{Fixed-grid symmetry defects and order recovery for explicit Laplace--rotation integrators in Bose--Einstein dynamics}
\hypersetup{bookmarksdepth=2,pdftitle={Fixed-grid symmetry defects and order recovery for explicit Laplace-rotation integrators},pdfauthor={Tianqi Zhang and Fei Xue}}
\journal{Applied Numerical Mathematics}
\begin{document}
\begin{frontmatter}
\title{\paperTitle}
\author[zufe]{Tianqi Zhang}
\ead{tianqiz@zufe.edu.cn}
\author[clemson]{Fei Xue\corref{cor}}
\ead{fxue@clemson.edu}
\cortext[cor]{Corresponding author.}
\affiliation[zufe]{organization={School of Data Sciences, Zhejiang University of Finance and Economics},addressline={18 Xueyuan Street},city={Hangzhou},postcode={310018},state={Zhejiang},country={China}}
\affiliation[clemson]{organization={School of Mathematical and Statistical Sciences, Clemson University},addressline={220 Parkway Dr.},city={Clemson},postcode={29634},state={SC},country={USA}}
\begin{abstract}
A factorization exact for the continuous Laplace--rotation flow need not preserve time symmetry when implemented stage by stage on a Fourier grid. We analyze two explicit exact integrators (EEIs) for rotating Bose--Einstein condensates. For the seven-factor EEI, a constant-state Fourier coefficient proves that its fixed-grid logarithm has a nonzero quadratic defect on every centered even grid for any nonzero rotation vector. For a distinct five-operator EEI, a nonconstant-mode certificate proves the same obstruction on nonsingular coefficient branches whenever at least two rotation components are nonzero. In real-coefficient compact splitting, merging adjacent linear stages with fractions $c_j$ multiplies either defect by the strictly positive factor $\kappa=\sum_jc_j^2$; nominal high order therefore cannot be guaranteed uniformly at fixed grid. For the seven-factor implementation, the accumulated-response formula predicts $\kappa_4/\kappa_6\approx0.7606$ for nominal fourth- and sixth-order compact schemes (CS4/CS6). Scalar dipolar and spin-1 tests of that implementation verify the predicted amplitude ratio and reveal a crossover toward first-order compact behavior under time refinement. Pairing a forward linear half-step with its method adjoint around the exact nonlinear flow restores symmetry of the complete step; classical compositions then recover fourth and sixth order at fixed grid. The analysis provides explicit structural diagnostics and a direct structure-preserving repair.
\end{abstract}
\begin{keyword}
Fourier pseudospectral methods \sep time-splitting methods \sep order reduction \sep symmetric methods \sep method adjoints \sep Bose--Einstein condensates
\MSC[2020] 65M70 \sep 65P10 \sep 81Q05
\end{keyword}
\end{frontmatter}

\section{Introduction}\label{sec:intro}

High-order time-splitting formulas are often derived by first identifying exactly solvable subflows and then composing their time maps. One essential assumption is easy to overlook when the subflows are themselves implemented by several numerical operations: \emph{the map to which the composition is applied must have the requisite symmetry on the discrete space}. Exactness of a factorization in the continuous setting does not by itself establish this property for its sampled Fourier realization.

This distinction arises in rotating Bose--Einstein condensate (BEC) dynamics. Compact splittings combine the Laplacian and rotation into one linear subproblem, leaving density and spin interactions to a nonlinear subflow. Liu and Zhang~\cite{LiuZhang2025} and Liu et al.~\cite{LiuMengTangZhang2025} construct explicit exact integrators (EEIs) for the \emph{continuous} Laplace--rotation flow and use compact fourth-order Yoshida-type stage fractions~\cite{Yoshida1990}. Their formulations motivate a natural numerical question: after replacing the continuous EEI factors by the prescribed sequence of Fourier and nodal multipliers, does the resulting finite-dimensional map still possess the cancellation needed for the nominal compact order? The distinction matters because a coordinate-dependent phase can take a resolved trigonometric polynomial outside the Fourier space; subsequent discrete transforms act on its interpolant rather than on the unaltered continuous function.

Fix a Fourier grid $h$, and write $I_h$ for the identity on its nodal space. A step family $\Psi_h(t)$ is \emph{time-symmetric} if $\Psi_h(-t)\circ\Psi_h(t)=I_h$. For a linear family this is different from \emph{unitarity}, $\Psi_h(t)^*\Psi_h(t)=I_h$, where $*$ denotes conjugate transpose: norm preservation of the individual factors does not imply time symmetry of their ordered product. The composition argument needs the appropriate cancellation in the \emph{implemented} map, not merely unitary stages.

Let $\bm\Omega\in\mathbb R^3$ denote the rotation vector, $A_h$ the semidiscrete Laplace--rotation generator, and $E_h(t)$ the \emph{implemented seven-factor} EEI map. The exact semidiscrete linear subflow is $e^{tA_h}$; these maps should not be identified without proof. We use $O_h(t^k)$ for a remainder bounded by $C_h|t|^k$ in a fixed-grid norm, locally uniformly on the stated state neighborhood when the map is nonlinear. Our first result is that, even though $E_h'(0)=A_h$ and each factor is unitary,
\begin{equation}\label{eq:intro-defect}
 \log E_h(t)=tA_h+t^2D_{2,h}+O_h(t^3),\qquad D_{2,h}\ne0
 \quad\text{if }\bm\Omega\ne0.
\end{equation}
The coefficient $D_{2,h}$ is a \emph{fixed-grid quadratic defect}. We exhibit a nonzero low-frequency Fourier coefficient of its action on the constant grid state for every nonzero rotation vector and every centered even grid under consideration. This supplies a structural certificate rather than an inference from a convergence plot. The constant-state certificate is specific to the seven-factor realization. We also analyze the distinct shear-only five-operator EEI of Liu--Wei--Zhang~\cite{LiuWeiZhang2026}. Its grid map fixes constants, so a different, nonconstant Fourier-mode certificate is needed; Theorem~\ref{thm:eei5_certificate} supplies one for rotations with at least two nonzero components, using suitable nonsingular coefficient branches. The two proofs apply different Fourier inputs and establish nonvanishing of the respective full-grid operators, independently of particular BEC trajectories.

Why does a quadratic defect matter for a high-order composition? The compact formulas merge adjacent exact linear exponentials before replacing each merged slot by one EEI call. If the resulting real linear-stage fractions are $c_j$, their quadratic errors add with a factor $\sum_j c_j^2>0$ and therefore cannot cancel. More precisely, their leading fixed-grid \emph{local} error is $t^2(\sum_j c_j^2)D_{2,h}U$. Transporting this contribution along a nonlinear trajectory gives a leading \emph{global} term of order $t$, provided its accumulated response does not vanish. The common coefficient $\sum_jc_j^2$ then predicts an amplitude ratio between nominal fourth- and sixth-order compact schemes, even when a finite range of steps appears high order.

We next repair the symmetry at the level of the \emph{complete time step}. For an invertible forward linear approximation $L_{F,h}$, its \emph{method adjoint} is $L_{F,h}^\dagger(t)=L_{F,h}(-t)^{-1}$, where $\dagger$ does not denote a conjugate transpose. If $N_h(t)$ is the exact nonlinear flow, define
\begin{equation}\label{eq:intro-FB-block}
 S_{2,h}^{\rm FB}(t)=L_{F,h}^\dagger(t/2)\circ N_h(t)\circ L_{F,h}(t/2).
\end{equation}
The rightmost operation acts first. In this terminology, a time-symmetric map is equivalently \emph{self-adjoint as a numerical method}: $\Psi_h^\dagger(t)=\Psi_h(t)$. This is distinct from Hermitian self-adjointness. Both linear half-steps advance the solution forward by half a step; the \emph{backward} in forward--backward (FB) refers to the adjoint return, not negative physical time. The complete map satisfies $(S_{2,h}^{\rm FB})^\dagger=S_{2,h}^{\rm FB}$ and is second order. Classical symmetric composition then yields orders four and six at fixed grid~\cite{Strang1968,HairerLubichWanner2006,mclachlan2002splitting,blanes2008splitting}. Unlike a nested repair that first symmetrizes every linear half-step, this construction requires only a forward linear call and its adjoint per basic block. The adjoint principle is classical~\cite{BlanesDieleMarangiRagni2010}; here it provides a direct repair for the identified discrete obstruction.

The analytical evidence covers both EEI grid maps, while the dynamical tests focus on the seven-factor realization. An arbitrary-axis scalar dipolar model and an axial spin-1 model compare compact fourth-order splitting (CS4), its sixth-order extension (CS6), and the adjoint-paired methods (FB-S4/FB-S6). Both tests reveal a late first-order component in CS, with CS4/CS6 amplitude ratios approaching the theoretical prediction; FB retains its designed order on resolved steps. These tests implement the published seven-factor stage formula on our specified grids, not the original authors' software. The spin-1 comparison uses the exact zero-field nonlinear flow. The very small spinor crossover differences are interpreted as a structural diagnostic rather than a general measure of physical simulation error.

Related Fourier and rotating-coordinate approaches~\cite{bao2003,AntoineBaoBesse2013,BaoWang2006,bao2006rotating,BaoMarahrensTangZhang2013,Bader2018}, exact spin subflows~\cite{BaoCai2018Spinor,SymesMcLachlanBlakie2016,SymesBlakie2017,JainAminPu2023}, and other compact rotation schemes~\cite{ChenYuanZhang2026,LiuYuanZhang2025,LiuXieYuanZhangZhao2026RSDA4,WeiZhang2026} provide broader context; they are not automatically covered by either certificate. Defects from approximate nonlinear subflows are another, distinct issue~\cite{EinkemmerOstermann2014}. Section~\ref{sec:model} establishes only the semidiscrete framework needed for the argument. Section~\ref{sec:eei-defect} proves the obstruction and discusses its interpolation mechanism; Section~\ref{sec:fbstrang} derives the compact consequences and paired repair. The detailed BEC subflows and numerical tests follow in Sections~\ref{sec:linear} and~\ref{sec:experiments}, with derivations and implementation checks in the supplement.

\section{The semidiscrete problem and the symmetry question}\label{sec:model}

The central issue is already present in the linear Laplace--rotation equation. We therefore begin with its Fourier-grid generator and postpone the detailed BEC interactions until they are needed to construct the nonlinear subflows in Subsection~\ref{sec:nonlinear}. This separation is important: the obstruction proved in Section~\ref{sec:eei-defect} depends on the \emph{linear grid realization}, not on the particular spin interaction.

For context, the full class of equations is
\begin{equation}\label{eq:spinor-full-PDE}
i\partial_t\Psi
=\underbrace{\left(-\frac12\Delta-\bm\Omega\cdot\bm L\right)\Psi
+M\Psi-i\sum_{j=x,y,z}A_j\partial_j\Psi}_{\mathcal H_{\rm lin}\Psi}
+\mathcal G(\Psi).
\end{equation}
Here $\bm\Omega\in\mathbb R^3$ is the rotation vector, $\bm L=(L_x,L_y,L_z)$ is angular momentum, and the constant Hermitian matrices $M$ and $A_j$ represent internal couplings and directional spin--orbit coupling (SOC). The term $\mathcal G$ contains the trap and density or spin interactions. For spin $F\in\{1/2,1,2\}$ there are $s=2F+1$ components, $\Psi=(\psi_F,\ldots,\psi_{-F})^T\in\mathbb C^s$; the scalar case has $s=1$ and $M=A_j=0$. The superscripts $T$ and $*$ mean transpose and conjugate transpose, an overbar denotes complex conjugation, and $I_s$ is the $s\times s$ identity matrix. The detailed choices of $M,A_j,\mathcal G$ and the spin conventions are stated in Sections~\ref{subsec:linear-soc} and~\ref{subsec:spinor-nonlinearities}.

In Cartesian coordinates,
\[
L_x=-i(y\partial_z-z\partial_y),\qquad
L_y=-i(z\partial_x-x\partial_z),\qquad
L_z=-i(x\partial_y-y\partial_x).
\]
The scalar dipolar benchmark and the spinor benchmarks share this Laplace--rotation part. Their different nonlinearities will matter for evaluating complete trajectories, but not for proving that the EEI linear map has a nonzero quadratic defect.

\subsection{Fixed-grid semidiscrete system}\label{subsec:fourier-discretization}

On the centered box \(\mathcal D=\prod_{j=x,y,z}[-\ell_j,\ell_j)\), let \(\ell_j>0\) and \(N_j\ge4\) be even. Nodes and spacings are
\begin{equation}\label{eq:grid_convention}
x_{\alpha,j}=-\ell_\alpha+jh_\alpha,\qquad h_\alpha=2\ell_\alpha/N_\alpha,\qquad 0\le j<N_\alpha,\quad \alpha=x,y,z.
\end{equation}
The sampled spinor \(U\) belongs to \(\mathcal X_h=\mathbb C^{N_x\times N_y\times N_z\times s}\), with the component index stored last (a scalar array when \(s=1\)). We write \(h\) for the fixed grid and \(I_h\) for the identity on this space. Let \(\mathcal F_j\) be the discrete Fourier transform (DFT) in direction \(j\), and let \(\mathcal F=\mathcal F_x\mathcal F_y\mathcal F_z\). For even \(N_j\), the FFT-order wavenumber vector \(k_j\) is \((\pi/\ell_j)(0,\ldots,N_j/2-1,-N_j/2,\ldots,-1)\); the negative Nyquist frequency is retained in first derivatives. Define
\begin{equation}\label{eq:derivative_matrices}
D_{j,h}=\mathcal F_j^{-1}\operatorname{diag}(ik_j)\mathcal F_j,
\qquad \Delta_h=\sum_{j=x,y,z}D_{j,h}^2,
\end{equation}
Let $\mathscr T_h$ be the tensor-product space of signed Fourier modes selected by this FFT ordering, and let $\mathcal S_h$ denote sampling at the grid nodes. Retaining the negative Nyquist mode ensures $D_{\alpha,h}\mathcal S_hv=\mathcal S_h(\partial_\alpha v)$ for $v\in\mathscr T_h$. This interpolation identity is used in the representation-error estimate below.
Let \(X,Y,Z\) denote nodal coordinate multiplication. Set \(L_{x,h}=-i(YD_{z,h}-ZD_{y,h})\) and define \(L_{y,h},L_{z,h}\) cyclically. The semidiscrete problem is
\begin{equation}\label{eq:spinor-semidiscrete}
\dot U=\mathcal L_hU+\mathcal N_h(U),\qquad
\mathcal L_h=-i\mathcal H_{{\rm lin},h},\quad \mathcal N_h(U)=-i\mathcal G_h(U),
\end{equation}
where \(\mathcal G_h\) is the corresponding grid nonlinearity and
\begin{equation}\label{eq:spinor-discrete-Hlin}
\mathcal H_{{\rm lin},h}U
=\left(-\frac12\Delta_h-\bm\Omega\cdot\bm L_h\right)U+UM^T-i\sum_jD_{j,h}UA_j^T.
\end{equation}
Coordinate multipliers are sampled on this half-open box. Their periodic extensions are not smooth, so the Fourier implementation is not identified with a continuous whole-space rotation. Nevertheless, the finite-dimensional ODE above is well defined.

Smoothness and state derivatives are understood on the underlying real vector space. Spatial operators act on the grid indices, while \(M,A_j\) act on the spin index. For a stored row \(u=\Psi^T\), one has \(uM^T=(M\Psi)^T\); this explains the transposes in \eqref{eq:spinor-discrete-Hlin}. They do not change the physical coupling matrices.

For grid multi-index \(\mathbf j=(j_x,j_y,j_z)\) and spin index \(r\), use
\[
\langle U,W\rangle_h=h_xh_yh_z\sum_{\mathbf j}\sum_{r=1}^s
U_{\mathbf j,r}\overline{W_{\mathbf j,r}},\qquad
\|U\|_h^2=\Mh(U)=\langle U,U\rangle_h.
\]
This inner product is linear in its first argument, and \(\Mh\) is the discrete mass. Fourier first derivatives are skew-Hermitian; coordinate multiplication commutes with differentiation in a transverse direction. Hence the discrete angular-momentum operators, and therefore \(\mathcal H_{{\rm lin},h}\), are Hermitian. The exact semidiscrete linear flow \(e^{\eta\mathcal L_h}\) is consequently unitary and reversible. The question below is whether a practical factorization realizes the same reversal identity.

The operator needed for the main theorem is the common Laplace--rotation part. Denote its grid generator by
\begin{equation}\label{eq:scalar-generator}
A_h=\frac i2\Delta_h+A_{R,h},\qquad
A_{R,h}=\Omega_x(YD_{z,h}-ZD_{y,h})+\Omega_y(ZD_{x,h}-XD_{z,h})+\Omega_z(XD_{y,h}-YD_{x,h}).
\end{equation}
When $M=A_j=0$, $\mathcal L_h=A_h$ in the scalar problem and $A_h$ acts componentwise in the spinor problem. This is the setting of the compact EEI comparisons. More general linear couplings enter only the paired-construction results. Let \(\varphi_h(t)\) denote the exact flow of the complete semidiscrete equation~\eqref{eq:spinor-semidiscrete}, and set \(\mathfrak f_h(U)=\mathcal L_hU+\mathcal N_h(U)\). In what follows, all estimates are for this finite-dimensional system at fixed \(h\). Smoothness assumptions below always refer to an open neighborhood of the relevant finite-dimensional states.

\section{Quadratic defects of two EEI grid realizations}\label{sec:eei-defect}
We test the symmetry of two concrete maps assembled from Fourier-grid operations. In the seven-factor EEI, a coordinate-dependent phase may leave the resolved trigonometric space; in the five-operator EEI of~\cite{LiuWeiZhang2026}, coordinate-dependent shears create a different interpolation mismatch. The question in each case is whether the \emph{complete} grid map retains the continuous second-order cancellation. Throughout this section the generator is the Laplace--rotation operator \(A_h\); additional spin couplings enter later in Section~\ref{sec:linear}.

\subsection{Continuous factors and their grid realizations}
\label{subsec:eei_formula}\label{subsec:eei-formula}

The continuous EEI factorization used in~\cite{LiuZhang2025}, following~\cite{bernier2021exact,bernier2021kinetic}, is
\begin{align}
 E(\tau)&=\exp\bigl(i\tau(\Delta/2+\bm\Omega\cdot\bm L)\bigr)\nonumber\\
 &=e^{-i\tau r_\tau}
 e^{-\tau(\xi_1y+\xi_2z)\partial_x}
 e^{-\tau\xi_3z\partial_y}
 e^{i\tau\nabla^T\mathsf K(\tau)\nabla}
 e^{-\tau\eta_1x\partial_y}
 e^{-\tau(\eta_2x+\eta_3y)\partial_z}
 e^{-i\tau q_\tau},
 \label{eq:continuous_eei_formula}
\end{align}
where $\bm x=(x,y,z)^T$ and
\[
 q_\tau(\bm x)=\bm x^T\mathsf P(\tau)\bm x,
 \qquad r_\tau(\bm x)=\zeta_1x^2+\zeta_2y^2+\zeta_3z^2.
\]
The real scalars \(\xi_j,\eta_j,\zeta_j\) and real symmetric matrices \(\mathsf P,\mathsf K\) depend on \(\bm\Omega\) and the signed step \(\tau\). Appendix~\ref{app:coefficients} lists their formulas and removable values at zero. They are analytic near zero on the nonsingular branch considered here.

The factorization and coefficients are those of Liu and Zhang~\cite[Eqs.~(2.3)--(2.7)]{LiuZhang2025}; the grid map below corresponds to their Eq.~(2.12). For axial rotation, two shear stages become identities, leaving five nonidentity operator factors as in Liu et al.~\cite[Eq.~(2.11)]{LiuMengTangZhang2025}. This axial specialization of the \emph{seven-factor} construction is not the distinct shear-only EEI5 of Liu--Wei--Zhang~\cite{LiuWeiZhang2026}. Algorithm~1 of~\cite{LiuMengTangZhang2025} applies the same EEI half-step before and after its nonlinear update. The linear part of its time-reversal proof \cite[Lemma~2.3]{LiuMengTangZhang2025} uses the inverse identity for the continuous exponential in Eq.~(2.2). The additional question here is whether the finite Fourier realization satisfies that identity. Our spinor comparisons take the zero-field, zero-Zeeman specialization, where the nonlinear update is exact. The general magnetic-field Magnus subsolver is not tested.

Factors are numbered in application order, beginning with \(Q_{1,h}\). The symmetric quadratic-form representatives in Appendix~\ref{app:coefficients} leave the exponents unchanged. The continuous identity is understood on the whole space; a rectangular periodic box is not invariant under arbitrary rotations.

For a grid function \(f\), let \(\mathsf{Diag}_h(f)\) denote diagonal multiplication by its nodal values. For a real nodal function \(q\), and a real shear coefficient \(c\) independent of the differentiated coordinate \(x_\alpha\), define \(\mathcal P_h(q;\tau)=\mathsf{Diag}_h(e^{-i\tau q})\) and 
\begin{align}
 \mathcal T_{\alpha,h}(c;\tau)=\mathcal F_\alpha^{-1}\mathsf{Diag}_h(e^{-i\tau k_\alpha c})\mathcal F_\alpha,\quad 
 \mathcal K_h(K;\tau)=\mathcal F^{-1}\mathsf{Diag}_h(e^{-i\tau\bm k^TK\bm k})\mathcal F.
 \label{eq:elementary_grid_maps}
\end{align}
Here \(\bm k=(k_x,k_y,k_z)^T\). The shear multiplier is evaluated in mixed physical/Fourier coordinates, and the quadratic Fourier multiplier in full Fourier coordinates; \(K\) is a real symmetric \(3\times3\) matrix. These definitions give the implemented stages
\begin{align}
 Q_{1,h}(\tau)&=\mathcal P_h(q_\tau;\tau),&
 Q_{2,h}(\tau)&=\mathcal T_{z,h}(\eta_2x+\eta_3y;\tau),\nonumber\\
 Q_{3,h}(\tau)&=\mathcal T_{y,h}(\eta_1x;\tau),&
 Q_{4,h}(\tau)&=\mathcal K_h(\mathsf K(\tau);\tau),\nonumber\\
 Q_{5,h}(\tau)&=\mathcal T_{y,h}(\xi_3z;\tau),&
 Q_{6,h}(\tau)&=\mathcal T_{x,h}(\xi_1y+\xi_2z;\tau),\nonumber\\
 Q_{7,h}(\tau)&=\mathcal P_h(r_\tau;\tau),&
 E_h(\tau)&=Q_{7,h}(\tau)\cdots Q_{1,h}(\tau).
 \label{eq:seven_discrete_stages}
\end{align}
Thus \(Q_{1,h}\) is the rightmost, first-applied phase and \(Q_{7,h}\) the leftmost, last-applied phase. Equivalently, nodal multiplication by a phase is followed by trigonometric interpolation when the result is viewed as a Fourier polynomial. Every stage is unitary on the nodal space, being a unit-modulus multiplier or its Fourier conjugate. Its inverse is obtained by conjugating the phase \emph{with the stored coefficients held fixed}, generally different from recomputing those coefficients at \(-\tau\).

Canceling adjacent transforms gives the actual mixed-space sequence
\[
 (x,y,z)\to(x,y,k_z)\to(x,k_y,k_z)\to(k_x,k_y,k_z)
 \to(k_x,k_y,z)\to(k_x,y,z)\to(x,y,z).
\]
This mixed-space evaluation uses three forward and three inverse one-dimensional FFT passes. The scalar code uses this schedule; the fused axial spinor version uses its specialization (Supplementary Section~S6). Reversing this sequence and conjugating its stored multipliers applies \(E_h(\tau)^{-1}\); no dense matrix inverse is formed.

For later use, write $\bm\Omega=(a,b,c)^T$. The zero-step exterior phases in~\eqref{eq:continuous_eei_formula} are
\begin{equation}\label{eq:q0_explicit}
 q_0=\nu_xx^2+\nu_yy^2+\nu_zz^2,\quad r_0=-q_0,\qquad
 (\nu_x,\nu_y,\nu_z)=\tfrac14\bigl(-(b^2+c^2),c^2-a^2,a^2+b^2\bigr).
\end{equation}
The equality follows from the coefficient limits in Appendix~\ref{app:coefficients}; it also defines the $\nu_\alpha$ used in the nonvanishing theorem.

\begin{mylemma}[Consistency of the specified realization]
\label{lem:eei_consistency}
The map \eqref{eq:seven_discrete_stages} is analytic near zero and unitary for real nonsingular steps. It satisfies \(E_h(0)=I_h\) and \(E_h'(0)=A_h\).
\end{mylemma}
\begin{proof}
Let \(\xi=(\xi_1,\xi_2,\xi_3)^T\) and \(\eta=(\eta_1,\eta_2,\eta_3)^T\). The coefficient limits are
\begin{equation}
 \mathsf K(0)=\tfrac12I,\qquad
 \xi(0)=(c,-b,a)^T,\qquad \eta(0)=(-c,b,-a)^T,
 \label{eq:coefficient_zero_shears}
\end{equation}
These limits, together with the phase limits in~\eqref{eq:q0_explicit}, follow from Appendix~\ref{app:coefficients}. The two phase derivatives cancel. From \eqref{eq:coefficient_zero_shears}, the first-order shear coefficients are \(\eta_1=-c\), \(\eta_2=b\), \(\eta_3=-a\) on the first two shears and \(\xi_3=a\), \(\xi_1=c\), \(\xi_2=-b\) on the return shears. Substituting these signs into \eqref{eq:seven_discrete_stages} gives exactly \(A_{R,h}\). The Fourier stage contributes \(i\Delta_h/2\), so differentiating the product at zero gives \(A_h\). Analyticity follows from the coefficient formulas, and unitarity from \eqref{eq:elementary_grid_maps}.
\end{proof}

\subsection{The quadratic coefficient of the seven-factor map}
\label{subsec:direct_D2}

Let \(D_{2,h}\) be defined by the local expansion
\begin{equation}
 \log E_h(\tau)=\tau A_h+\tau^2D_{2,h}+O_h(\tau^3).
 \label{eq:EEI_log_D2}
\end{equation}
The matrix logarithm is the local branch through zero. Expanding its exponential gives
\begin{equation}\label{eq:eei-residual-prediction}
E_h(-\tau)E_h(\tau)-I_h=2\tau^2D_{2,h}+O_h(\tau^3).
\end{equation}
Thus a nonzero \(D_{2,h}\) is a second-order obstruction to time symmetry, not merely a failure of an exact group identity. It avoids both a numerical logarithm and subtracting nearly equal states.

\begin{myproposition}[Direct quadratic-defect identities]
\label{prop:direct_D2}
For the implemented map,
\begin{align}
 D_{2,h}&=\tfrac12\bigl(E_h''(0)-A_h^2\bigr),\label{eq:direct_D2}\\
 D_{2,h}&=\left.\tfrac14\frac{d^2}{d\tau^2}
 \bigl(E_h(-\tau)E_h(\tau)-I_h\bigr)\right|_{\tau=0},\label{eq:direct_D2_reversal}\\
 E_h''(0)&=\textstyle \sum_{j=1}^7 Q_{j,h}''(0)
       +2\sum_{r>s}Q_{r,h}'(0)Q_{s,h}'(0).
 \label{eq:stage_second_derivative}
\end{align}
\end{myproposition}
\begin{proof}
Exponentiating \eqref{eq:EEI_log_D2} gives
\(E_h(\tau)=I_h+\tau A_h+\tau^2(A_h^2/2+D_{2,h})+O(\tau^3)\), proving \eqref{eq:direct_D2}. Equation~\eqref{eq:direct_D2_reversal} follows from \eqref{eq:eei-residual-prediction}. Twice differentiating the ordered product in \eqref{eq:seven_discrete_stages}, whose factors equal the identity at zero, gives \eqref{eq:stage_second_derivative}.
\end{proof}

The first identity separates consistency, which fixes \(E_h'(0)\), from the second-order cancellation required by symmetry. The stage formula also provides a direct implementation of \(D_{2,h}v_h\), using the coefficient derivatives in Appendix~\ref{app:coefficients}.

\subsection{Nonvanishing proved with a constant state}
\label{subsec:certificate}

Let \(1_h\) be the array whose entries are all one. Every shear and Fourier multiplier fixes it, and \(A_h1_h=0\). The first quadratic phase does not fix it. This simple input therefore isolates the effect of the two exterior phases and the intervening linear stages. Its role is diagnostic; it is not proposed as a trapped condensate profile.

Write \(N_{\rm g}=N_xN_yN_z\), and let \(\Lambda_h=\prod_{\alpha=x,y,z}\{-N_\alpha/2,\ldots,N_\alpha/2-1\}\). For a scalar array \(v_h\), define the normalized nodal DFT by
\[
 \widehat v_h^{\,d}(\mathbf m)=\frac1{N_{\rm g}}\sum_{\mathbf j}v_{\mathbf j}
 e^{-2\pi i\sum_\alpha m_\alpha j_\alpha/N_\alpha},
 \qquad \mathbf m\in\Lambda_h.
\]
The superscript \(d\) distinguishes the nodal transform from continuous Fourier coefficients. Set \(v_R=\operatorname{Re}(D_{2,h}1_h)\), taking the real part entrywise \emph{before} the DFT. Let \(e_\alpha\) be the unit integer index in direction \(\alpha\). The first nonzero modes used below are non-Nyquist modes.

%\Needspace{13\baselineskip}
\Needspace{16\baselineskip}
\begin{mytheorem}[Nonvanishing for every nonzero rotation]
\label{thm:constant_certificate}
For the grid \eqref{eq:grid_convention}, the signed derivatives \eqref{eq:derivative_matrices}, and the EEI realization \eqref{eq:seven_discrete_stages},
\begin{equation}
 \widehat v_R^{\,d}(e_\alpha)
 =-\frac{\nu_\alpha\pi^2}
 {N_\alpha^2\sin^2(\pi/N_\alpha)},\qquad \alpha=x,y,z,
 \label{eq:constant_fourier_certificate}
\end{equation}
where the coefficients \(\nu_\alpha\) are defined in \eqref{eq:q0_explicit}. In particular, \(D_{2,h}\ne0\) for every \(\bm\Omega\ne0\), and
\begin{equation}
 \frac{\|D_{2,h}1_h\|_h}{\|1_h\|_h}
 \ge \frac{|\nu_\alpha|\pi^2}
 {N_\alpha^2\sin^2(\pi/N_\alpha)}
 \quad\hbox{if }\nu_\alpha\ne0.
 \label{eq:constant_D2_lower_bound}
\end{equation}
Since \(N_\alpha\sin(\pi/N_\alpha)\le\pi\), this also gives the grid-independent lower bound
\[
 \frac{\|D_{2,h}1_h\|_h}{\|1_h\|_h}
 \ge \max_{\alpha=x,y,z}|\nu_\alpha|>0
 \qquad(\bm\Omega\ne0).
\]
This is independent of the box half-widths and remains uniform under grid refinement.
\end{mytheorem}
\begin{proof}
Expand the real phase polynomials as
\(q_\tau=q_0+\tau q_1+O(\tau^2)\) and
\(r_\tau=-q_0+\tau r_1+O(\tau^2)\).
Put \(H_h(\tau)=Q_{6,h}(\tau)\cdots Q_{2,h}(\tau)\).
Every factor in \(H_h\) fixes \(1_h\), so \(H_h(\tau)1_h=1_h\).
Since the two exterior phase derivatives cancel, Lemma~\ref{lem:eei_consistency} also gives \(H_h'(0)=A_h\). Applying the first phase, then \(H_h\), then the last phase yields
\[
 E_h(\tau)1_h
 =1_h+\tau^2\{-i(q_1+r_1)-iA_hq_0\}+O(\tau^3).
\]
In this calculation the terms in \(q_0^2\) cancel. Because \(A_h^21_h=0\), Proposition~\ref{prop:direct_D2} gives the complete coefficient
\begin{equation}
 D_{2,h}1_h
 =\tfrac12\Delta_hq_0-i\{A_{R,h}q_0+q_1+r_1\}.
 \label{eq:constant_D2_complete}
\end{equation}

Since \(q_1\) and \(r_1\) are real, taking the entrywise real part of~\eqref{eq:constant_D2_complete} gives
\[
 v_R
 =\tfrac12\Delta_hq_0+\operatorname{Im}(A_{R,h}q_0).
\]
For a real one-dimensional nodal vector, the paired non-Nyquist frequencies contribute a real spectral derivative; only the unpaired signed Nyquist frequency can contribute an imaginary part. Because \(q_0\) is a sum of separate coordinate squares, \(D_{\alpha,h}q_0\) depends only on \(x_\alpha\). Write
\[
 D_{\alpha,h}q_0=d_\alpha^R+i d_\alpha^I,
\]
where \(d_\alpha^R,d_\alpha^I\) are real nodal vectors in the \(\alpha\) coordinate, and write \(\widehat{\,\cdot\,}^{\,d}_\alpha\) for the one-dimensional normalized DFT in that coordinate. The transform of \(d_\alpha^I\) is supported only at the signed Nyquist index \(-N_\alpha/2\). For \(\beta\ne\alpha\), tensor-product factorization of the normalized nodal DFT gives
\begin{equation}
 \widehat{x_\beta d_\alpha^I}^{\,d}(\mathbf m)
 =\widehat{d_\alpha^I}^{\,d}_\alpha(m_\alpha)\,
  \widehat{x_\beta}^{\,d}_\beta(m_\beta)
  \prod_{\gamma\notin\{\alpha,\beta\}}\delta_{m_\gamma,0},
 \label{eq:nyquist_tensor_support}
\end{equation}
where \(\delta_{m,0}\) is the Kronecker delta. Thus the transverse multiplier \(x_\beta\) can create many \(\beta\)-frequencies, but it cannot move the Nyquist support in the derivative direction: every such term still has \(m_\alpha=-N_\alpha/2\). Each target index \(e_x,e_y,e_z\) has \(\alpha\)-component either \(0\) or \(1\), which cannot equal \(-N_\alpha/2\) when \(N_\alpha\ge4\). Hence \(\widehat{\operatorname{Im}(A_{R,h}q_0)}^{\,d}\) vanishes at all three target modes. Applying the normalized DFT to the displayed pointwise-real identity therefore reduces the theorem's three coefficients to those of \(\tfrac12\Delta_hq_0\).

For one coordinate with half-width \(\ell\), direct summation gives
\begin{equation}
 \frac1N\sum_{j=0}^{N-1}(-\ell+jh)^2e^{-2\pi ij/N}
 =\frac{h^2}{2\sin^2(\pi/N)},\qquad h=2\ell/N.
 \label{eq:coordinate_square_DFT}
\end{equation}
To verify it, put \(z=e^{-2\pi i/N}\). The finite geometric sum and its derivatives give
\(\sum jz^j=-N/(1-z)\) and
\(\sum j^2z^j=-N^2/(1-z)-2Nz/(1-z)^2\).
The terms containing \((1-z)^{-1}\) cancel because \(Nh=2\ell\), leaving
\(-2h^2z/(1-z)^2=h^2/(2\sin^2(\pi/N))\).
At index \(e_\alpha\), the Laplacian multiplier is \(-\pi^2/\ell_\alpha^2\). Substituting \eqref{eq:coordinate_square_DFT} into the DFT of \(v_R\) proves \eqref{eq:constant_fourier_certificate}.

If all three coefficients in \eqref{eq:q0_explicit} vanished, then \(b^2+c^2=a^2+b^2=0\), forcing \(\bm\Omega=0\). At least one Fourier coefficient is therefore nonzero when \(\bm\Omega\ne0\). With the normalized DFT used above, Parseval reads
\[
 \|v_h\|_h^2=|\mathcal D|\sum_{\mathbf m\in\Lambda_h}|\widehat v_h^{\,d}(\mathbf m)|^2,
 \qquad \|1_h\|_h^2=|\mathcal D|.
\]
Because pointwise real projection cannot increase the grid norm,
\[
 \frac{\|D_{2,h}1_h\|_h}{\|1_h\|_h}
 \ge
 \frac{\|v_R\|_h}{\|1_h\|_h}
 \ge
 \left|\widehat v_R^{\,d}(e_\alpha)\right|.
\]
Together with \eqref{eq:constant_fourier_certificate}, this proves \eqref{eq:constant_D2_lower_bound}. Since \(\sin x\le x\) for \(x\ge0\), its right-hand side is at least \(|\nu_\alpha|\); choosing an index that maximizes \(|\nu_\alpha|\) gives the grid-independent bound.
\end{proof}

The identity \(\nu_x+\nu_y+\nu_z=0\) shows that \(q_0\) is harmonic as an ordinary polynomial, although the Laplacian of its periodic trigonometric interpolant is not zero. The coefficient above measures this difference. It is not a pure phase error: \(\operatorname{Re}(i\gamma1_h)=0\) for real \(\gamma\), whereas \(v_R\) has a nonzero Fourier mode. The constant state proves nonvanishing for every allowed grid and nonzero rotation; it does not provide a lower bound for every localized state.

\subsection{The distinct five-operator EEI: a nonconstant-mode certificate}
\label{subsec:eei5-extension}

The newer five-operator EEI of Liu, Wei and Zhang~\cite[Eqs.~(3.15)--(3.19)]{LiuWeiZhang2026} is not a reduction of the preceding seven-factor product: its four shear stages and central Fourier multiplier have different coefficients and no exterior quadratic phases. We therefore cannot apply Theorem~\ref{thm:constant_certificate} to it. In fact, all five stages fix the constant grid state exactly, so its quadratic defect annihilates $1_h$. The appropriate question is whether the complete grid map has a nonzero quadratic coefficient on some \emph{nonconstant} state.

Write $\bm\Omega=(a,b,c)^T$ and denote the coefficients in the cited equation by $\lambda_1,\ldots,\lambda_4$, $\mu_1,\ldots,\mu_4$, and $M_5(\tau)$. They depend on the signed step and rotation vector; their explicit rational--trigonometric formulas are those in~\cite[Eqs.~(3.16)--(3.19)]{LiuWeiZhang2026}. With the elementary grid maps of~\eqref{eq:elementary_grid_maps}, the \emph{application-order} stages are
\begin{align}
 P_{1,h}(\tau)&=\mathcal T_{x,h}(\mu_3y+\mu_4z;\tau),&
 P_{2,h}(\tau)&=\mathcal T_{z,h}(\mu_1x+\mu_2y;\tau),\nonumber\\
 P_{3,h}(\tau)&=\mathcal T_{y,h}(\lambda_3x+\lambda_4z;\tau),&
 P_{4,h}(\tau)&=\mathcal K_h(M_5(\tau);\tau),\nonumber\\
 P_{5,h}(\tau)&=\mathcal T_{x,h}(\lambda_1y+\lambda_2z;\tau),&
 E_{5,h}(\tau)&=P_{5,h}(\tau)\cdots P_{1,h}(\tau).
 \label{eq:eei5-grid-stages}
\end{align}
Here the same centered, signed-Nyquist differentiation convention as in~\eqref{eq:derivative_matrices} is used. The branch represented by~\eqref{eq:eei5-grid-stages} has coefficients analytic near $\tau=0$ when $bc\ne0$. Its removable limits satisfy
\begin{equation}\label{eq:eei5-zero-coefficients}
\begin{split}
 (\lambda_1,\lambda_2,\lambda_3,\lambda_4)\big|_0
 &=\bigl((a^2+c^2)/(2c),\ (a^2-b^2)/(2b),\ -c,\ a\bigr),\\
 (\mu_1,\mu_2,\mu_3,\mu_4)\big|_0
 &=\bigl(b,\ -a,\ (c^2-a^2)/(2c),\ -(a^2+b^2)/(2b)\bigr),\\
 M_5(0)&=\tfrac12I_3.
\end{split}
\end{equation}
The shear generators at zero then sum to $A_{R,h}$, while the Fourier stage gives $i\Delta_h/2$. Thus $E_{5,h}(0)=I_h$ and $E_{5,h}'(0)=A_h$. Define
\begin{equation}\label{eq:eei5-log-coefficient}
 \log E_{5,h}(\tau)=\tau A_h+\tau^2 D_{2,h}^{(5)}+O_h(\tau^3).
\end{equation}

\begin{mytheorem}[Quadratic defect of the five-operator grid EEI]
\label{thm:eei5_certificate}
On every centered even rectangular grid~\eqref{eq:grid_convention}, the implementation~\eqref{eq:eei5-grid-stages} satisfies
\[
 D_{2,h}^{(5)}1_h=0,\qquad D_{2,h}^{(5)}\ne0
 \quad\text{whenever }bc\ne0.
\]
In particular, the standard stagewise Fourier realization of the five-operator EEI is not time-symmetric on this nonsingular coefficient branch.
\end{mytheorem}
\begin{proof}
Since every shear and Fourier multiplier fixes a constant, the first assertion follows from $E_{5,h}(\tau)1_h=1_h$ and $A_h1_h=0$. To prove the second assertion we use a Fourier input rather than the constant state. Let $k_\alpha^{(1)}=\pi/\ell_\alpha$ be the first positive wavenumber and $\bar x=-\ell_x/N_x$, $\bar z=-\ell_z/N_z$ the means of the respective coordinate arrays. We keep the normalized DFT and indices $e_x,e_y$ of Subsection~\ref{subsec:certificate}. For a coordinate array $y_h$ depending only on $y$,
\begin{equation}\label{eq:eei5-yhat}
 \widehat y_h^{\,d}(e_y)
 =-\frac{h_y}{1-e^{-2\pi i/N_y}}\ne0,\qquad
 \widehat{D_{y,h}y_h}^{\,d}(e_y)=ik_y^{(1)}\widehat y_h^{\,d}(e_y),\quad
 \widehat{D_{y,h}^2y_h}^{\,d}(e_y)=-(k_y^{(1)})^2\widehat y_h^{\,d}(e_y).
\end{equation}

For each stage write $\log P_{j,h}(\tau)=\tau B_{j,h}^{(0)}+\tau^2 B_{j,h}^{(1)}+O_h(\tau^3)$. With all coefficients in the following generators evaluated at zero, the application-order first derivatives are
\begin{equation}\label{eq:eei5-first-generators}
\begin{aligned}
B_{1,h}^{(0)}&=-(\mu_3Y+\mu_4Z)D_{x,h}, & B_{2,h}^{(0)}&=-(\mu_1X+\mu_2Y)D_{z,h},\\
B_{3,h}^{(0)}&=-(\lambda_3X+\lambda_4Z)D_{y,h}, & B_{4,h}^{(0)}&=i\Delta_h/2,\\
B_{5,h}^{(0)}&=-(\lambda_1Y+\lambda_2Z)D_{x,h}.
\end{aligned}
\end{equation}
The terms $B_{j,h}^{(1)}$ contain derivatives of the step-dependent coefficients; in particular the central Fourier stage contributes $i\sum_{\alpha,\beta}M_{5,\alpha\beta}'(0)D_{\alpha,h}D_{\beta,h}$. The second-order product expansion gives
\begin{equation}\label{eq:eei5-commutator-expansion}
 D_{2,h}^{(5)}=\sum_{j=1}^5 B_{j,h}^{(1)}+
 \tfrac12\sum_{r>s}[B_{r,h}^{(0)},B_{s,h}^{(0)}].
\end{equation}
All coefficient identities below follow by differentiating the formulas of~\cite[Eqs.~(3.16)--(3.19)]{LiuWeiZhang2026} at zero; Appendix~\ref{app:eei5-coefficients} reproduces the necessary coefficient expressions and gives their elementary Taylor expansions.

First suppose $a\ne0$. Use the pure $x$-mode
$u_x(\mathbf j)=\exp(2\pi i j_x/N_x)$, so $D_{x,h}u_x=ik_x^{(1)}u_x$ and $D_{y,h}u_x=D_{z,h}u_x=0$.
Besides~\eqref{eq:eei5-zero-coefficients}, the only coefficient identities needed for the Fourier mode $e_x+e_y$ are
\[
 \mu_3(0)-\lambda_1(0)=-a^2/c,\qquad
 \mu_3'(0)+\lambda_1'(0)=a^3/(2b).
\]
To isolate the target Fourier mode, note that the $y$-dependent first-order action on $u_x$ comes only from the two $x$-shears. An intermediate $z$-derivative of a transverse coordinate has zero $z$-mean, while the terms involving a $y$-derivative of $y_h$ can contribute at $e_x+e_y$. In~\eqref{eq:eei5-commutator-expansion}, the surviving shear--shear terms contribute
$(ik_x^{(1)}/2)(\mu_3-\lambda_1)(\lambda_3X+\lambda_4Z)D_{y,h}y_h$;
the Fourier--shear terms contribute
$(k_x^{(1)}/4)(\mu_3-\lambda_1)D_{y,h}^2y_h$.
The coefficients in these expressions are evaluated at zero. The remaining commutators either annihilate the input or have zero coefficient at $e_x+e_y$; the constant-coefficient Fourier-stage correction cannot create the mixed target mode from $u_x$. Taking the normalized DFT at that index and using~\eqref{eq:eei5-yhat} yields
\begin{align}
 \frac{\widehat{D_{2,h}^{(5)}u_x}^{\,d}(e_x+e_y)}{\widehat y_h^{\,d}(e_y)}
 ={}&-ik_x^{(1)}(\mu_3'+\lambda_1')
 -\tfrac12k_x^{(1)}k_y^{(1)}(\mu_3-\lambda_1)
 (\lambda_3\bar x+\lambda_4\bar z)\nonumber\\
 &-\tfrac14k_x^{(1)}(k_y^{(1)})^2(\mu_3-\lambda_1)\nonumber\\
 ={}&\frac{k_x^{(1)}k_y^{(1)} a^2}{2c}(-c\bar x+a\bar z)
 +\frac{k_x^{(1)}(k_y^{(1)})^2 a^2}{4c}
 -i\frac{k_x^{(1)} a^3}{2b}.
 \label{eq:eei5-nonzero-mode}
\end{align}
The first two terms on the last line are real. The last term has nonzero imaginary part because $ab\ne0$. Since the denominator on the left is nonzero, this proves $D_{2,h}^{(5)}u_x\ne0$.

It remains to consider $a=0$ (while $bc\ne0$). The coefficient values~\eqref{eq:eei5-zero-coefficients} give, in application order,
\[
 B_{1,h}^{(0)}=B_{5,h}^{(0)}=-\tfrac12(cY-bZ)D_{x,h},\quad
 B_{2,h}^{(0)}=-bXD_{z,h},\quad B_{3,h}^{(0)}=cXD_{y,h},\quad B_{4,h}^{(0)}=i\Delta_h/2.
\]
The shear coefficients are even functions of $\tau$ in this case, hence their $B_{j,h}^{(1)}$ vanish. Direct differentiation of the central Fourier multiplier gives
$\sum_jB_{j,h}^{(1)}=-\tfrac i2D_{x,h}(cD_{y,h}-bD_{z,h})$.
The $B_{1,h}^{(0)},B_{5,h}^{(0)}$ commutators cancel in~\eqref{eq:eei5-commutator-expansion}, and $[B_{2,h}^{(0)},B_{3,h}^{(0)}]=0$. Consequently,
\begin{equation}\label{eq:eei5-a0-identity}
 D_{2,h}^{(5)}=\frac i4\bigl([\Delta_h,X]-2D_{x,h}\bigr)
 (cD_{y,h}-bD_{z,h}).
\end{equation}
Take the pure $y$-mode $u_y(\mathbf j)=\exp(2\pi i j_y/N_y)$. Since $D_{x,h}u_y=D_{z,h}u_y=0$ and $[\Delta_h,X]u_y=(D_{x,h}^2x_h)u_y$, the coefficient at $e_x+e_y$ is
\begin{equation}\label{eq:eei5-a0-nonzero}
 \widehat{D_{2,h}^{(5)}u_y}^{\,d}(e_x+e_y)
 =\frac{ck_y^{(1)}(k_x^{(1)})^2}{4}\widehat x_h^{\,d}(e_x)\ne0,
\end{equation}
where $\widehat x_h^{\,d}(e_x)=-h_x/(1-e^{-2\pi i/N_x})$. This completes both cases.
\end{proof}

The cyclic coordinate variants of the five-factor scheme~\cite[Eqs.~(3.26), (3.28)]{LiuWeiZhang2026} inherit the same nonvanishing result. To see why, let $\pi$ be a cyclic permutation of the three coordinate axes, and let $\pi h$ denote the grid with its side lengths and node counts permuted accordingly. Permuting nodal tensor indices defines an invertible map $C_{\pi,h}$ from the original nodal space to the relabeled one. Coordinate multipliers, directional DFTs, and shear stages all transform under this relabeling, so the implemented cyclic variant satisfies
\[
 E_{5,h}^{(\pi)}(\tau;\bm\Omega)
 =C_{\pi,h}^{-1}E_{5,\pi h}(\tau;\pi\bm\Omega)C_{\pi,h}.
\]
Taking local logarithms shows that its quadratic defect is $C_{\pi,h}^{-1}D_{2,\pi h}^{(5)}C_{\pi,h}$. It is therefore nonzero whenever the $b$- and $c$-components of $\pi\bm\Omega$ are nonzero. Given at least two nonzero rotation components, one can choose a cyclic labeling that places both in these positions. Thus \emph{at least one} of the available coordinate branches is covered by Theorem~\ref{thm:eei5_certificate}; this does not assert that every branch is nonsingular. Pure single-axis rotations require a separate definition or analysis of the singular coefficient limits and are not included. The time-step decomposition of~\cite[Section~3.1.3]{LiuWeiZhang2026} does not change this \emph{local} conclusion: sufficiently small signed steps use the analytic five-factor branch.

The coefficient can also be detected without evaluating a matrix logarithm. By the same Taylor calculation as~\eqref{eq:eei-residual-prediction},
\[
 E_{5,h}(-\tau)E_{5,h}(\tau)-I_h=2\tau^2D_{2,h}^{(5)}+O_h(\tau^3).
\]
Equations~\eqref{eq:eei5-nonzero-mode} and~\eqref{eq:eei5-a0-nonzero} identify explicit inputs for which the leading reversal residual is nonzero. The size of an accumulated defect on a given nonlinear trajectory is a separate question.

\subsection{Statewise diagnostics and the source of the defect}\label{sec:mechanism}\label{subsec:representation}
We return to the seven-factor EEI $E_h$ for the statewise diagnostics in this subsection and in Supplementary Section~S2; the residual bounds below have not been proved for $E_{5,h}$. The complete seven-factor grid map may act weakly on localized states. For \(U\ne0\), define
\begin{equation}\label{eq:statewise-defect}
d_h(U)=\frac{\|D_{2,h}U\|_h}{\|U\|_h}.
\end{equation}
For a spinor, \(D_{2,h}\) acts on each component. Equation~\eqref{eq:eei-residual-prediction} predicts a relative reversal residual \(2\tau^2d_h(U)+O_h(\tau^3)\). Its size depends on the state and the probe step; a small residual does not establish symmetry of the map.

To see how the continuous cancellation changes, define stage matrices by
\(\log Q_{j,h}(\tau)=\tau B^{(0)}_{j,h}+\tau^2B^{(1)}_{j,h}+O_h(\tau^3)\). With the commutator \([A,B]=AB-BA\), expansion of the product gives
\begin{equation}\label{eq:EEI_D2_formula}
D_{2,h}=\sum_{j=1}^7B^{(1)}_{j,h}
+\frac12\sum_{r>s}[B^{(0)}_{r,h},B^{(0)}_{s,h}].
\end{equation}
The corresponding continuous expression vanishes because the factorization is an exact flow. This cancellation need not survive interpolation. Recall that \(\mathscr T_h\) is the resolved trigonometric space; let \(\mathcal J_h\) be interpolation onto it. For continuous stage generators \(B_r,B_s\), write \(B_{j,h}=\mathcal J_hB_j|_{\mathscr T_h}\). On functions for which the displayed operations are defined,
\begin{align}\label{eq:representation_identity}
\mathcal J_h[B_r,B_s]v_h-[B_{r,h},B_{s,h}]v_h
={}&\mathcal J_hB_r(I-\mathcal J_h)B_sv_h\\
&-\mathcal J_hB_s(I-\mathcal J_h)B_rv_h,\qquad v_h\in\mathscr T_h.\nonumber
\end{align}
Coordinate-weighted intermediate functions need not remain in \(\mathscr T_h\). Their interpolation changes the commutators, even though every nodal stage preserves norm. The two terms in~\eqref{eq:representation_identity} can cancel, as can different terms in~\eqref{eq:EEI_D2_formula}; intermediate Fourier tails alone are therefore not lower bounds on the complete defect. Collecting these representation errors yields an exact identity for the complete defect, and hence a state-dependent upper bound. The complete residual proof and bounds are in Supplementary Section~S2.

\Needspace{14\baselineskip}
\begin{myproposition}[Complete-defect residual estimate]\label{prop:main-residual-estimate}
Let \(v\in\mathscr T_h\) interpolate \(U\), let \(\mathcal S_h\) denote nodal sampling, and let \(B_j\), \(1\le j\le7\), be the continuous first-order generators, i.e., the derivatives at zero of the seven EEI factors in application order (Supplementary Section~S2). Set \(r_j=(I-\mathcal J_h)B_jv\) and
\[
\epsilon_1=\max_{j,\alpha}\|\mathcal S_h\partial_\alpha r_j\|_h,\quad
\epsilon_2=\max_j\|\mathcal S_h\Delta r_j\|_h,\quad
C_{\Omega,\ell}=(|b|+|c|)\ell_x+(|a|+|c|)\ell_y+(|a|+|b|)\ell_z,
\]
where \(\bm\Omega=(a,b,c)^T\), \(\alpha\in\{x,y,z\}\), and ordinary spatial derivatives are taken before sampling. Then
\begin{equation}\label{eq:main-residual-bound}
D_{2,h}U=\tfrac12\sum_{r>s}\mathcal S_h(B_s r_r-B_r r_s),\qquad
\|D_{2,h}U\|_h\le3C_{\Omega,\ell}\epsilon_1+\tfrac32\epsilon_2.
\end{equation}
The result also holds componentwise for a spinor.
\end{myproposition}

To see the identity, subtract the sampled continuous quadratic BCH coefficient (zero by continuous exactness) from the discrete one~\eqref{eq:EEI_D2_formula}. The individual stage-second coefficients agree at nodes, while each first-generator commutator differs by~\eqref{eq:representation_identity}; their sum gives the displayed residual identity. Counting the first- and second-derivative generators yields the norm bound, as detailed in Supplementary Section~S2. The identity identifies which differentiated interpolation residuals enter the complete map. The inequality is a conservative consequence that discards their cancellations. Neither statement depends only on an unweighted tail of \(U\). The separate Gaussian estimate in Supplementary Section~S2 concerns only the first phase and is not a bound on \(D_{2,h}U\).

For the axial spinor tests, put \(\bm\Omega=(0,0,\omega)\). The two \(z\)-shears become identities, leaving five factors. Set \(q_{\rm ax}=\omega^2(x^2-y^2)/4\), understood as a multiplication operator, and define their first-order generators, in application order, by
\[
(J_1,J_2,J_3,J_4,J_5)=(iq_{\rm ax},\omega XD_{y,h},i\Delta_h/2,-\omega YD_{x,h},-iq_{\rm ax}).
\]
Then the derivatives in Appendix~\ref{app:coefficients} give
\begin{equation}\label{eq:eei-quadratic-coefficient}
D_{2,h}=i\omega^3XY-i\omega D_{x,h}D_{y,h}
+\frac12\sum_{j>k}[J_j,J_k].
\end{equation}
This formula is used for the spinor diagnostics. Here \(q_{\rm ax}=-q_0\): it is named separately to avoid interchanging the first and last phases of the seven-factor formula. On a cubic grid, Theorem~\ref{thm:constant_certificate} gives the positive coefficient \(\omega^2\pi^2/[4N^2\sin^2(\pi/N)]\) at the first \(x\)-mode. Thus the axial result is a specialization of the same defect, not a second independent obstruction.

\section{Compact compositions and adjoint-paired order recovery}\label{sec:fbstrang}

Having proved a quadratic defect for each EEI realization, we determine its effect on compact composition and then construct the adjoint-paired repair. The compact-error analysis in Subsections~\ref{subsec:compact-stages}--\ref{subsec:accumulated} uses the Laplace--rotation generator \(A_h\), componentwise for spinors (\(M=A_j=0\)); the paired repair also covers the general coupled linear maps of Section~\ref{sec:linear}.

\subsection{The compact fourth- and sixth-order stage lists}\label{subsec:compact-stages}

Let \(\Phi_{A,h}(t)=e^{tA_h}\) and form the exact-subflow Strang map
\[
T_{2,h}(t)=\Phi_{A,h}(t/2)\circ N_h(t)\circ\Phi_{A,h}(t/2).
\]
This differs from the full flow \(\varphi_h(t)\) by \(O_h(t^3)\). The symmetric compositions use
\begin{align}
 a^{(2)}&=(1),\qquad a^{(4)}=(\gamma_1,\gamma_0,\gamma_1),
 &\gamma_1&=\frac1{2-2^{1/3}},\quad\gamma_0=1-2\gamma_1,\label{eq:S4-FB-section2}\\
 a^{(6)}&=(\delta_1,\delta_2,\delta_3,\delta_4,\delta_3,\delta_2,\delta_1),
 &\delta_4&=1-2(\delta_1+\delta_2+\delta_3),\label{eq:S6-FB-section2}
\end{align}
with \(r_2=1\), \(r_4=3\), \(r_6=7\) entries and
\[
\delta_1\simeq0.7845136104775541,\quad
\delta_2\simeq0.2355732133593595,\quad
\delta_3\simeq-1.1776799841788665.
\]
These are the classical real fourth- and sixth-order compositions~\cite{Yoshida1990,HairerLubichWanner2006,blanes2008splitting}. Formal order statements use the exact coefficient solution; the displayed decimals are the approximations stored by the code, not exact solutions of the algebraic order conditions. At order four, \(2\gamma_1+\gamma_0=1\) and \(2\gamma_1^3+\gamma_0^3=0\); order six additionally satisfies the fifth-degree and commutator conditions.

In the exact-subflow composition, adjacent linear exponentials combine by their group law. With \(a_j=a_j^{(p)}\) and \(r=r_p\), the resulting linear fractions are
\begin{equation}\label{eq:compact-linear-fractions}
 c_1=\frac{a_1}{2},\qquad
 c_j=\frac{a_{j-1}+a_j}{2}\quad(2\le j\le r),\qquad
 c_{r+1}=\frac{a_r}{2}.
\end{equation}
Replacing each combined exponential by one evaluation of \(E_h\) defines the implemented compact step
\begin{equation}\label{eq:compact-map}
 C_{p,h}(\tau)=E_h(c_{r+1}\tau)\circ N_h(a_r\tau)\circ E_h(c_r\tau)
 \circ\cdots\circ N_h(a_1\tau)\circ E_h(c_1\tau).
\end{equation}
We call this CS\(p\); \(p\) is the nominal order of its exact-subflow formula. Negative fractions are used with their signs. In particular,
\begin{equation}\label{eq:published-CS4}
 c^{(4)}=\left(\frac{\gamma_1}{2},\frac{1-\gamma_1}{2},
 \frac{1-\gamma_1}{2},\frac{\gamma_1}{2}\right),\qquad
 a^{(4)}=(\gamma_1,\gamma_0,\gamma_1).
\end{equation}
These are the published compact fourth-order fractions of~\cite{LiuZhang2025,LiuMengTangZhang2025}. Our CS4 tests use this outer-stage formula with the specified EEI grid realization and, for spin 1, the exact zero-field nonlinear flow. CS6 is the eight-linear-stage extension derived from~\eqref{eq:S6-FB-section2}, not a sixth-order implementation reported in those papers.

Equation~\eqref{eq:compact-map} evaluates each merged linear slot directly. It does not replace it by two successive EEI evaluations, since in general \(E_h(s)E_h(t)\ne E_h(s+t)\). The continuous factorization and the temporal composition coefficients are unchanged; the issue is their fixed-grid realization.

\Needspace{19\baselineskip}
\subsection{The quadratic defect survives compact composition}\label{subsec:bch}
\begin{myproposition}[Quadratic obstruction for real linear-stage fractions]\label{prop:bch}\label{thm:even_defect_composition}
Let $\mathcal E_h(t)$ be a smooth, first-order-consistent linear grid map whose local expansion is
\[
 \log\mathcal E_h(t)=tA_h+t^2\mathcal D_h+O_h(t^3),
 \qquad \mathcal D_h\ne0.
\]
Suppose an exact-subflow splitting of $A_h+\mathcal N_h$ has order at least two, with fixed real linear fractions $c_1,\ldots,c_q$, $\sum_jc_j=1$. Replace its linear exponentials by $\mathcal E_h(c_j\tau)$, keeping the exact nonlinear subflows, and denote the resulting step by $C_h^{\mathcal E}$. If the signed maps are sufficiently smooth near the states considered, then
\begin{equation}\label{eq:composition_local_defect}
 C_h^{\mathcal E}(\tau)U-\varphi_h(\tau)U
 =\tau^2\kappa\mathcal D_h U+O_h(\tau^3),\qquad
 \kappa=\sum_{j=1}^q c_j^2>0.
\end{equation}
For the CS$p$ fractions of~\eqref{eq:compact-linear-fractions},
\begin{equation}\label{eq:kappa-values}
 \kappa_p=\sum_{j=1}^{r_p+1}(c_j^{(p)})^2,\qquad
 \kappa_2=\tfrac12,\quad
 \kappa_4\simeq0.97455368,\quad \kappa_6\simeq1.28125595.
\end{equation}
Taking $(\mathcal E_h,\mathcal D_h)=(E_h,D_{2,h})$ gives the seven-factor obstruction, while $(\mathcal E_h,\mathcal D_h)=(E_{5,h},D_{2,h}^{(5)})$ gives the five-factor obstruction under Theorem~\ref{thm:eei5_certificate}. Thus a real-coefficient compact formula built from either specified realization is not second order \emph{as a map on the entire fixed-grid state space} in the stated parameter regimes.
\end{myproposition}
\begin{proof}
For either signed stage length, the local logarithm gives
\[
 \mathcal E_h(t)=e^{tA_h}+t^2\mathcal D_h+O_h(t^3).
\]
Begin with the exact-subflow baseline, and reinsert the linear-map perturbations one at a time. At the $j$th linear slot the perturbation is $c_j^2\tau^2\mathcal D_h$. Its input is $U+O_h(\tau)$, and the Fr\'echet derivative (on the underlying real state space for the nonlinear stages) of the remaining near-identity maps is $I_h+O_h(\tau)$. Thus this slot contributes $c_j^2\tau^2\mathcal D_hU+O_h(\tau^3)$. The stage count is fixed, so summation proves~\eqref{eq:composition_local_defect}. The baseline is accurate to $O_h(\tau^3)$ by its assumed order, not by any group identity for the approximate linear maps.

For CS, each exact-subflow Strang block satisfies $T_{2,h}(a_j\tau)=\varphi_h(a_j\tau)+O_h(\tau^3)$. Uniform smoothness over finitely many signed stages preserves the required remainder on composition. The \emph{exact full flows}, not the Strang maps, combine to $\varphi_h(\tau\sum_j a_j)=\varphi_h(\tau)$. Higher-order conditions of the baseline do not cancel the quadratic perturbation inserted into its merged linear slots.

Since the real fractions sum to one, at least one is nonzero, and hence $\kappa>0$. Because $\mathcal D_h\ne0$, some input state has $\mathcal D_hU\ne0$. The local error cannot then be $O_h(\tau^3)$ uniformly for all grid states, as claimed.
\end{proof}

A nonsymmetric map can still have second or higher order, but here the nonzero quadratic coefficient rules out second-order consistency as an operator. The positive-square argument uses real fractions; the observable global response may still vanish on particular trajectories.

\subsection{A common accumulated response and a testable amplitude ratio}\label{subsec:accumulated}

We first specialize to the seven-factor CS$p$ map~\eqref{eq:compact-map}, which is the version tested in Section~\ref{sec:experiments}. Its local coefficient predicts more than a possible first-order rate. Along the exact Laplace--rotation/nonlinear trajectory \(U(t)=\varphi_h(t)U^0\), let \(\mathcal V_h(t,s)\) solve
\[
 \partial_t\mathcal V_h(t,s)=\mathfrak f_h'(U(t))\mathcal V_h(t,s),\qquad
 \mathcal V_h(s,s)=I_h,
\]
and define
\begin{equation}\label{eq:accumulated-response}
 \mathcal R_h(T)=\int_0^T\mathcal V_h(T,t)D_{2,h}U(t)\,\dd t.
\end{equation}
Assume that the vector field and the time-step maps are sufficiently smooth on a common neighborhood of the exact trajectory and the numerical stage states for $0\le t\le T$, with the derivatives needed for the following fixed-grid Taylor expansions uniformly bounded there. Then, at $T=n\tau$,
\begin{equation}\label{eq:global-CS-expansion}
 (C_{p,h}(\tau))^nU^0-U(T)=\kappa_p\tau\mathcal R_h(T)+O_h(\tau^2).
\end{equation}
Indeed, the local quadratic error first gives global error \(O_h(\tau)\) by the Lipschitz recurrence. Substituting it into the Taylor error recurrence yields the linearized equation forced by \(\kappa_pD_{2,h}U\). Variation of constants gives~\eqref{eq:accumulated-response}; the local cubic terms, quadratic error terms and Riemann-sum error contribute \(O_h(\tau^2)\) globally.

When \(\mathcal R_h(T)\ne0\), the normalized endpoint difference at steps \(\tau\) and \(\tau/2\) satisfies
\begin{equation}\label{eq:scaled-self-difference}
 \frac{\|U^{\rm CS,p}_\tau(T)-U^{\rm CS,p}_{\tau/2}(T)\|_h}
 {\kappa_p\tau\|U^{\rm CS,p}_{\tau/2}(T)\|_h}
 \longrightarrow \mathcal C_h(T):=\frac{\|\mathcal R_h(T)\|_h}{2\|U(T)\|_h}.
\end{equation}
Here \(U^{\rm CS,p}_\tau(T)=(C_{p,h}(\tau))^nU^0\). In particular, the CS4/CS6 difference ratio tends to \(\kappa_4/\kappa_6\), without fitting an exponent or identifying the exact endpoint. The argument also applies to a compact step that uses $E_{5,h}$: with $D_{2,h}$ replaced by $D_{2,h}^{(5)}$ in~\eqref{eq:accumulated-response}, the global expansion has the same factor $\kappa_p$. The EEI5 statement is conditional on its own accumulated response; the experiments below test the ratio only for EEI7. For either map, a nonzero operator defect need not produce a nonzero response on every trajectory, and higher-order terms can dominate over a finite step range.

We have thus separated a grid-level fact---the quadratic compact defect cannot cancel as an operator---from the trajectory-dependent question of when its accumulated effect becomes visible. The next step addresses the former directly: we restore time symmetry without assuming that the defect is small on the states being computed.

\subsection{Symmetry of the complete splitting step}\label{subsec:method-adjoint}\label{subsec:free-ordering}\label{subsec:fb-block}

We now address the obstruction without altering the internal EEI coefficients. The essential change is to enforce adjoint symmetry on the \emph{whole} linear--nonlinear--linear step, rather than require each approximate linear substep to be symmetric.

\subsubsection{The forward and adjoint-return maps}\label{subsec:time-splitting}

To repair the defect, choose a linear update \(L_{F,h}(\eta)\) that is locally invertible on \(\mathcal X_h\) for the signed stages used below, and assume \(L_{F,h}(0)=I_h\) and \(L_{F,h}(\eta)=I_h+\eta\mathcal L_h+O(\eta^2)\). It may consist of separate rotation, Fourier, SOC, and internal-spin operations or a combined propagator. Its paired return is
\begin{equation}\label{eq:paired-return}
L_{B,h}(\eta):=L_{F,h}(\eta)^\dagger=L_{F,h}(-\eta)^{-1}.
\end{equation}
The dagger denotes the \emph{method adjoint} defined in the introduction, not the matrix conjugate transpose. Both maps approximate advancement by $+\eta$; the minus sign in the adjoint definition is inside the inverse.

Inverting the negative-step expansion gives \(L_{B,h}(\eta)=I_h+\eta\mathcal L_h+O(\eta^2)\), so the paired map is also first-order consistent. For example, if \(L_{F,h}=\Phi_2\circ\Phi_1\), then \(L_{B,h}=\Phi_1^\dagger\circ\Phi_2^\dagger\), with all maps evaluated at \(\eta\). Thus pairing does not require computing a dense inverse: it reverses the implemented factors and replaces each by its method adjoint.

Let \(N_h(\eta)\) be the exact flow of \(\dot U=\mathcal N_h(U)\), so \(N_h(-\eta)=N_h(\eta)^{-1}\). The step \eqref{eq:intro-FB-block} applies, from right to left, a forward linear half-step, a nonlinear full step, and its adjoint-paired linear half-step. The full block, rather than either linear half-step in isolation, is required to be symmetric. The construction uses only consistency and the genuine method adjoint; its concrete EEI and coupled-spin realizations are given in Section~\ref{sec:linear}.

\paragraph{Implementation of a signed block}
For signed length \(\eta\), apply \(L_{F,h}(\eta/2)\), then \(N_h(\eta)\), then \(L_{F,h}(-\eta/2)^{-1}\). Form the return by reversing factor order and taking each factor's method adjoint, not by replaying the forward order with a negative step. Each composition coefficient scales this \emph{entire} block. Merge adjacent stages only with a verified exact group law; generally \(L_{F,h}(s)L_{F,h}(t)\ne L_{F,h}(s+t)\).

Adjunction reverses composition, \((\Phi\circ\Psi)^\dagger=\Psi^\dagger\circ\Phi^\dagger\), and the exact nonlinear flow satisfies \(N_h^\dagger=N_h\). In the FB block~\eqref{eq:intro-FB-block}, the linear half-steps are mutual adjoints rather than individually self-adjoint. Reversing the sequence and taking the method adjoint of each map leaves the complete block unchanged. This is an \emph{adjoint-paired} palindrome about \(N_h\). When the elementary forward factors are themselves self-adjoint, it is an ordinary palindromic factor sequence. A grouped EEI instead requires its actual method adjoint, including the signed coefficient dependence of its internal factors.

\begin{myproposition}[Complete-step adjoint pairing]\label{thm:FB-second-order}
Let \(L_{F,h}(t)=I_h+t\mathcal L_h+O_h(t^2)\) be a smooth locally invertible linear update, and let \(N_h\) be the exact nonlinear flow. Then the FB step~\eqref{eq:intro-FB-block} is consistent, self-adjoint, and at least second order at fixed \(h\).
\end{myproposition}
\begin{proof}
The adjoint rules and \((L_{F,h}^\dagger)^\dagger=L_{F,h}\) preserve the product, whose first-order term is \(\tau\mathfrak f_h(U)\). Write
\[
 S(\tau)U=U+\tau\mathfrak f_h(U)+\tau^2B_h(U)+O_h(\tau^3).
\]
Expanding \(S(-\tau)S(\tau)U=U\) gives \(2B_h=\mathfrak f_h'\mathfrak f_h\), the exact flow's second Taylor coefficient. Hence the local error is \(O_h(\tau^3)\).
\end{proof}

For \(\Theta_h=N_h\circ L_{F,h}\), the classical symmetrization \(\Theta_h^\dagger(\tau/2)\circ\Theta_h(\tau/2)\) gives FB after the adjacent exact nonlinear half-flows merge~\cite{HairerLubichWanner2006,mclachlan2002splitting,BlanesDieleMarangiRagni2010}. Thus the forward linear realization need not be symmetrized in advance.

\paragraph{Comparison with linear-level symmetrization}
An alternative is to repair the linear propagator first and then use it on both sides of the nonlinear update:
\begin{equation}\label{eq:linear-level-sym}
 L_{{\rm sym},h}(t)=L_{F,h}^\dagger(t/2)\circ L_{F,h}(t/2),\qquad
 S_{2,h}^{\rm lin}(\tau)=L_{{\rm sym},h}(\tau/2)\circ N_h(\tau)\circ L_{{\rm sym},h}(\tau/2).
\end{equation}
The first map is consistent and self-adjoint, so the same argument proves second order for the resulting Strang-type step. Expanded in the original forward map and its adjoint, it is
\[
 S_{2,h}^{\rm lin}(\tau)=
 \underbrace{L_{F,h}^\dagger(\tau/4)\circ L_{F,h}(\tau/4)}_{\text{self-adjoint linear half-step}}
 \circ N_h(\tau)\circ
 \underbrace{L_{F,h}^\dagger(\tau/4)\circ L_{F,h}(\tau/4)}_{\text{self-adjoint linear half-step}}.
\]
A \emph{primitive linear call} means one evaluation of \(L_{F,h}\) or \(L_{F,h}^\dagger\), not one FFT. This nested construction uses four such calls and one nonlinear call per isolated block; FB uses two linear calls and the same nonlinear call. This comparison shows why separate half-step symmetrization is unnecessary.

The following auxiliary result characterizes the shortest complete-step construction when only the specified primitive maps are allowed.

\begin{mycorollary}[Minimum calls in an isolated-step black-box class]\label{cor:FB-stage-minimality}
Count each nonzero call to \(L_{F,h}(a\tau)\), \(L_{F,h}^\dagger(a\tau)\), or \(N_h(b\tau)\) once, before factor or interstep fusion. The admissible constructions are \emph{finite pure compositions} of these three types of primitive maps, with real constant coefficients independent of the maps and the state. No additional processors, correction maps, state-dependent coefficients, additive combinations of maps, or model-specific identities are allowed. Require consistency and self-adjointness for \emph{every} smooth locally invertible first-order linear realization \(L_{F,h}\) and exact nonlinear flow \(N_h\). At least three calls are necessary. With exactly three calls, the only choices are
\[
 L_{F,h}^\dagger(\tau/2)\circ N_h(\tau)\circ L_{F,h}(\tau/2),
 \qquad
 L_{F,h}(\tau/2)\circ N_h(\tau)\circ L_{F,h}^\dagger(\tau/2).
\]
\end{mycorollary}
\begin{proof}
Consistency requires both generators. Test the admissible linear subclass \(L_{F,h}(t)=e^{tA}\), \(N_h(t)=e^{tB}\), with arbitrary matrices \(A,B\). A two-call consistent step has quadratic logarithmic coefficient \(\pm[A,B]/2\), so is not symmetric for all such matrices.

For three calls, also allow \(L_{F,h}(t)=e^{tA+t^2Q}\), whose adjoint is \(e^{tA-t^2Q}\), with arbitrary \(Q\). One linear call cannot cancel this independent even defect. There must therefore be two linear calls and one nonlinear call. If the nonlinear call is outermost, the exact-linear subclass again leaves \(\pm[A,B]/2\). It must be central. Let \(\alpha,\gamma\) be the left and right linear fractions, and let \(\varsigma_L,\varsigma_R\in\{1,-1\}\) denote forward or adjoint orientation. Consistency gives \(\alpha+\gamma=1\) and nonlinear fraction 1. The coefficient of \(\tau^2\) in its logarithm is
\[
 \frac{\alpha-\gamma}{2}[A,B]+
 (\varsigma_L\alpha^2+\varsigma_R\gamma^2)Q.
\]
Since \([A,B]\) and \(Q\) are independent, its vanishing requires \(\alpha=\gamma=1/2\) and \(\varsigma_L=-\varsigma_R\). The adjoint identity and Proposition~\ref{thm:FB-second-order} prove sufficiency.
\end{proof}

This lowest-order characterization concerns one isolated block, not all high-order schemes or equal-error work. Exact nonlinear half-flows can, for example, be shared between successive nonlinear--linear--nonlinear steps. Supplementary Section~S4 gives these alternative counts and the precise conditions for legal factor fusion.

\subsection{Higher-order compositions of the repaired blocks}\label{subsec:FB-high-order}

Compose either complete symmetric block with the lists~\eqref{eq:S4-FB-section2}--\eqref{eq:S6-FB-section2}:
\begin{equation}\label{eq:composition-family}
 S_{p,h}^{X}(\tau)=S_{2,h}^{X}(a^{(p)}_{r_p}\tau)\circ\cdots\circ
 S_{2,h}^{X}(a^{(p)}_1\tau),\qquad X\in\{\mathrm{lin},\mathrm{FB}\},\quad p=2,4,6.
\end{equation}
The finite-order composition conditions for a symmetric second-order map give the designed orders~\cite{Yoshida1990,HairerLubichWanner2006,BlanesCasas2016}. In particular, both the forward and the adjoint return of each block are scaled by its signed coefficient. The adjacent approximate linear maps in this construction are not merged using an unavailable group identity.

\begin{myproposition}[Fixed-grid convergence]\label{prop:FB-high-order}
Fix \(h\) and \(T>0\). On an open neighborhood of a closed ball containing the exact trajectory and all numerical stage states, assume that the vector fields, the forward and adjoint linear maps, and the exact nonlinear flow are \(C^8\), with the required time, state, and mixed derivatives through total order eight bounded for \(|t|\le t_{*,h}\). Let \(a_*\) be the largest absolute coefficient in the chosen list. For the exact coefficient solution, either choice \(X\in\{\mathrm{lin},\mathrm{FB}\}\) and \(p=2,4,6\), there are \(\tau_{0,h}\le t_{*,h}/a_*\) and \(C_{h,T}\), possibly different for the two choices, such that
\[
 \|(S_{p,h}^{X}(\tau))^nU^0-U(n\tau)\|_h\le C_{h,T}\tau^p,
 \qquad 0<\tau\le\tau_{0,h},\quad n\tau\le T.
\]
\end{myproposition}
\begin{proof}
The classical symmetric-composition argument applies to finite Taylor expansions of the discrete maps: method self-adjointness removes even powers from the formal local logarithm, and the exact Yoshida coefficient conditions cancel the remaining terms through degree \(p\). Bounded derivatives give a remainder \(C_h\tau^{p+1}\) on the common neighborhood, including all intermediate negative-time stages. The state derivative of the complete step is \(I_h+O_h(\tau)\), so its Lipschitz constant is at most \(1+K_h\tau\) for small positive \(\tau\). The error recurrence is therefore
\(e_{n+1}\le(1+K_h\tau)e_n+C_h\tau^{p+1}\).
Discrete Gronwall proves the claim. Finite Taylor remainders suffice; no convergent infinite modified series or exact embedding into a modified flow is assumed.
\end{proof}

Both routes have the same designed orders, not necessarily the same error constants. Before fusion, their \(r_p\)-block compositions use \(4r_p\) versus \(2r_p\) primitive linear calls and the same \(r_p\) nonlinear calls: \(12\) versus \(6\) linear calls at order four, and \(28\) versus \(14\) at order six. This motivates our use of FB for the repaired computations.

Signed stages must avoid coefficient singularities. The conditions \(\|\bm\Omega\|t_{*,h}<\pi/2\) for the EEI and \(|\Omega_j|t_{*,h}<\pi\) for the rotation shears are convenient sufficient choices. Neither \(\tau_{0,h}\) nor \(C_{h,T}\) is mesh-uniform. Although \(\|\Delta_h\|=O(h^{-2})\), this operator norm alone gives no sharp usable step window: the Laplacian subflow is exact, whereas splitting errors involve commutators acting on the trajectory. Coefficient rounding and accumulated arithmetic error are distinct effects.

\section{Realizing the paired steps for BEC models}\label{sec:linear}

We now construct the linear maps in \eqref{eq:paired-return}. The main CS/FB tests use the seven-factor EEI; the additional coupled-spin tests use rotation-shear-decomposition-acceleration (RSDA), with eight elementary factors when all couplings are active. Combined EEI or other invertible first-order modules may replace them, provided that the return uses the adjoint of the implemented forward map \cite{Bader2018,bernier2021kinetic,LiuMengTangZhang2025,LiuZhang2025}.

The generator in \eqref{eq:spinor-discrete-Hlin} splits as
\begin{equation}\label{eq:linear-generator-split}
\begin{aligned}
\mathcal L_h&=\mathcal L_{\Delta,h}+\sum_j\mathcal L_j^{\rm rot}+\mathcal L_M+\sum_j\mathcal L_j^{\rm soc},\\
\mathcal L_{\Delta,h}U&=(i/2)\Delta_hU, &\mathcal L_j^{\rm rot}U&=i\Omega_jL_{j,h}U,\\
\mathcal L_MU&=-iUM^T, &\mathcal L_j^{\rm soc}U&=-D_{j,h}UA_j^T.
\end{aligned}
\end{equation}
No commutativity among these generators is assumed.

\subsection{Laplacian and rotation maps}\label{subsec:linear-spatial}

The Laplacian map \(L_\Delta(\tau)=\exp(\tau\mathcal L_{\Delta,h})\) is a scalar Fourier multiplier applied to each spin component \cite{bao2003,BaoWang2006,Bader2018}. Single-axis rotations use RSDA \cite{BernierCasasCrouseilles2020,LiuYuanZhang2025,LiuXieYuanZhangZhao2026RSDA4}. For a transverse coordinate \(q\), define the shear \(\mathcal S_j(\alpha;q)U=\mathcal F_j^{-1}(e^{-i\alpha k_jq}\mathcal F_jU)\); it is the elementary map of~\eqref{eq:elementary_grid_maps} with the step included in \(\alpha\). For \(\theta_j=-\Omega_j\tau\), \(a_j=-\tan(\theta_j/2)\), and \(b_j=\sin\theta_j\), set
\begin{equation}\label{eq:rsda-axis-maps}
\begin{aligned}
R_x(\tau)&=\mathcal S_y(a_x;z)\circ\mathcal S_z(b_x;y)\circ\mathcal S_y(a_x;z),\\
R_y(\tau)&=\mathcal S_z(a_y;x)\circ\mathcal S_x(b_y;z)\circ\mathcal S_z(a_y;x),\\
R_z(\tau)&=\mathcal S_x(a_z;y)\circ\mathcal S_y(b_z;x)\circ\mathcal S_x(a_z;y).
\end{aligned}
\end{equation}
Each shear is unitary, with inverse obtained by negating its parameter. Since \(a_j,b_j\) are odd in \(\tau\), the palindromic order in \eqref{eq:rsda-axis-maps} gives \(R_j(-\tau)=R_j(\tau)^{-1}\) on the grid. Expansion gives \(R_j(\tau)=I_h+\tau\mathcal L_j^{\rm rot}+O(\tau^2)\). This verifies consistency and reversibility of the shear product, without equating it to an exact rotation exponential. Avoiding \(\theta_j=(2k+1)\pi\) ensures that the formulas are defined; larger rotations can be subdivided.

\subsection{Internal-spin and directional SOC maps}\label{subsec:linear-soc}

The internal map is \(\Phi_M(\tau)U=Ue^{-i\tau M^T}\). The directional SOC flow \(\Phi_j\), after an FFT in direction \(j\), updates \(\widehat U(k_j)\) to \(\widehat U(k_j)e^{-i\tau k_jA_j^T}\) \cite{ChenYuanZhang2026,LiuYuanZhang2025,LiuXieYuanZhangZhao2026RSDA4}. These exact unitary groups are method self-adjoint. Diagonalize each constant Hermitian matrix once and reuse its eigenvectors; each signed stage then only updates scalar phase multipliers for its one-dimensional wavenumbers. No new matrix diagonalization is required for each mode or step, and the different coupling matrices need not be simultaneously diagonalizable.

\subsection{Full forward and paired maps}\label{subsec:linear-full}

One fully active ordering is
\begin{equation}\label{eq:full-linear-pair}
\begin{aligned}
L_{F,h}&=\Phi_z\circ\Phi_y\circ\Phi_x\circ\Phi_M\circ R_z\circ R_y\circ R_x\circ L_\Delta,\\
L_{B,h}&=L_\Delta\circ R_x\circ R_y\circ R_z\circ\Phi_M\circ\Phi_x\circ\Phi_y\circ\Phi_z,
\end{aligned}
\end{equation}
with every factor evaluated at the same signed step. Inactive factors may be omitted.

\begin{myproposition}[Consistency, adjunction, and unitarity]\label{prop:linear-pair-admissible}
Suppose each active factor or grouped module is locally invertible and first-order consistent with its part of \(\mathcal L_h\). Reversing their order and taking their method adjoints gives a first-order pair related by the method adjoint. If all forward factors are unitary, so are both complete maps.
\end{myproposition}
\begin{proof}
The product and adjoint rules in Section~\ref{subsec:free-ordering} give consistency and the adjoint identity. Products, inverses, and method adjoints of unitary maps are unitary.
\end{proof}

All factors in \eqref{eq:full-linear-pair} are self-adjoint, so reversing their order suffices. For a grouped EEI that is not self-adjoint on the grid, whether seven- or five-factor, the reversed product must use its actual method adjoint, including the adjoints of the operations inside that group. In particular, the construction applies algebraically to the analytic EEI5 branches of Theorem~\ref{thm:eei5_certificate}; the tested FB implementations below use the seven-factor map. The reversal imposes symmetry at complete-step level. For \(m\) self-adjoint modules it uses \(2m\) calls, versus \(4m-2\) for separate linear-half-step palindromes when their central pairs permit exact group-law fusion. Supplementary Section~S4 gives the assumptions and alternative arrangements; no group law is assumed for the RSDA maps.

The ordered operations of each signed FB block, including the adjoint return, are specified in Section~\ref{subsec:time-splitting}. The signed composition coefficients scale the complete block.

\begin{table}[H]
\centering\small
\caption{Calls per isolated macrostep, before cross-stage or cross-step fusion. A linear call evaluates the EEI or its method adjoint. The FFT column uses three directional FFT/inverse-FFT pairs per call, valid for the scalar and fused axial spinor realizations. It excludes nonlinear work and additional SOC transforms. The linear-level rows count~\eqref{eq:linear-level-sym}; they are not separate experiments.}
\label{tab:linear-cost}
\begin{tabular}{@{}lrrr@{}}\toprule
Construction & Linear calls & Directional FFT pairs & Nonlinear calls\\\midrule
Compact CS4 & 4 & 12 & 3\\
Linear-level \(S_{4,h}^{\rm lin}\) & 12 & 36 & 3\\
Complete-step FB-S4 & 6 & 18 & 3\\\midrule
Compact CS6 & 8 & 24 & 7\\
Linear-level \(S_{6,h}^{\rm lin}\) & 28 & 84 & 7\\
Complete-step FB-S6 & 14 & 42 & 7\\\bottomrule
\end{tabular}
\end{table}

Table~\ref{tab:linear-cost} distinguishes two baselines. FB halves the primitive linear calls relative to the specified nested repair, with the same nonlinear-stage count. Relative to compact CS, FB uses 50\% more linear calls at order four and 75\% more at order six. These counts do not compare equal-error costs. Independently constructed self-adjoint shear propagators have their own factor lists and costs.

For the axial spinor EEI, mixed-Fourier evaluation cancels redundant transform pairs without changing the discrete factor product or its method adjoint, giving the three-pair count in the table. The historical spinor convergence data used the algebraically equivalent unfused evaluation, with five pairs per call. Supplementary Section~S6 gives both schedules, their numerical comparisons and the timing records. No complete-EEI group law is used in either evaluation.

A self-adjoint rearrangement of four compact linear calls need not retain fourth order; Supplementary Section~S4 gives a counterexample. It does not invalidate correctly composed repaired blocks or establish a universal four-call impossibility. Section~\ref{subsec:scalar-benchmark} compares observed CS/FB calls at several \emph{self-difference} thresholds.

\subsection{Nonlinear BEC models and spin conventions}\label{subsec:spinor-nonlinearities}

The symmetry argument has so far required only that the nonlinear substep be an exact reversible flow. We now specify the scalar and spinor nonlinearities for which this assumption holds, and then give their updates. These definitions are also those used in the numerical experiments; they are not additional hypotheses in the EEI nonvanishing theorem.

\subsubsection{Spin matrices and local invariants}\label{subsec:spinor-linear}

The matrix $M$ can describe Zeeman, Rabi, or other constant internal couplings~\cite{StamperKurnUeda2013,KawaguchiUeda2012,BaoCai2011,BaoCai2018Spinor}; the matrices $A_j$ can describe Rashba, Dresselhaus, or anisotropic SOC~\cite{LinJimenezGarciaSpielman2011,CampbellJuzeliunasSpielman2011}. For spin 1 and 2, let \(\bm f=(f_x,f_y,f_z)\). The spin matrices in the basis \(m=F,F-1,\ldots,-F\) are determined by
\(f_z=\operatorname{diag}(F,F-1,\ldots,-F),\qquad (f_+)_{m+1,m}=\sqrt{F(F+1)-m(m+1)},\)
for \(-F\le m<F\), with all other entries of \(f_+\) zero.

The subscripts here are magnetic quantum numbers, not row numbers. Then \(f_x=(f_++f_+^*)/2\) and \(f_y=(f_+-f_+^*)/(2i)\). For example, \(M=pf_z+qf_z^2\), \(p,q\in\mathbb R\), gives the spin-1 Zeeman terms; the ladder definition above fixes the explicit matrices without listing their entries. The same ladder definition fixes the five-component spin-2 matrices, so the subsequent nonlinear formulas require no change of basis.

At each point set \(\rho=\Psi^*\Psi\), \(F_j=\Psi^*f_j\Psi\), and \(\bm F=(F_x,F_y,F_z)\), with \(\bm F\cdot\bm f=\sum_jF_jf_j\). The scalar \(F\) labels the spin class; \(\bm F\) is the state-dependent spin vector. For the two-component examples, we use the Pauli matrices
\[
\sigma_x=\begin{pmatrix}0&1\\1&0\end{pmatrix},\quad
\sigma_y=\begin{pmatrix}0&-i\\i&0\end{pmatrix},\quad
\sigma_z=\begin{pmatrix}1&0\\0&-1\end{pmatrix}.
\]

\subsubsection{Potential and nonlinear interactions}\label{subsec:nonlinear-potential}

The interaction coefficients and trapping potential \(V=V(\bm x)\) are real. For the scalar density \(\rho=|\psi|^2\), set
\begin{equation}\label{eq:scalar-potential}
\mathcal G_0(\psi)=W[\rho]\psi,\qquad
W[\rho]=V+\beta\rho+\lambda_{\rm P}\Phi_{\rm P}[\rho],\qquad
\Phi_{\rm P}[\rho]=\partial_{\bm d\bm d}(K_{\rm C}*\rho).
\end{equation}
Here \(K_{\rm C}(\bm x)=1/(4\pi|\bm x|)\) is the free-space Poisson kernel, \(*\) is convolution, \(\bm d\) is a unit dipole direction, and \(\partial_{\bm d\bm d}=(\bm d\cdot\nabla)^2\). The subscript P specifies the Poisson-directional convention. In the often-used convention \(U_{\rm dd}=-\delta-3\partial_{\bm d\bm d}K_{\rm C}\), with \(\delta\) the Dirac distribution,
\begin{equation}\label{eq:dipole_convention_conversion}
\beta\rho+\lambda_{\rm P}\Phi_{\rm P}[\rho]
=(\beta-\lambda_{\rm P}/3)\rho-(\lambda_{\rm P}/3)(U_{\rm dd}*\rho).
\end{equation}
Supplementary Section~S8 defines the discrete nonlocal operator used in the experiment, based on the far-field smooth approximation (FSA) method~\cite{LiuZhangFSA2026}. In contrast, the following spinor interactions act pointwise. For the density-coupled spin-\(1/2\) model,
\(\mathcal G_{1/2,\ell}(\Psi)=\left(V+\sum_{m=1}^2\beta_{\ell m}|\psi_m|^2\right)\psi_\ell,\qquad \ell=1,2,\)
where the two components are relabeled \(1,2\). For spin 1 and 2,
\begin{equation}\label{eq:spin2-nonlinear}
\begin{aligned}
\mathcal G_1(\Psi)&=[(V+c_0\rho)I_3+c_1\bm F\cdot\bm f]\Psi,\\
\mathcal G_2(\Psi)&=[(V+c_0\rho)I_5+c_1\bm F\cdot\bm f]\Psi+c_2A_{00}\mathcal A\overline\Psi.
\end{aligned}
\end{equation}
In the spin-2 basis \(m,n=2,1,0,-1,-2\), the pairing matrix and amplitude are
\begin{equation}\label{eq:spin2-A-matrix}
\mathcal A_{mn}=\frac{(-1)^m}{\sqrt5}\delta_{m,-n},\qquad
A_{00}=\Psi^T\mathcal A\Psi=\frac{2\psi_2\psi_{-2}-2\psi_1\psi_{-1}+\psi_0^2}{\sqrt5}.
\end{equation}
Subsection~\ref{sec:nonlinear} gives the exact updates \cite{SymesMcLachlanBlakie2016,SymesBlakie2017}.

The coefficients \(c_0,c_1,c_2\) multiply density, spin-exchange, and singlet-pairing interactions, respectively. The last term is present only for spin 2. In particular, \(A_{00}\) is computed from the spinor; it is distinct from the constant SOC matrices \(A_j\). Including \(V\) in \(\mathcal G\) keeps the spinor subproblem local in physical space. For the scalar subproblem, the entire density array, and hence its nonlocal potential, remains fixed during the exact nonlinear flow.

\subsection{Exact nonlinear flows}\label{sec:nonlinear}

We must verify the second ingredient of the repair: the nonlinear map really is an exact, reversible subflow, including for signed composition stages. The following formulas do so for each model just defined.

For the scalar density functional in~\eqref{eq:scalar-potential}, the grid subproblem has the exact solution
\begin{equation}\label{eq:scalar-exact-flow}
N_{0,h}(t)u=\exp\{-itW_h[|u|^2]\}\,u,
\qquad W_h[\rho]=V+\beta\rho+\lambda_{\rm P}\Phi_{{\rm P},h}[\rho].
\end{equation}
The multiplication is pointwise, but the potential need not be local: every density entry is invariant, so the entire nonlocal potential is unchanged throughout the substep. Recomputing it after the update gives the same array. Consequently \(N_{0,h}(-t)=N_{0,h}(t)^{-1}\), and the discrete mass is preserved for every real \(t\). This verifies the hypothesis used in Section~\ref{sec:fbstrang} for the scalar model without replacing the nonlocal problem by a local one.

For the spinor models, each node of \(\dot U=\mathcal N_h(U)\) is an independent ODE. The following established exact updates \cite{BaoCai2015SOC,SymesMcLachlanBlakie2016,SymesBlakie2017,BaoCai2018Spinor} are defined for all real steps, including the negative stages in the high-order compositions. Coefficients are recomputed from each substep input and remain constant by local invariants, not throughout the full PDE evolution.

\subsubsection{Spin-\texorpdfstring{\(1/2\)}{1/2} and spin-1 flows}\label{subsec:nonlinear-spinhalf}

Density invariance gives the spin-\(1/2\) pointwise update \(N_{{\rm dc},h}\):
\[
\psi_\ell(\tau)=\exp\!\left[-i\tau\left(V+\sum_{m=1}^2\beta_{\ell m}|\psi_m(0)|^2\right)\right]\psi_\ell(0),\qquad \ell=1,2.
\]
For spin 1, the local equation is \(i\partial_t\Psi=[(V+c_0\rho)I_3+c_1\bm F\cdot\bm f]\Psi\). Hermiticity gives \(\partial_t\rho=0\), and the spin commutation relations yield
\(\partial_tF_j=-ic_1\Psi^*[f_j,\bm F\cdot\bm f]\Psi =c_1(\bm F\times\bm F)_j=0.\)
Consequently \(a=V+c_0\rho\) and the Hermitian spin matrix \(H_{\rm sp}=c_1\bm F\cdot\bm f\) remain constant throughout this local step. The exact update is \cite{SymesMcLachlanBlakie2016}
\begin{equation}\label{eq:spin1-local-update}
N_{1,h}(\tau)\Psi=e^{-i\tau a}e^{-i\tau H_{\rm sp}}\Psi.
\end{equation}

\begin{myproposition}[Reversibility of the spin-\texorpdfstring{\(1/2\)}{1/2} and spin-1 updates]\label{prop:spinhalf-spin1-local-reversible}
For every real \(\tau\), \(N_{{\rm dc},h}\) and \(N_{1,h}\) are exact pointwise flows, preserve density, and satisfy \(N_h(-\tau)=N_h(\tau)^{-1}\).
\end{myproposition}
\begin{proof}
For spin \(1/2\), each component undergoes a phase rotation with invariant density. For spin 1, invariance of \(\rho,\bm F\) fixes \(a,H_{\rm sp}\), including when recomputed at the endpoint, so
\(e^{i\tau a}e^{i\tau H_{\rm sp}}e^{-i\tau a}e^{-i\tau H_{\rm sp}}=I_3.\)
Both nonlinear flows preserve pointwise density and are method self-adjoint; their state-dependent matrix representations do not make them linear unitary operators.
\end{proof}

\subsubsection{Spin-2 singlet-pairing flow}\label{subsec:nonlinear-spin2}

The spin-2 local equation is
\begin{equation}\label{eq:spin2-local-ode}
i\partial_t\Psi=[(V+c_0\rho)I_5+c_1\bm F\cdot\bm f]\Psi+c_2A_{00}\mathcal A\overline\Psi.
\end{equation}
The exact nonlinear dynamics and their use in symplectic splitting are developed in \cite{SymesBlakie2017}; see also \cite{Wang2011Spin2,BaoCai2018Spinor}. We write the flow in the normalization of \eqref{eq:spin2-A-matrix}. The identities \(\Psi^T\mathcal A f_j\Psi=0\) and \((\bm F\cdot\overline{\bm f})\mathcal A+\mathcal A(\bm F\cdot\bm f)=0\) imply invariance of \(\rho,\bm F\). In contrast to \(\rho\) and \(\bm F\), \(A_{00}\) is generally not constant along the full local flow. Differentiating \(\Psi^T\mathcal A\Psi\) gives
\(\partial_tA_{00}=-2i\bigl[V+(c_0+c_2/5)\rho\bigr]A_{00}.\)
The spin term vanishes by \(\Psi^T\mathcal A f_j\Psi=0\). Removing the scalar phase and spin rotation gives a pairing subsystem with constant singlet amplitude: a linear system of two five-entry spinors.

From the input state, compute \(\rho,\bm F,A_{00}\), set \(\alpha=\rho/5\), and take \(\chi\ge0\) with \(\chi^2=\alpha^2-|A_{00}|^2/5\). Cauchy--Schwarz and \(\mathcal A^*\mathcal A=I_5/5\) give \(|A_{00}|^2\le\rho^2/5\), so \(\chi\) is real. With \(\operatorname{sinc}(z)=\sin(z)/z\) and \(\operatorname{sinc}(0)=1\), define
\begin{equation}\label{eq:spin2-b-update}
b=\cos(c_2\chi\tau)\Psi+ic_2\tau\operatorname{sinc}(c_2\chi\tau)\left(\frac\rho5\Psi-A_{00}\mathcal A\overline\Psi\right).
\end{equation}
The exact update is
\begin{equation}\label{eq:spin2-local-update}
N_{2,h}(\tau)\Psi=e^{-i\tau(V+(c_0+c_2/5)\rho)}e^{-i\tau c_1\bm F\cdot\bm f}b.
\end{equation}

The \(\operatorname{sinc}\) form covers \(\chi=0\). Exact reversibility needs an additional check: the coefficients on a negative-time step are recomputed from its input, not held fixed from the original state.

\begin{myproposition}[Reversibility of the spin-2 update]\label{prop:spin2-local-reversible}
For every real \(\tau\), \eqref{eq:spin2-b-update}--\eqref{eq:spin2-local-update} define an exact pointwise flow satisfying \(N_{2,h}(-\tau)=N_{2,h}(\tau)^{-1}\) and preserving density.
\end{myproposition}

The proof separates the scalar phase, spin rotation, and gauged singlet-pairing flow. In the pairing subsystem, a two-by-two generator satisfies $K^2=\chi^2I_2$; equivariance under phase and spin rotations then ensures that the negative-time step remains the inverse when its coefficients are recomputed from the endpoint. The details, including the endpoint coefficient update, are given in Supplementary Section~S9.

These exact flows and Proposition~\ref{prop:linear-pair-admissible} give:

\begin{mycorollary}[Exact discrete mass preservation]\label{cor:mass-preservation}
For a unitary linear update and its method adjoint, and any of the nonlinear flows above, \(\Mh(S_{p,h}^{\rm FB}(\tau)U)=\Mh(U)\), \(p=2,4,6\), in exact arithmetic, including negative stages.
\end{mycorollary}
\begin{proof}
Every linear factor preserves the discrete norm, and every nonlinear update preserves pointwise density. Their compositions therefore preserve \(\Mh\).
\end{proof}

\section{Numerical experiments}\label{sec:experiments}

We compare compact CS4/CS6~\eqref{eq:compact-map} with paired FB-S4/S6~\eqref{eq:composition-family}, using the same \emph{seven-factor EEI} and exact nonlinear subflows. The fixed-grid experiments test the stage-square prediction and recovery of order. The distinct EEI5 result is analytic; it is not part of these dynamical tests. Self-differences assess temporal refinement, not spatial accuracy or equal-accuracy cost.

\subsection{Error measures and numerical protocol}\label{subsec:numerical-setup}

Write \(U^{X,p}_\tau(T)\) for an endpoint with method \(X\in\{\mathrm{CS},\mathrm{FB}\}\), nominal order \(p\in\{4,6\}\), and macrostep \(\tau\). We report the raw self-difference and its observed order,
\begin{equation}\label{eq:self-convergence}
 e^{X,p}_h(\tau)=\frac{\|U^{X,p}_\tau(T)-U^{X,p}_{\tau/2}(T)\|_h}
 {\|U^{X,p}_{\tau/2}(T)\|_h},\qquad
 p_{\rm obs}(\tau)=\log_2\frac{e^{X,p}_h(\tau)}{e^{X,p}_h(\tau/2)}.
\end{equation}
A rate needs three endpoints. These quantities need no reference solution but are not exact semidiscrete errors. Supplementary Section~S1 defines phase-aligned and scalar density differences and reports representative checks. Against a reference \(U^{\rm ref}\), use
\begin{equation}\label{eq:reference-error}
 E_{\rm rel}(\tau)=\min_{\theta\in\mathbb R}
 \frac{\|U_\tau(T)-e^{i\theta}U^{\rm ref}(T)\|_h}{\|U^{\rm ref}(T)\|_h}.
\end{equation}
The unaligned norm is the one in Proposition~\ref{prop:FB-high-order}. All initial states have unit discrete mass.

CS and FB use identical stable EEI and spin-subflow routines (Supplementary Section~S3). A reporting threshold $10^{-12}$ flags spinor self-differences too small to support resolved rates; selected values appear in Supplementary Section~S1; all flagged values are retained in the run records. Formal order statements concern exact coefficients and arithmetic, not the roundoff floor.

The scalar runs renormalize mass after each complete macrostep (never inside signed stages); the largest recorded FB-S6 correction is $2.78\times10^{-15}$. This alters the floating-point protocol, although it is the identity in exact arithmetic on a unit-mass state. The spinor runs do not rescale mass. No unrescaled scalar control is claimed.

\subsection{The compact comparison on an arbitrary-axis scalar problem}\label{subsec:scalar-benchmark}

On \([-10,10)^3\), use \(160^3\) points, \(\bm\Omega=(-0.7,1,-\sqrt3)\), \((\beta,\lambda_{\rm P})=(2000,1500)\), dipole direction \(\bm d=(1/2,\sqrt3/2,0)\), and \(T=0.025\). Set
\begin{align}
 V(\bm x)&=-0.1(x^2+y^2)+0.075(x^2+y^2)^2+z^2,\label{eq:scalar-trap}\\
 u^0(\bm x)&=C_{\rm sc}e^{-|\bm x|^2/2}
 \{1+\sin(\pi z)+i\cos[\pi(x+y+z)]\},\label{eq:scalar-initial}
\end{align}
where \(C_{\rm sc}>0\) normalizes the sampled mass. The coefficient of \(z^2\) is one. The fixed nonlocal discretization in Supplementary Section~S8 uses doubled padding and smoothing length \(\varepsilon_{\rm FSA}=2\); the conversion to the dipolar-kernel convention is~\eqref{eq:dipole_convention_conversion}. The prescribed excited state and quartic trap are used to compare time integrators, not to determine vortex equilibria; related stationary studies include~\cite{AftalionDu2001,AftalionDanaila2004}.

For each of CS4, CS6, FB-S4 and FB-S6, compute six endpoints at \(\tau=1/n\), \(n=200,400,800,1600,3200,6400\). The twelve compact-CS endpoints and twelve FB endpoints have identical recorded model and grid configurations. Their source records are distinguished in Supplementary Section~S7.

\begin{figure}[tbp]
\centering
\includegraphics[width=0.83\linewidth]{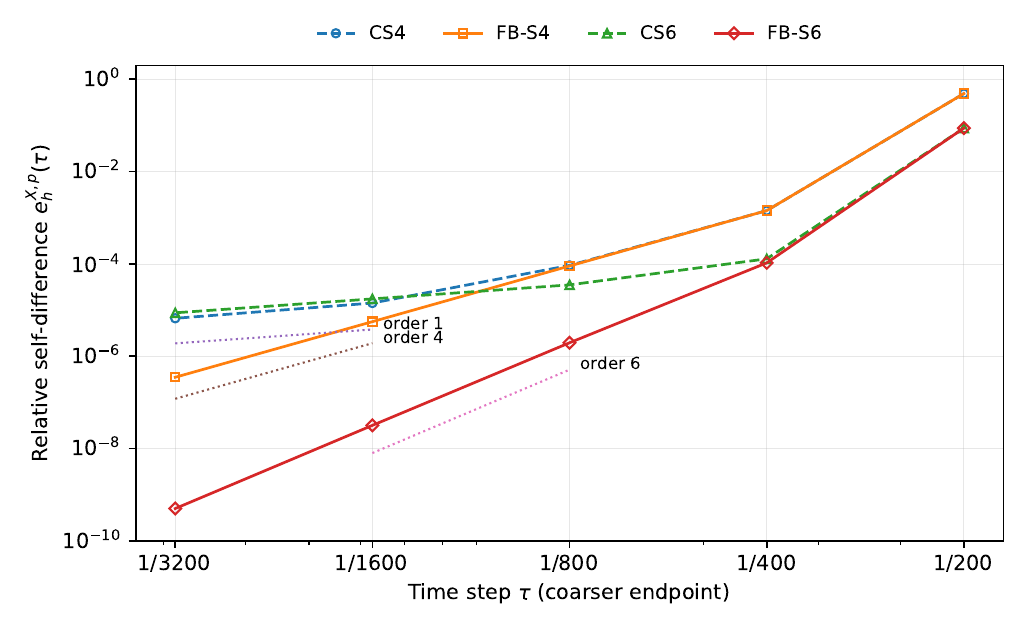}
\caption{Scalar \(160^3\) fixed-grid self-convergence of compact CS4/CS6 and paired FB-S4/S6. All five adjacent-step differences are shown, including the coarse non-asymptotic interval. Short guide segments indicate first-, fourth- and sixth-order decay. Supplementary Section~S1 gives selected numerical records and metric definitions.}
\label{fig:scalar-CS-FB}
\end{figure}

Figure~\ref{fig:scalar-CS-FB} shows a change of dominant error, not a uniform low-order slope across all steps. The raw CS4 rates are \(8.4225,3.9423,2.7087,1.0981\), while CS6 gives \(9.4098,1.8727,1.0089,1.0023\). The very large first rates are outside the asymptotic range, not superconvergence. The scalar CS4 transition occurs at raw self-difference levels around \(10^{-5}\)--\(10^{-6}\); its finest CS4/CS6 differences are \(6.6422\times10^{-6}\) and \(8.7370\times10^{-6}\). These are differences between time refinements, not certified errors relative to the physical solution. FB retains its designed behavior, with finest rates \(4.0016\) and \(5.9867\) and an FB-S6 difference of \(5.0240\times10^{-10}\).

\Needspace{5\baselineskip}
Phase alignment leaves the CS4/CS6 finest rates nearly unchanged, at \(1.0957\) and \(1.0023\). The scalar density is a different, nonlinear observable: its CS4 rate is still \(1.8888\), whereas its CS6 rate is \(1.0040\). Thus density also resolves the CS6 contribution, but does not yet show a first-order CS4 interval. We do not infer an observable's rate from the full-state rate alone.

The observed-work comparison is also threshold-dependent: compact CS uses fewer EEI calls at a raw self-difference target of $10^{-4}$, while FB needs fewer at $10^{-5}$; neither compact method reaches $10^{-6}$ within the recorded steps. Supplementary Section~S1 gives the complete call-count table. These are self-difference thresholds, not equal-accuracy work--precision comparisons.

\paragraph{Endpoint defect and Fourier content.}
The saved FB-S6 endpoint has $d_h=7.822731$ and outer-band norm fractions $\eta_{0.5}=0.299116$, $\eta_{0.8}=0.101428$. Here $\eta_\theta$ is the fraction of the discrete Fourier norm at modes with $\max_j |k_j|h_j/\pi\ge\theta$ (Supplementary Section~S1). The initial defect cannot be resolved from the available precision checks, whereas independent assemblies and reversal probes support the terminal value. The complete residual estimate of Proposition~\ref{prop:main-residual-estimate} is conservative by a factor of about 392 at that endpoint. Postprocessing taper and window checks in Supplementary Section~S1 show that a Fourier-tail fraction alone does not determine the defect. These are diagnostics of one retained FB endpoint, not a spatial-accuracy test.

\subsection{A fourth-to-first-order crossover on the spinor grid}\label{subsec:spinor-compact}\label{subsec:eei-structure}

Use spin 1 on \([-16,16)^3\) with \(128^3\) nodes, so \(h=1/4\), \(\bm\Omega=(0,0,0.2)\), \(c_0=10\), \(c_1=1\), \(M=A_j=0\), and \(V=|\bm x|^2/2\). Set \(g_{\rm G}=e^{-|\bm x|^2/2}\) and
\begin{equation}\label{eq:spinor-perturbed-data}
 \Psi^0=C_{\rm pert}g_{\rm G}
 \begin{pmatrix}
 1+\tfrac14(x+iy)+\varepsilon e^{i\theta}\\
 \tfrac34+\tfrac15z+\tfrac{3}{20}ix+\varepsilon e^{-i\theta}\\
 \tfrac35-\tfrac{3}{20}x+\tfrac15iy+\varepsilon e^{i\widetilde\theta}
 \end{pmatrix},\quad
 \begin{aligned}\theta&=3x+4y+2z,\\\widetilde\theta&=-3x+4y-2z,\end{aligned}
\end{equation}
with \(\varepsilon=0.1\) and mass-normalizing \(C_{\rm pert}>0\). This zero-field specialization uses the exact spin-1 flow, not the magnetic-field Magnus approximation of~\cite{LiuMengTangZhang2025}. Evolve to \(T=0.02\) with \(\tau=1/n\), \(n=100,200,400,800,1600,3200\), corresponding to \(2,4,8,16,32,64\) macrosteps. All four methods use the same configuration without mass rescaling.

\begin{figure}[tbp]
\centering
\includegraphics[width=0.83\linewidth]{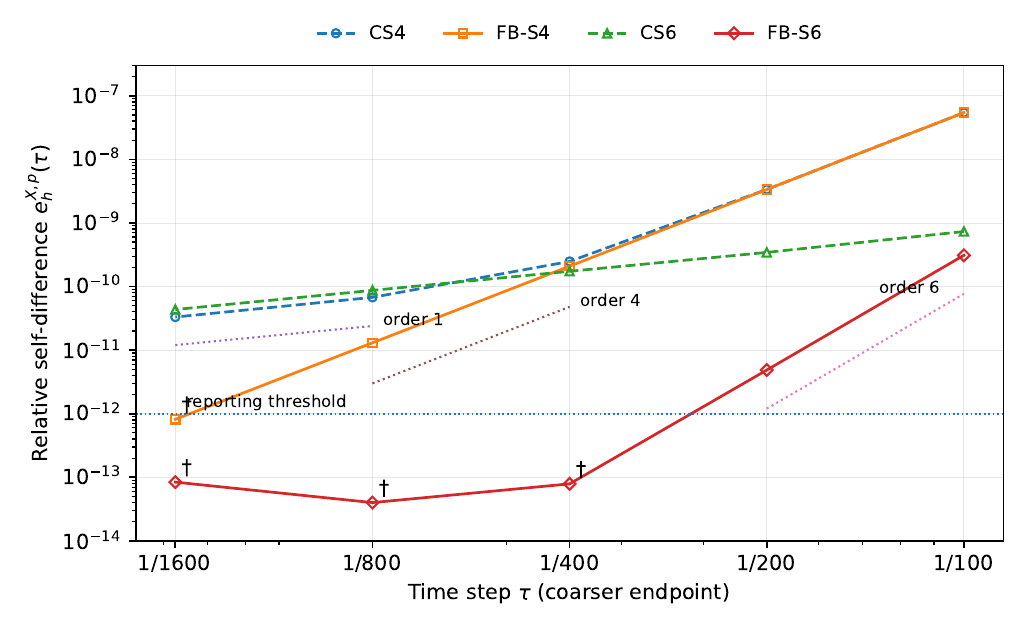}
\caption{Compact CS versus paired FB on the fixed \(128^3\) spin-1 grid. CS4 initially tracks fourth-order decay before the structural contribution becomes visible; CS6 exposes it earlier. Daggers mark self-differences below the \(10^{-12}\) reporting threshold, not additional resolved FB order data. Guide segments have slopes one, four and six.}
\label{fig:spinor-CS-FB}
\end{figure}

\begin{table}[tbp]
\centering\small
\caption{Spinor three-level orders from steps \(1/n,1/(2n),1/(4n)\). A dash marks a rate involving a raw self-difference below the reporting threshold. Supplementary Section~S1 records representative values; all raw values are retained in the run records.}
\label{tab:spinor-orders}
\begin{tabular}{@{}rcccc@{}}\toprule
\(n\)&CS4&FB-S4&CS6&FB-S6\\\midrule
100&4.0122&4.0173&1.0900&5.9949\\
200&3.7622&4.0042&0.9871&--\\
400&1.8753&4.0010&0.9968&--\\
800&1.0282&--&0.9993&--\\\bottomrule
\end{tabular}
\end{table}

The first CS4 test alone would suggest fourth order. Subsequent refinement gives the crossover in Table~\ref{tab:spinor-orders}: \(4.0122\), \(3.7622\), \(1.8753\), then \(1.0282\). The last two differences, \(6.7573\times10^{-11}\) and \(3.3134\times10^{-11}\), nearly halve. Thus the spinor crossover is resolved in the temporal self-differences at order \(10^{-10}\)--\(10^{-11}\). CS6 approaches first order earlier, with its last two differences \(8.7002\times10^{-11}\) and \(4.3523\times10^{-11}\). Phase-aligned rates give the same conclusion. This is a direct test of the compact stage assembly, not of an alternative sequence retaining extra linear slots.

FB-S4 remains fourth order over the resolved intervals. For FB-S6, the differences \(3.0985\times10^{-10}\) and \(4.8587\times10^{-12}\) give rate \(5.9949\). Finer differences reach \(7.89\times10^{-14}\), \(4.00\times10^{-14}\), and \(8.40\times10^{-14}\); their nonmonotone floor is not evidence of a lower formal order. The stable EEI and spin-1 formulas were already used in these runs. Supplementary Section~S7 records the subsequent implementation checks; they do not certify production differences at the roundoff floor.

\subsection{The predicted defect-amplitude ratio in both models}\label{subsec:scalar-scaling}

Equation~\eqref{eq:scaled-self-difference} predicts the same normalized amplitude for CS4 and CS6 within a fixed model and grid. Table~\ref{tab:scalar-scaling} tests this without fitting an exponent. The two models need not have the same common response \(\mathcal C_h(T)\).

\begin{table}[tbp]
\centering\small
\caption{Common-response test: normalized raw self-differences for compact CS4 and CS6. The finest scalar and spinor CS4/CS6 ratios differ from \(\kappa_4/\kappa_6=0.760623735\) by \(0.0497\%\) and \(0.0891\%\), respectively. Earlier rows are retained to show entry into the scaling regime.}
\label{tab:scalar-scaling}
\begin{tabular}{@{}lrrr@{}}\toprule
Model & \(n\) & \(ne_h^{\rm CS,4}(1/n)/\kappa_4\) & \(ne_h^{\rm CS,6}(1/n)/\kappa_6\)\\\midrule
Scalar & 800 & \(7.6306\times10^{-2}\) & \(2.1991\times10^{-2}\)\\
Scalar & 1600 & \(2.3345\times10^{-2}\) & \(2.1855\times10^{-2}\)\\
Scalar & 3200 & \(2.1810\times10^{-2}\) & \(2.1821\times10^{-2}\)\\
Spinor & 400 & \(1.0175\times10^{-7}\) & \(5.4204\times10^{-8}\)\\
Spinor & 800 & \(5.5470\times10^{-8}\) & \(5.4323\times10^{-8}\)\\
Spinor & 1600 & \(5.4398\times10^{-8}\) & \(5.4350\times10^{-8}\)\\
\bottomrule\end{tabular}
\end{table}

At the last scalar pair, the measured ratio is \(0.760245414\); at the last spinor pair it is \(0.761301084\). Agreement of these relative amplitudes with the same stage-square law complements the slopes in Figures~\ref{fig:scalar-CS-FB} and~\ref{fig:spinor-CS-FB}. The agreement supports the stage-square coefficient of the leading fixed-grid defect, beyond the observed reduction in slope.

The crossover also clarifies the relation to earlier successful fourth-order tests. The scalar accuracy test in~\cite[Example~3, Table~5]{LiuZhang2025} uses \((x+iy)e^{-|\bm x|^2/2}/\pi^{3/4}\); the spinor tests in~\cite[Section~3.1]{LiuMengTangZhang2025} likewise use polynomial--Gaussian envelopes. A small accumulated response can remain masked over their tested step range. We have not recomputed those original experiments or identified their individual error components. Thus fourth-order behavior in a finite window is compatible with a nonzero grid-level obstruction, as the spinor crossover shows.

\subsection{Why an initial reversal test is insufficient}\label{subsec:eei-evolving-defect}
The operator defect depends on the evolving state. For the initial spin-1 state at $h=1/4$, the raw reversal residual at probe step $\delta=0.005$ is $1.10\times10^{-15}$; at the stored FB-S6 endpoint the defect is $d_h=1.09\times10^{-5}$, and the measured reversal residual $5.4471\times10^{-10}$ agrees with the prediction $2\delta^2d_h=5.4391\times10^{-10}$. Thus testing only the initial state can miss the structural contribution that develops along the trajectory.

The spin-1 interaction mixes carrier frequencies: the two prescribed carriers generate, among others, $(3,12,2)$, whose $y$-component lies at $95.5\%$ of the Fourier cutoff for $h=1/4$. This is consistent with the recorded rise in the statewise defect, but is not a unique explanation of its magnitude. Supplementary Section~S5 provides grid-refinement and trajectory diagnostics; the grid-independent constant-state certificate remains valid irrespective of these statewise values.

\subsection{Model coverage and independent checks}\label{subsec:temporal-convergence}\label{subsec:generic-matrix-SOC}

The compact CS comparisons isolate the EEI with $M=A_j=0$. Additional archived FB experiments use~\eqref{eq:full-linear-pair} for spin \(1/2\), 1 and 2, with rotation, SOC and internal couplings; the spin-2 model includes singlet pairing. A separate spin-\(1/2\) test has pairwise noncommuting Hermitian matrices. Their resolved refinements approach the designed orders. Supplementary Section~S5 provides parameter specifications and representative convergence data, including points at which the observed rates are limited by numerical resolution. They establish coverage of the paired construction, not a nonvanishing theorem for arbitrary products of additional couplings.

An independent unsplit DOP853 check uses \(6^3\) nodes on \([-3,3)^3\), all three spin classes, nonzero rotation/SOC/internal coupling, and \(c_2=0.9\) for spin 2. At \(T=0.08\), FB steps \(0.08,0.04,0.02\) are compared with two independently refined ODE tolerances without mass rescaling. These independent ODE integrations check the additional RSDA-based spinor implementation, not the large-grid EEI compact-composition experiments. The supplement summarizes these checks, identifies the compact MATLAB data sources and discusses the endpoint-difference roundoff floor.

\section{Conclusion}\label{sec:conclusion}

A continuous exact factorization and its Fourier-grid realization are different maps. For the seven-factor Laplace--rotation EEI, that difference produces a nonzero quadratic coefficient in the local logarithm on every centered even grid whenever the rotation is nonzero. A constant-state, low-frequency certificate proves the obstruction directly. We also prove a separate nonconstant Fourier-mode certificate for the five-operator EEI of~\cite{LiuWeiZhang2026} on nonsingular coefficient branches. The two proofs use different inputs, but the unified composition proposition applies to both maps once their quadratic defects are established. Both are operator-level results, independent of a fitted convergence rate.

Once the implemented EEI is substituted into a real-coefficient \emph{merged} compact splitting, the quadratic defect survives with the positive factor $\kappa_p=\sum_jc_j^2$. Its accumulated response explains both the possible late crossover to first-order temporal behavior and a quantitative CS4/CS6 error-amplitude ratio. The scalar dipolar and spin-1 experiments test these predictions for the seven-factor EEI, while the five-operator result is established by its independent certificate. High-order behavior over a finite step window is compatible with a nonzero fixed-grid defect, since the crossover depends on the trajectory and step range.

The practical repair is to pair a forward linear half-step and its genuine method adjoint around the exact nonlinear flow. Symmetry then holds for the complete block, and classical compositions recover fourth and sixth order at fixed grid without separately symmetrizing each linear half-step. Within the stated primitive-map class, the complete-step repair avoids the extra calls of separately symmetrized linear half-steps. It nevertheless uses more linear calls than compact CS, so its computational advantage depends on the accuracy sought.

The observed crossovers occur at raw self-difference levels around $10^{-5}$--$10^{-6}$ (scalar) and $10^{-10}$--$10^{-11}$ (spinor); their practical significance depends on the application and spatial resolution. The general lesson is to check symmetry \emph{after discretization and factor assembly}, before invoking continuous high-order composition theory.

\appendix

\section{The EEI coefficients and their values at zero}
\label{app:coefficients}

The formulas in this appendix are the coefficients of Theorem~1, Eqs.~(2.4)--(2.7), of Liu--Zhang~\cite{LiuZhang2025}, rewritten in the notation of Section~\ref{subsec:eei_formula}. Write \(\operatorname{sym}(B)=(B+B^T)/2\). Denoting the matrices displayed in their Eq.~(2.7) by \(K_{\rm LZ}\) and \(P_{\rm LZ}\), we store the equivalent symmetric quadratic-form representatives \(\mathsf K=\operatorname{sym}(K_{\rm LZ})\) and \(\mathsf P=\operatorname{sym}(P_{\rm LZ})\). This leaves the Fourier and physical quadratic exponents unchanged and matches the representation used by the code. Together with the action order in Eq.~\eqref{eq:seven_discrete_stages}, these formulas give the stagewise Fourier realization corresponding to their Eq.~(2.12).

Write \(\bm\Omega=(a,b,c)^T\), \(\omega=|\bm\Omega|\), and \(\vartheta=\omega\tau\). Define
\[
 \mu_\vartheta=\frac{1-\cos\vartheta}{\vartheta^2},\quad s_\vartheta=\frac{\sin\vartheta}{\vartheta},\quad
 \kappa_x=\cos\vartheta+a^2\mu_\vartheta\tau^2,\quad \kappa_z=\cos\vartheta+c^2\mu_\vartheta\tau^2,
\]
and \(g(u,v,w)=uv\mu_\vartheta\tau+w s_\vartheta\) for real scalars \(u,v,w\).
At \(\vartheta=0\), use \(\mu_\vartheta=1/2\) and \(s_\vartheta=1\). In a neighborhood of zero where \(\kappa_x\kappa_z\ne0\), the scalar coefficients are given below. A simple sufficient condition for this branch is \(|\bm\Omega|\,|\tau|<\pi/2\): then \(\cos\vartheta>0\), while \(\kappa_x\) and \(\kappa_z\) equal \(\cos\vartheta\) plus nonnegative terms. This condition is convenient rather than maximal; the implementation also checks the actual denominators for EEI substeps.
\begin{align}
 \zeta_1&=\frac{b^2+c^2}{2\kappa_x}\mu_\vartheta,&
 \zeta_2&=\frac{a^2-c^2}{2\kappa_z}\mu_\vartheta,&
 \zeta_3&=-\frac{a^2+b^2}{2}\mu_\vartheta,\nonumber\\
 \xi_1&=g(-a,b,c)\frac{\kappa_x}{\kappa_z},&
 \xi_2&=-g(a,c,b)\kappa_z,&
 \xi_3&=g(c,b,a)\frac{\kappa_z}{\kappa_x},\nonumber\\
 \eta_1&=-\frac{g(a,b,c)}{\kappa_x},&
 \eta_2&=\frac{g(a,c,b)}{\kappa_z},&
 \eta_3&=\frac{g(c,b,-a)}{\kappa_z}.
 \label{eq:actual_eei_scalar_coefficients}
\end{align}
These are the explicit coefficients used in the arbitrary-angle construction of~\cite{LiuZhang2025}.

Let
\[
 A_{\rm rot}=\begin{pmatrix}0&c&-b\\-c&0&a\\b&-a&0\end{pmatrix},
 \qquad W_{\rm rot}(\tau)=e^{\tau A_{\rm rot}}.
\]
The real \(3\times3\) matrix \(A_{\rm rot}\) is distinct from the grid generator \(A_h\); \(W_{\rm rot}\) is distinct from the density-dependent potential \(W_h\). For a three-vector \(v=(v_1,v_2,v_3)^T\), set
\[
 \Theta^{(1)}_{\tau,v}=\begin{pmatrix}
 1&\tau v_1&\tau^2v_1v_3+\tau v_2\\0&1&\tau v_3\\0&0&1
 \end{pmatrix},\qquad
 \Theta^{(2)}_{\tau,v}=\begin{pmatrix}
 1&0&0\\\tau v_1&1&0\\\tau v_2&\tau v_3&1
 \end{pmatrix}.
\]
With \(\xi=(\xi_1,\xi_2,\xi_3)^T\), \(\eta=(\eta_1,\eta_2,\eta_3)^T\), \(\zeta=(\zeta_1,\zeta_2,\zeta_3)^T\), the symmetric quadratic-form matrices are
\begin{align}
 \mathsf K(\tau)&=\operatorname{sym}\left[
 \tfrac12(\Theta^{(2)}_{-\tau,\xi})^T W_{\rm rot}(\tau)(\Theta^{(2)}_{\tau,\eta})^T\right],\nonumber\\
 \mathsf P(\tau)&=\operatorname{sym}\left[
 \frac{W_{\rm rot}(\tau)^T\Theta^{(1)}_{\tau,\xi}\Theta^{(2)}_{\tau,\eta}-I}{2\tau^2}\right].
 \label{eq:actual_eei_matrix_coefficients}
\end{align}
Every \(\xi,\eta\) in these expressions is evaluated at the given signed step \(\tau\). For \(\bm\Omega=0\), \(\mathsf P=0\), \(\mathsf K=I/2\), and the map reduces to the Laplacian multiplier.

Taylor expansion gives the limits in \eqref{eq:q0_explicit} and \eqref{eq:coefficient_zero_shears}. For a direct second-order Taylor calculation of the stages, the remaining derivatives are
\begin{align*}
 \xi'(0)&=\tfrac12(-ab,-ac,bc)^T,\;
 \eta'(0)=\tfrac12(-ab,ac,bc)^T,\; \zeta'(0)=0,\;
 \mathsf K'(0)=\tfrac12\begin{pmatrix}0&-c&b\\-c&0&-a\\b&-a&0\end{pmatrix}.
\end{align*}
The diagonal entries of \(\mathsf P'(0)\) are zero, and its upper off-diagonal entries are
\begin{align*}
 \mathsf P'_{12}(0)=\frac{c(a^2-b^2-2c^2)}{4},\;
 \mathsf P'_{13}(0)=\frac{b(a^2+2b^2+c^2)}{4},\;
 \mathsf P'_{23}(0)=\frac{a(-2a^2-b^2+c^2)}{4}.
\end{align*}
Symmetry supplies the lower entries. In particular, \(r_1=0\) and \(q_1=\bm x^T\mathsf P'(0)\bm x\) in the proof of Theorem~\ref{thm:constant_certificate}.

\section{Coefficient formulas and zero-step expansion for the five-operator EEI}\label{app:eei5-coefficients}

Here we record the identities from Liu--Wei--Zhang~\cite[Eqs.~(3.16)--(3.19)]{LiuWeiZhang2026} needed in the independent five-operator proof. They also specify exactly which signed analytic branch is being analyzed. Write $\bm\Omega=(a,b,c)^T$, $\omega^2=a^2+b^2+c^2$ and, for $\omega>0$,
\[
 m(t)=\frac{1-\cos(\omega t)}{(\omega t)^2},\qquad
 s(t)=\frac{\sin(\omega t)}{\omega t},\qquad
 g(u,v,w;t)=uv\,m(t)t+w\,s(t),
\]
\[
 g_1(u,v,w;t)=-uv(\omega^2+w^2)m(t)t+w(u^2-v^2)s(t).
\]
With the time argument suppressed in $g,g_1,m$, the four coefficients of each shear family are
\begin{align}
 \lambda_1&=\frac{(a^2+c^2)m}{g(a,b,c)},&
 \lambda_2&=\frac{g_1(a,b,c)m}{g(a,b,c)g(a,c,b)},&
 \lambda_3&=-g(a,b,c),&
 \lambda_4&=\frac{2abm}{g(a,c,b)},\label{eq:eei5-lambda-formulas}\\
 \mu_1&=g(a,c,b),&
 \mu_2&=-\frac{2acm}{g(a,b,c)},&
 \mu_3&=-\frac{g_1(a,c,b)m}{g(a,b,c)g(a,c,b)},&
 \mu_4&=-\frac{(a^2+b^2)m}{g(a,c,b)}.\label{eq:eei5-mu-formulas}
\end{align}
The central Fourier quadratic form is represented by the symmetric part of
\[
 \frac12(I-tH_1(t))e^{tA_{\rm rot}}
 (I+tH_4(t)^T)(I+tH_3(t)^T)(I+tH_2(t)^T),
\]
where
\[
 A_{\rm rot}=\begin{pmatrix}0&c&-b\\-c&0&a\\b&-a&0\end{pmatrix},\quad
 H_1=\begin{pmatrix}0&\lambda_1&\lambda_2\\0&0&0\\0&0&0\end{pmatrix},\quad
 H_2=\begin{pmatrix}0&0&0\\\lambda_3&0&\lambda_4\\0&0&0\end{pmatrix},
\]
\[
 H_3=\begin{pmatrix}0&0&0\\0&0&0\\\mu_1&\mu_2&0\end{pmatrix},\qquad
 H_4=\begin{pmatrix}0&\mu_3&\mu_4\\0&0&0\\0&0&0\end{pmatrix}.
\]
We take the symmetric part only because $\bm k^TM_5\bm k$ depends on that part; the implemented Fourier multiplier is unchanged. Near zero,
\[
 m(t)=\tfrac12+O(t^2),\quad s(t)=1+O(t^2),\quad
 g(u,v,w;t)=w+\tfrac12uv\,t+O(t^2),
\]
\[
 g_1(u,v,w;t)=w(u^2-v^2)-\tfrac12uv(\omega^2+w^2)t+O(t^2).
\]
For $bc\ne0$ these expansions give the limits~\eqref{eq:eei5-zero-coefficients}. In particular,
\[
 \lambda_1'(0)=-\frac{ab(a^2+c^2)}{4c^2},\qquad
 \mu_3'(0)=\frac{a(a^2b^2+2a^2c^2+b^2c^2)}{4bc^2},
 \qquad \lambda_1'(0)+\mu_3'(0)=\frac{a^3}{2b}.
\]
Differentiating the central quadratic form gives
\begin{equation}\label{eq:eei5-M-jet}
 M_5'(0)=\tfrac12\operatorname{sym}\bigl(A_{\rm rot}-H_1(0)+H_4(0)^T+H_3(0)^T+H_2(0)^T\bigr).
\end{equation}
When $a=0$,~\eqref{eq:eei5-M-jet} reduces to
\[
 M_5'(0)=-\tfrac14c(e_xe_y^T+e_ye_x^T)
 +\tfrac14b(e_xe_z^T+e_ze_x^T),
\]
so the central Fourier logarithm contributes
$-\tfrac i2 D_{x,h}(cD_{y,h}-bD_{z,h})$ at quadratic order, as used in~\eqref{eq:eei5-a0-identity}. The $e_\alpha$ in this appendix are the standard Cartesian unit vectors (equivalently the unit multi-indices used in the grid DFT).

\section*{Supplementary material}
The supplement contains selected diagnostic records, representation-error proofs,
stable coefficients, FFT fusion and cost details, additional spinor/ODE checks,
and the technical FSA and spin-2 reversibility derivations. The main nonvanishing and composition proofs, the isolated-step minimum-call proof, the fixed-grid convergence argument and the direct compact comparisons are in the article.

\section*{Data and code availability}
The reproducibility archive at \url{https://doi.org/10.5281/zenodo.23138002} provides scalar and spinor implementations, configurations, source CSV/MAT reports, rebuild scripts and verification records. Supplementary Section~S7 identifies the datasets supporting each result. Large diagnostic terminal arrays are retained separately by the authors, with hashes and regeneration instructions in the archive.

\section*{Funding}
This work was supported in part by the U.S. National Science Foundation [DMS-2608905] (F. Xue). The authors were responsible for the research design, mathematical analysis, numerical experiments, manuscript preparation, and decision to submit the work for publication.

\section*{CRediT authorship contribution statement}
\noindent Tianqi Zhang: Software, Investigation, Validation.\par
\noindent Fei Xue: Methodology, Formal analysis, Funding acquisition, Validation, Writing -- original draft, Writing -- review \& editing.

\section*{Declaration of competing interests}
Declarations of interest: none.

\section*{Declaration of generative AI and AI-assisted technologies in the manuscript preparation process}
During the preparation of this work, the authors used OpenAI ChatGPT to assist with exploratory mathematical derivations, consistency checking, numerical verification and code development, and manuscript organization and language refinement. Anthropic Claude was used for critical review of mathematical arguments, assumptions, and manuscript presentation. The authors reviewed and checked AI-assisted material, made the final scientific and editorial decisions, and take full responsibility for the content of the article.

\begingroup
\small
\setlength{\bibsep}{0pt}
\bibliographystyle{plainnat}
\bibliography{BEC_refs}
\endgroup
\end{document}

\begin{center}
{\Large\bfseries Supplementary material}\par\medskip
{\large \paperTitle}\par\medskip
Tianqi Zhang and Fei Xue
\end{center}
\noindent This supplement provides selected self-convergence and statewise EEI diagnostics, the complete interpolation-residual proof, stable coefficient formulas, and additional spinor verification. The fourth-order CS formula is the compact composition used in the cited scalar and spin-1 studies; the sixth-order CS formula is its extension tested in this work. The linear-level repair illustrates the distinction between half-step and complete-step symmetrization; the numerical comparison is compact CS versus FB. Section~S6 documents the algebraically equivalent FFT implementation and its separate timing tests. Sections~S8 and S9 supply the benchmark nonlocal discretization and the complete spin-2 reversibility proof.

Throughout this supplement, ``EEI'' and $E_h$ refer to the seven-factor map, including its axial specialization. The distinct shear-only EEI5 is proved analytically in the main article and is not part of these experimental records.

Notation agrees with the main article: \(h\) denotes the fixed grid, \(\|\cdot\|_h\) its weighted norm, \(E_h(t)\) the specified EEI grid map, \(D_{2,h}\) its quadratic logarithmic coefficient, and \(N_h(t)\) the exact nonlinear flow. The method adjoint is \(\Psi_h^\dagger(t)=\Psi_h(-t)^{-1}\). CS denotes the compact formula with merged linear-stage fractions \(c_j\), and \(\kappa_p=\sum_jc_j^2\). FB denotes the composition of complete adjoint-paired blocks.

\section{Compact CS/FB diagnostics and scalar records}\label{sec:scalar-complete}
\subsection{Self-convergence conventions}
For each method and nominal order \(p=4,6\), write \(u=U_\tau(T)\) and \(v=U_{\tau/2}(T)\). We report
\begin{align}
e_h(\tau)&=\frac{\|u-v\|_h}{\|v\|_h},\label{eq:supp-self}\\
e_h^{\rm ph}(\tau)&=\min_{\theta\in\mathbb R}\frac{\|u-e^{i\theta}v\|_h}{\|v\|_h},\qquad
e_h^{\rho}(\tau)=\frac{\||u|^2-|v|^2\|_h}{\||v|^2\|_h}.\label{eq:supp-phase-density}
\end{align}
The density metric is used for the scalar problem. With the inner product linear in its first argument, the optimal phase is \(\arg\langle u,v\rangle_h\) if this inner product is nonzero; otherwise every phase gives the same norm. Each rate is \(\log_2[e(\tau)/e(\tau/2)]\). A dash means that the third endpoint was not computed. A dagger retains a reported rate for which at least one required raw difference is below the driver's heuristic threshold \(10^{-12}\); it is not interpreted as a resolved asymptotic order. This threshold is not a certified bound on roundoff.

\subsection{Compact CS and paired FB benchmarks}
The scalar test uses $[-10,10)^3$, $160^3$ nodes and $T=0.025$.
Its parameters are $\bm\Omega=(-0.7,1,-\sqrt3)$,
$(\beta,\lambda_{\rm P})=(2000,1500)$ and
$\bm d=(1/2,\sqrt3/2,0)$. The step denominators are
$n=200,400,800,1600,3200,6400$. CS and FB share the initial state,
FSA grid functional and stable EEI coefficients; mass is rescaled after
complete scalar macrosteps. Their configuration records agree.

The matched spin-1 test uses $[-16,16)^3$, $128^3$ nodes, $T=0.02$,
$\bm\Omega=(0,0,0.2)$, $(c_0,c_1)=(10,1)$ and $M=A_j=0$.
It uses the carrier state of the main article, steps $1/n$ with
$n=100,200,400,800,1600,3200$, and no mass rescaling.
CS4 uses the published compact outer formula in its exact zero-field
nonlinear specialization; CS6 is the extension studied here.

Table~\ref{tab:compact-endpoint-summary} retains the finest raw
self-difference in each run and a representative resolved three-level
order; the main article plots all five adjacent-step differences and
reports the CS crossover sequences. The phase-aligned orders agree
closely with raw orders. For the scalar benchmark the finest CS4/CS6
density orders are $1.8888/1.0040$, respectively; a density observable
need not have the same leading response as the complex wavefunction.

\begin{table}[htbp]
\centering\small
\caption{Selected fixed-grid self-convergence results. The column $n_{\rm last}$
gives the coarser denominator of the finest raw difference. Rates are
computed from three endpoints; $n_p$ identifies the coarser denominator
used for the listed rate. The spinor FB-S6 terminal values are
roundoff-limited, so its rate is taken earlier.}
\label{tab:compact-endpoint-summary}
\begin{tabular}{@{}llrrrr@{}}\toprule
Model & Method & $n_{\rm last}$ & $e_h(1/n_{\rm last})$ & $n_p$ & $p_{\rm obs}$\\\midrule
Scalar & CS4 & 3200 & $6.6422\times10^{-6}$ & 1600 & 1.0981\\
 & CS6 & 3200 & $8.7370\times10^{-6}$ & 1600 & 1.0023\\
 & FB-S4 & 3200 & $3.5080\times10^{-7}$ & 1600 & 4.0016\\
 & FB-S6 & 3200 & $5.0242\times10^{-10}$ & 1600 & 5.9867\\\midrule
Spin 1 & CS4 & 1600 & $3.3134\times10^{-11}$ & 800 & 1.0282\\
 & CS6 & 1600 & $4.3523\times10^{-11}$ & 800 & 0.9993\\
 & FB-S4 & 1600 & $8.1635\times10^{-13}$ & 400 & 4.0010\\
 & FB-S6 & 1600 & $8.4021\times10^{-14}$ & 100 & 5.9949\\\bottomrule
\end{tabular}
\end{table}
Raw spinor differences below the heuristic $10^{-12}$ threshold are
retained as measured numbers, not interpreted as resolved orders.
The threshold is not a certified bound on roundoff. Unrounded endpoint and phase/density records are identified in the
reproducibility archive discussed in Section~\ref{sec:reproducibility};
their full histories are not duplicated here.

\FloatBarrier
\subsection{Initial and terminal-state diagnostics}\label{subsec:scalar-endpoint-diagnostics}
The saved scalar FB-S6 terminal array uses \(\tau=1/6400\), or 160 macrosteps. The initial state is reconstructed from the analytic profile and discrete normalization, rather than from a saved time-zero array. The Taylor identity \(D_{2,h}=(E_h''(0)-A_h^2)/2\) and the independent sum of stage-generator commutators give terminal values agreeing to \(4.0\times10^{-12}\|U\|_h\) in the action itself. This is a numerical agreement, not an interval error certificate.

Define \(\eta_\theta(U)\) as the relative DFT norm in the set \(\max_\alpha |k_\alpha|h_\alpha/\pi\ge\theta\). Also let \(b_1(U)\) be the fraction of discrete mass in the boundary shell where at least one \(|x_\alpha|\ge\ell_\alpha-1\). Table~\ref{tab:scalar-endpoints} reports both, alongside \(d_h(U)=\|D_{2,h}U\|_h/\|U\|_h\). The spectra are those of the original arrays, with no filtering or interpolation to a different grid.
\begin{table}[H]
\centering\small
\caption{Scalar initial and S6 terminal diagnostics. ``Unresolved'' means that cancellation and floating-point representation dominate the computed initial defect; it does not mean that the exact coefficient vanishes.}
\label{tab:scalar-endpoints}
\begin{tabular}{@{}lcccc@{}}\toprule
State & \(d_h\) & \(\eta_{0.5}\) & \(\eta_{0.8}\) & \(b_1\)\\\midrule
Initial & unresolved & \(2.68\times10^{-16}\) & \(1.80\times10^{-16}\) & \(2.88\times10^{-36}\)\\
FB-S6 endpoint & 7.822731 & 0.2991155 & 0.1014278 & \(2.384\times10^{-7}\)\\
\bottomrule
\end{tabular}
\end{table}

The initial defect cannot be resolved reliably: rebuilding the analytic
state with increased intermediate precision changes its computed value
from roughly $7.8\times10^{-12}$ to $3.9\times10^{-15}$.
At the saved FB-S6 endpoint, independent evaluations of $D_{2,h}U$
agree within $10^{-11}\|U\|_h$. The derivative-residual bound from
Proposition~\ref{prop:derivative-residual} gives $3064.56\|U\|_h$,
about $392$ times the measured defect norm; it is an upper bound,
not a quantitative prediction of the cancellation.

The outer-band fraction alone does not determine the statewise defect.
As a sensitivity check on the \emph{same stored terminal array}, a hard
outer-20\% spectral cutoff gives $d_h=7.5304$, a raised-cosine taper
covering normalized frequencies $q\in(0.7,0.9)$ gives $d_h=0.4253$,
and a physical window in width-one boundary strips gives
$d_h=7.6148$, compared with $7.8227$ for the original array.
Their relative state changes are $0.1014$, $0.0902$ and
$2.92\times10^{-4}$. These are postprocessing perturbations, not
reruns of the dynamics; they mix frequency and boundary effects and
do not identify a unique cause.

For a nonzero grid state $U$, define the normalized reversal residuals
\begin{align}
 E_{\rm raw}(\delta;U)
 &=\frac{\|E_h(-\delta)E_h(\delta)U-U\|_h}{\|U\|_h},\nonumber\\
 E_{\rm adj}(\delta;U)
 &=\frac{\|E_h^\dagger(-\delta)E_h(\delta)U-U\|_h}{\|U\|_h}.
 \label{eq:supp-reversal-residual-definitions}
\end{align}
The first probe applies the unmodified negative EEI step. The second
returns with the actual inverse $E_h^\dagger(-\delta)=E_h(\delta)^{-1}$,
so $E_{\rm adj}$ vanishes in exact arithmetic; its measured values
reflect finite-precision reversal error. Neither residual is phase-aligned
or mass-rescaled.

An independent linear reversal probe on the unfiltered endpoint gives
$E_{\rm raw}=1.56561\times10^{-7}$ at $\delta=10^{-4}$, versus the
quadratic prediction $2\delta^2d_h=1.56455\times10^{-7}$.
Across $\delta=2\times10^{-4},10^{-4},5\times10^{-5}$ the relative
remainder after removing $2\delta^2D_{2,h}U$ is
$0.01579,0.007891,0.003945$, consistent with an $O(\delta)$
relative term. The largest recorded scalar FB-S6 mass-rescaling
correction is $2.78\times10^{-15}$; the corresponding data do not
constitute an unrescaled convergence experiment.

\FloatBarrier

\subsection{Observed-work thresholds in the scalar benchmark}\label{subsec:scalar-observed-work}
The following table records the first adjacent-step pair meeting each raw self-difference threshold, with both endpoint integrations counted. The notation $n:(\#\mathrm{EEI},\#\mathrm{NL})$ gives the coarser denominator $n$ and the total linear/nonlinear calls. These records illustrate a threshold-dependent tradeoff, not an equal-accuracy work--precision comparison.

\begin{table}[H]
\centering
\small
\caption{Stage-call counts at three raw self-difference thresholds in the scalar benchmark. A dash means that no recorded pair reaches the threshold.}
\label{tab:scalar-observed-work}
\begin{tabular}{@{}cccc@{}}
\toprule
Threshold & Order & Compact CS: \(n:(\#\mathrm{EEI},\#\mathrm{NL})\) & Paired FB: \(n:(\#\mathrm{EEI},\#\mathrm{NL})\) \\
\midrule
\(10^{-4}\) & 4 & \(800:(240,180)\)   & \(800:(360,180)\) \\
             & 6 & \(800:(480,420)\)   & \(800:(840,420)\) \\
\(10^{-5}\) & 4 & \(3200:(960,720)\) & \(1600:(720,360)\) \\
             & 6 & \(3200:(1920,1680)\) & \(800:(840,420)\) \\
\(10^{-6}\) & 4 & ---                 & \(3200:(1440,720)\) \\
             & 6 & ---                 & \(1600:(1680,840)\) \\
\bottomrule
\end{tabular}
\end{table}
At \(10^{-4}\), the compact methods use fewer EEI calls at each nominal order; at \(10^{-5}\), the recorded FB pairs use fewer. Neither compact method reaches \(10^{-6}\) in the tested window. Since denominators double between measurements, the first qualifying pair can overshoot a threshold appreciably. These stage counts illustrate a benchmark-specific tradeoff; self-differences are not errors against a common reference solution, and the table is not an equal-accuracy work--precision or wall-clock comparison.

\FloatBarrier

\section{Representation errors and complete-map diagnostics}\label{sec:representation}

Let \(\mathscr T_h\) be the span of the signed resolved Fourier modes on \(\mathcal D=\prod_{\alpha=x,y,z}[-\ell_\alpha,\ell_\alpha)\), and let \(\Pi_h\) be the \(L^2(\mathcal D)\)-orthogonal projection onto it. Let \(\mathcal J_h\) denote interpolation at the centered grid nodes. Both are identity maps on \(\mathscr T_h\), but they differ on unresolved functions. The integer band is \(\Lambda_h=\prod_\alpha\{-N_\alpha/2,\ldots,N_\alpha/2-1\}\), and \(\bm k(\mathbf m)=(\pi m_x/\ell_x,\pi m_y/\ell_y,\pi m_z/\ell_z)\). Here \(q_\tau\) is the first physical-space quadratic phase of the seven-factor EEI and \(q_0=q_\tau|_{\tau=0}\); \(Q_1(\tau)v=e^{-i\tau q_\tau}v\) denotes its continuous action. The main article defines the coefficient explicitly and proves nonvanishing of the complete grid defect independently of the following estimates.

\subsection{A bound for the complete defect}\label{subsec:complete-defect-bound}
The following identity connects the complete coefficient to the interpolation errors of the intermediate generators. It uses derivatives of those errors, not only the Fourier tail of the input. Let \(\mathcal S_h\) sample a function at the grid nodes, let \(v\in\mathscr T_h\), and write \(U=\mathcal S_hv\). For \(\bm\Omega=(a,b,c)^T\), the continuous first-stage generators in application order are
\[
(B_1,\ldots,B_7)=\bigl(-iq_0,(ay-bx)\partial_z,cx\partial_y,
 i\Delta/2,-az\partial_y,(bz-cy)\partial_x,iq_0\bigr),
\]
where multiplication by \(q_0\) is understood and
\[
q_0=\tfrac14\{-(b^2+c^2)x^2+(c^2-a^2)y^2+(a^2+b^2)z^2\}.
\]
Define the smooth interpolation residuals \(r_j=(I-\mathcal J_h)B_jv\); thus \(\mathcal S_hr_j=0\). Their derivatives may be nonzero at the nodes. Define
\[
\epsilon_1=\max_{j,\alpha}\|\mathcal S_h\partial_\alpha r_j\|_h,
\quad \epsilon_2=\max_j\|\mathcal S_h\Delta r_j\|_h,
\quad C_{\Omega,\ell}=(|b|+|c|)\ell_x+(|a|+|c|)\ell_y+(|a|+|b|)\ell_z.
\]
\begin{myproposition}[Derivative-residual estimate]\label{prop:derivative-residual}
For the seven-factor grid realization,
\begin{align}
D_{2,h}U&=\tfrac12\sum_{r>s}\mathcal S_h(B_s r_r-B_r r_s),\label{eq:complete-residual-identity}\\
\|D_{2,h}U\|_h&\le3C_{\Omega,\ell}\epsilon_1+\tfrac32\epsilon_2.\label{eq:complete-residual-bound}
\end{align}
The same estimate holds componentwise for a spinor, with vector-valued norms.
\end{myproposition}
\begin{proof}
The quadratic logarithmic coefficient of the continuous exact factorization is zero. Its single-stage second coefficients are represented directly at nodes, so only the composed first generators contribute an interpolation discrepancy. Subtracting the continuous commutator sum from the discrete sum gives
\[
[B_{r,h},B_{s,h}]U-\mathcal S_h[B_r,B_s]v
=-\mathcal S_hB_r r_s+\mathcal S_hB_s r_r,
\]
which proves~\eqref{eq:complete-residual-identity}. At nodes, \(B_1r_j=B_7r_j=0\). Each of the four first-derivative generators is bounded by its coordinate-coefficient maximum times \(\epsilon_1\); their maxima sum to \(C_{\Omega,\ell}\). The Laplacian term is bounded by \(\epsilon_2/2\). Each generator occurs six times in the pair sum, with the prefactor \(1/2\). The triangle inequality therefore gives~\eqref{eq:complete-residual-bound}.
\end{proof}

All derivatives here are ordinary derivatives before sampling; no periodicity of a coordinate polynomial is assumed. The residuals include both spectral truncation and the effect of periodically representing coordinate-weighted profiles. Derivatives can amplify small residuals, while the commutator sum can cancel them. Consequently a small unweighted tail of \(U\) alone is neither an estimate of \(d_h(U)\) nor a converse lower bound. Equation~\eqref{eq:complete-residual-bound} is state dependent and allows constants hidden in the residuals to depend on resolution. Independent product-rule evaluations verify~\eqref{eq:complete-residual-identity} in the archived small-grid tests.

\subsection{The first phase and localized profiles}
\label{subsec:first_stage_tail}

Define the first-stage source strength by
\[
 \mathcal T_h(v_h)=\|(I-\Pi_h)(q_0v_h)\|_{L^2(\mathcal D)}.
\]
Since \(q_\tau=q_0+O(\tau)\) on the bounded box, Taylor expansion gives
\begin{align}
 (I-\Pi_h)Q_1(\tau)v_h
 &=-i\tau(I-\Pi_h)(q_0v_h)+O(\tau^2),\label{eq:first_stage_tail}\\
 (\mathcal J_h-\Pi_h)Q_1(\tau)v_h
 &=-i\tau(\mathcal J_h-\Pi_h)(q_0v_h)+O(\tau^2).
 \label{eq:first_stage_aliasing}
\end{align}
The first identity holds in \(L^2(\mathcal D)\), and the second on the fixed grid. In particular,
\(\|(I-\Pi_h)Q_1(\tau)v_h\|=|\tau|\mathcal T_h(v_h)+O(\tau^2)\).
The absolute value is needed for signed steps. These residuals do not measure loss of mass in \(Q_{1,h}\), which is unitary.

For clarity, use continuous Fourier coefficients normalized as
\[
 \widehat f(\mathbf m)=\frac1{|\mathcal D|}\int_{\mathcal D}
 f(\mathbf x)e^{-i\bm k(\mathbf m)\cdot\mathbf x}\,d\mathbf x,
 \qquad \mathbf m\in\mathbb Z^3.
\]
For a Fourier mode \(v_{\mathbf m}=e^{i\bm k(\mathbf m)\cdot\mathbf x}\in \mathscr T_h\), multiplication shifts coefficients, so
\begin{equation}
 \mathcal T_h(v_{\mathbf m})^2
 =\textstyle |\mathcal D|\sum_{\mathbf n\notin\Lambda_h}|\widehat q_0(\mathbf n-\mathbf m)|^2.
 \label{eq:shifted_tail}
\end{equation}
This formula gives two simple tests. First, for a constant or a fixed low mode the source decays only algebraically. Indeed, in one dimension with half-width \(\ell\),
\[
 x^2=\frac{\ell^2}{3}+\sum_{n\ne0}
 \frac{2\ell^2(-1)^n}{\pi^2n^2}e^{in\pi x/\ell}
 \quad\hbox{on }[-\ell,\ell],
\]
and the unresolved norm of a nonzero coordinate-square term is proportional to the square root of a tail of \(\sum n^{-4}\). Along a fixed-aspect-ratio refinement, with \(N=\min_\alpha N_\alpha\), this is of order \(N^{-3/2}\). Second, a mode approaching the cutoff can retain a source of order one: if \(\widehat q_0(\mathbf n_0)\ne0\) and \(\mathbf m_h+\mathbf n_0\notin\Lambda_h\), that one coefficient already gives
\(\mathcal T_h(v_{\mathbf m_h})\ge |\mathcal D|^{1/2}|\widehat q_0(\mathbf n_0)|\).

Localized profiles behave differently at the source level.
\begin{mylemma}[A localized first-stage source]
\label{lem:gaussian_tail}
Let \(v(\mathbf x)=p(\mathbf x)e^{-\mu|\mathbf x|^2}\), where \(p\) is a fixed polynomial and \(\mu>0\), and let \(v_h=\Pi_hv\). Set
\(\kappa_h=\min_\alpha \pi N_\alpha/(2\ell_\alpha)\) and \(\ell_{\min}=\min_\alpha \ell_\alpha\).
For the fixed box, there are constants independent of the grid resolution such that
\begin{equation}
 \mathcal T_h(v_h)\le C_1e^{-c\kappa_h^2}+C_2e^{-\mu \ell_{\min}^2}.
 \label{eq:gaussian_tail}
\end{equation}
The constants may depend on \(p,q_0,\mu\) and the box; \(C_2\) includes polynomial factors in its widths.
\end{mylemma}
\begin{proof}
Put \(w=q_0v\). Periodize \(v\) and \(w\) separately by
\[
 v^{\rm per}(x)=\sum_{\mathbf m\in\mathbb Z^3}v(x+2\bm\ell\odot\mathbf m),\qquad
 w^{\rm per}(x)=\sum_{\mathbf m\in\mathbb Z^3}w(x+2\bm\ell\odot\mathbf m),
\]
where \(\bm\ell=(\ell_x,\ell_y,\ell_z)\) and \(2\bm\ell\odot\mathbf m=(2\ell_xm_x,2\ell_ym_y,2\ell_zm_z)\) denotes componentwise multiplication. By the Poisson summation formula, the Fourier coefficients of each periodic representative are samples of the corresponding whole-space Fourier transform. Those transforms are a polynomial times a Gaussian, so their Fourier tails are bounded by \(Ce^{-c\kappa_h^2}\). On the original box, the noncentral image sums satisfy
\[
 \|v^{\rm per}-v\|_{L^2(\mathcal D)}+
 \|w^{\rm per}-w\|_{L^2(\mathcal D)}
 \le C e^{-\mu \ell_{\min}^2},
\]
with polynomial factors in the box widths absorbed into \(C\). Applying \(\Pi_h\) to the two periodic representatives and using the triangle inequality gives the stated two-term estimates for \((I-\Pi_h)v\) and \((I-\Pi_h)(q_0v)\). Finally,
\[
 \|(I-\Pi_h)(q_0(v_h-v))\|_{L^2}
 \le\|q_0\|_{L^\infty(\mathcal D)}\|v_h-v\|_{L^2},
\]
which has the same bound.
\end{proof}

For a sufficiently large box and a resolved Gaussian profile, both terms can be very small. The estimate is a fixed-box statement: the boundary term is controlled by the chosen box and does not disappear merely by increasing \(N\) at fixed \(\ell_\alpha\). Smooth profiles supported away from the boundary similarly give superalgebraic source decay. A localized polynomial--Gaussian vortex profile is included in Lemma~\ref{lem:gaussian_tail}; a vortex by itself is therefore not a guarantee of a large source. Oscillatory profiles also need separation from the Fourier cutoff.

The lemma concerns \(\mathcal T_h\), not \(\|D_{2,h}v_h\|_h\). Discrete derivatives in the complete defect can amplify small residuals. Moreover, a later unitary stage can undo an earlier perturbation; the absence of smoothing does not preclude cancellation. Intermediate localization is consequently a useful explanation to test, not a replacement for complete-map diagnostics along an evolving state.

An elementary shear has an analogous but different source. For
\(T_{\alpha\leftarrow\beta}(s)v(\mathbf x)=v(\mathbf x-sx_\beta e_\alpha)\), with \(\alpha\ne\beta\), Taylor expansion gives
\[
 (I-\Pi_h)T_{\alpha\leftarrow\beta}(s)v_h
 =-s(I-\Pi_h)(x_\beta\partial_\alpha v_h)+O(s^2).
\]
It vanishes on constants and modes independent of \(x_\alpha\). For a mode depending on \(x_\alpha\), the coordinate factor \(x_\beta\) has Fourier coefficients of order \(1/|n|\), yielding an \(N^{-1/2}\) tail. These concern individual stages; a palindromic product can cancel quadratic terms despite nonzero shear tails.

\subsection{Intermediate Fourier content}
\label{subsec:unresolved_tail_diagnostics}

On \([ -\pi,\pi)^3\), take \(\bm\Omega=(-0.7,1,-\sqrt3)^T\) and
\(N=24,32,48,64,96,128\) nodes per direction. The diagnostic uses a doubled grid with \(2N\) nodes per direction. Let \(\Pi_{N\leftarrow2N}\) retain, in its DFT, precisely the original signed band \([-N/2,N/2-1]^3\). For functions \(w,v\), denote their doubled-grid samples by \(w_f,v_f\) and define
\[
 \mathcal T_N(w;v)={\|(I-\Pi_{N\leftarrow2N})w_f\|_{h/2}}\big/{\|v_f\|_{h/2}}.
\]
This is a finite-band, sampled approximation to an unresolved continuous tail, not its exact norm. The two plotted quantities are
\[
 \mathcal T_E(v)=\mathcal T_N(q_{\tau_0}v;v),\qquad
 \mathcal T_G^{x\leftarrow y}(v)=\mathcal T_N(y\partial_xv;v),
 \qquad \tau_0=10^{-4}.
\]
The first uses the finite-step phase \(q_{\tau_0}\) to approximate \(q_0\). The main article\textquotesingle s formula for \(q_0\) shows that \(q_0\) is exactly diagonal as a quadratic form; \(q_{\tau_0}\) also has \(O(\tau_0)\) off-diagonal terms. On a periodic box, the coordinate cross terms \(xy,xz,yz\) can have slower one-direction Fourier tails than the diagonal square terms. Hence the \(N^{-3/2}\) description below is intentionally restricted to the plotted range, where the diagonal \(q_0\)-contribution is the visible behavior; it is not an \(N\to\infty\) statement for fixed nonzero \(\tau_0\).

\begin{figure}[tbp]
 \centering
 \includegraphics[width=0.92\linewidth]{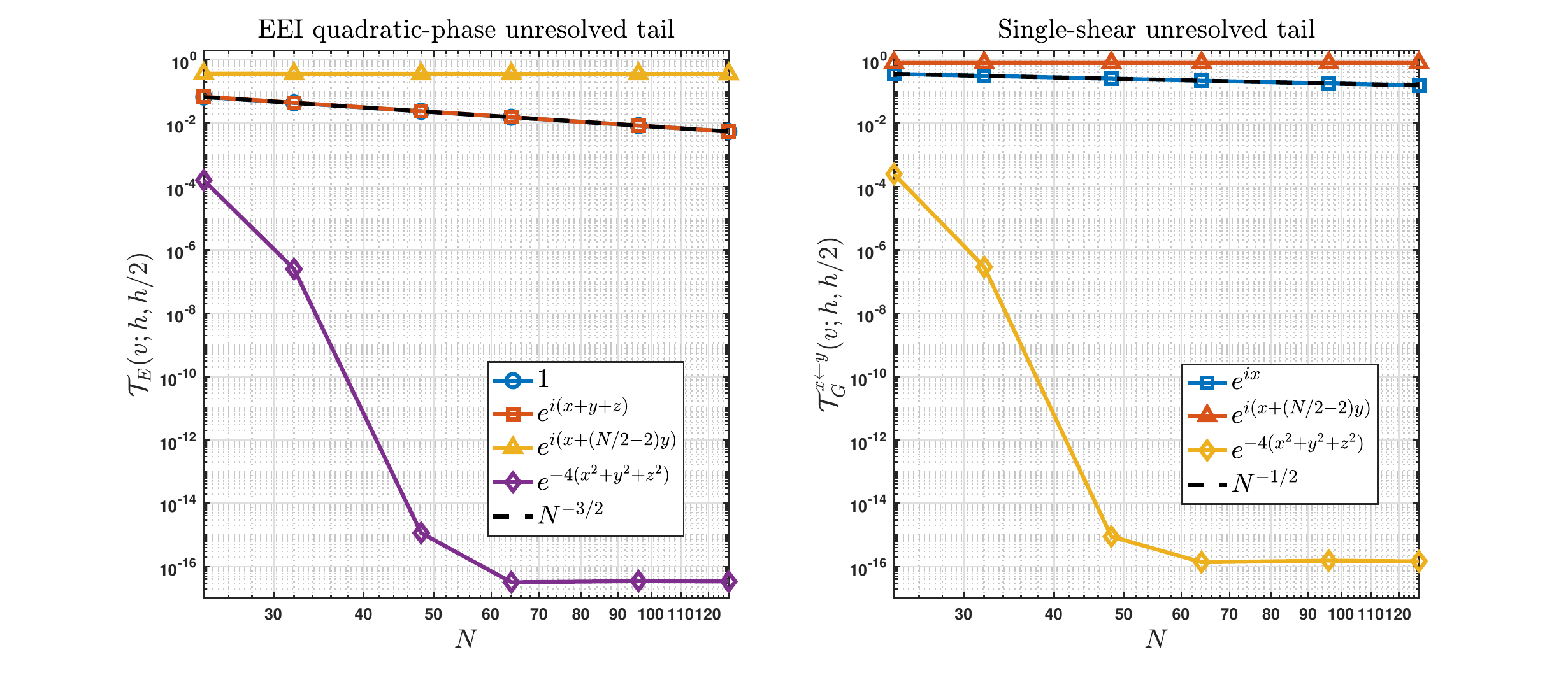}
 \caption{Finite-band unresolved-tail diagnostics on a doubled Fourier grid. Left: the EEI quadratic-phase source. Right: the source of a single \(x\)-directed shear. The cutoff-near mode is \(e^{i(x+(N/2-2)y)}\). These are intermediate-stage diagnostics, not complete-step errors.}
 \label{fig:unresolved_tail_diagnostics}
\end{figure}

In Figure~\ref{fig:unresolved_tail_diagnostics}, the EEI sources for the constant and fixed low mode \(e^{i(x+y+z)}\) have the expected \(N^{-3/2}\) behavior over the plotted range. The cutoff-near mode retains a source of order one. The Gaussian \(e^{-4(x^2+y^2+z^2)}\) rapidly reaches the sampling/floating-point floor of this diagnostic; the plotted near-zero level is not an exact vanishing statement. For the single-shear source, the mode \(e^{ix}\) exhibits \(N^{-1/2}\) decay, whereas constants and modes depending only on \(y\) or \(z\) give zero because \(\partial_xv=0\).

These results illustrate the source estimates of this section. Small tails of localized data are compatible with the high orders observed in localized-profile tests, including those of~\cite{LiuZhang2025}. The figure alone does not show that the complete defect remains small along every such trajectory. Next, we test the complete implemented propagator.

The historical archive retains complete-map reversal and equal-step group-defect checks. They are separate from the CS/FB convergence comparisons and unnecessary for the nonvanishing proof.

\FloatBarrier

\section{Stable evaluation of the implemented subflows}\label{sec:archived-scalar}
Exact-arithmetic defects differ from floating-point cancellation. The coefficient formulas below are used in the scalar CS and FB runs; the axial specialization and stable spin updates are shared by the spinor comparisons. Stable evaluation does not make the grid map symmetric.

\subsection{Cancellation-free coefficients}\label{subsec:stable-coefficients}
Direct evaluation of the phase matrix subtracts the identity from an order-one matrix before division by \(2\tau^2\). The following algebraically equivalent formulas avoid this cancellation. Use \(\bm\Omega=(a,b,c)^T\), \(\vartheta=|\bm\Omega|\tau\), and
\(A_{\rm rot}=\left(\begin{smallmatrix}0&c&-b\\-c&0&a\\b&-a&0\end{smallmatrix}\right)\).
Define
\[
 \mu_\vartheta=\tfrac12\operatorname{sinc}^2(\vartheta/2),\quad
 s_\vartheta=\operatorname{sinc}(\vartheta),\quad
 W_{\rm rot}=I+\tau s_\vartheta A_{\rm rot}+\tau^2\mu_\vartheta A_{\rm rot}^2,
\]
where \(\operatorname{sinc}(\vartheta)=\sin(\vartheta)/\vartheta\), with value one at zero. Small arguments are evaluated with their removable-limit series. Set \(\kappa_x=\cos\vartheta+a^2\mu_\vartheta\tau^2\) and \(\kappa_z=\cos\vartheta+c^2\mu_\vartheta\tau^2\).

Here \(\xi_j,\eta_j\) are the signed-step coefficients in the main article; equivalently, put \(g(u,v,w)=uv\mu_\vartheta\tau+w s_\vartheta\) and use
\[
\begin{aligned}
\xi_1&=g(-a,b,c)\kappa_x/\kappa_z,&\xi_2&=-g(a,c,b)\kappa_z,&\xi_3&=g(c,b,a)\kappa_z/\kappa_x,\\
\eta_1&=-g(a,b,c)/\kappa_x,&\eta_2&=g(a,c,b)/\kappa_z,&\eta_3&=g(c,b,-a)/\kappa_z.
\end{aligned}
\]

To stabilize \(\mathsf P\), define matrices \(V_{\tau},U_{\tau},D_{\tau}\) by
\[
 W_{\rm rot}^T=I+\tau V_{\tau},\quad
 T_\xi=I+\tau U_{\tau},\quad
 T_\eta=I+\tau D_{\tau},
\]
so that
\[
 V_{\tau}=-s_\vartheta A_{\rm rot}+\tau \mu_\vartheta A_{\rm rot}^2,\quad
 U_{\tau}=\begin{pmatrix}0&\xi_1&\xi_2+\tau\xi_1\xi_3\\0&0&\xi_3\\0&0&0\end{pmatrix},\quad
 D_{\tau}=\begin{pmatrix}0&0&0\\\eta_1&0&0\\\eta_2&\eta_3&0\end{pmatrix}.
\]
The direct formula is \(\mathsf P=\operatorname{sym}(W_{\rm rot}^TT_\xi T_\eta-I)/(2\tau^2)\), where \(\operatorname{sym}(B)=(B+B^T)/2\). Expanding the product gives
\begin{equation}
 \mathsf P=\tfrac12\operatorname{sym}
 \left(J+V_{\tau}U_{\tau}+V_{\tau}D_{\tau}+U_{\tau}D_{\tau}+\tau V_{\tau}U_{\tau}D_{\tau}\right),
 \quad J=(V_{\tau}+U_{\tau}+D_{\tau})/\tau.
 \label{eq:stable_P_identity}
\end{equation}
The last quotient is not evaluated by subtraction and division. Its diagonal is \(\mu_\vartheta\operatorname{diag}(A_{\rm rot}^2)\), and substitution of the scalar coefficients gives
\begin{align*}
 J_{12}&=(a^2-c^2)\mu_\vartheta\tau(cs_\vartheta-ab\mu_\vartheta\tau)/\kappa_z,\; 
 J_{13}=(a^2+b^2)\mu_\vartheta\tau(bs_\vartheta+ac\mu_\vartheta\tau)+\xi_1\xi_3,\\
 J_{21}&=-(b^2+c^2)\mu_\vartheta\tau(ab\mu_\vartheta\tau+cs_\vartheta)/\kappa_x,\; 
 J_{23}=a(c^2-a^2)\mu_\vartheta\tau s_\vartheta/\kappa_x+bc \mu_\vartheta(1+\kappa_z/\kappa_x),\\
 J_{31}&=b(a^2+b^2)\mu_\vartheta\tau s_\vartheta/\kappa_z+ac \mu_\vartheta(1+1/\kappa_z),\;
 J_{32}=-a(a^2+b^2)\mu_\vartheta\tau s_\vartheta/\kappa_z+bc \mu_\vartheta(1+1/\kappa_z).
\end{align*}
These expressions follow by substituting the scalar coefficients and collecting powers of \(\tau\); for example, \(\kappa_z-\kappa_x=(c^2-a^2)\mu_\vartheta\tau^2\) removes the apparent division in \(J_{12}\). The archive checks all entries at positive and negative steps.

The direct phase-matrix formula in the main article and~\eqref{eq:stable_P_identity} are identical in exact arithmetic. Stable evaluation does not remove \(D_{2,h}\) or replace adjoint pairing; it avoids an additional floating-point error caused by coefficient cancellation. All scalar CS/FB convergence runs use the same cancellation-free evaluation. Earlier direct-versus-stable coefficient tests and reference-sensitivity records remain in the historical archive; they are separate from the compact-CS comparisons.

\FloatBarrier

\subsection{Spinor updates and verification scope}
For spin 1, if \(\mathsf S=\bm F\cdot\bm f\), \(r=|\bm F|\) and \(\theta=c_1t\), the polynomial update is evaluated as
\[
 e^{-i\theta\mathsf S}u=u-i\theta\sinc(\theta r)\mathsf S u
 -\frac{\theta^2}{2}\sinc^2(\theta r/2)\mathsf S^2u.
\]
This avoids forming \(\cos(\theta r)-1\) by subtraction. Here \(\rho=u^*u\), \(F_j=u^*f_ju\), and the spin matrices are those defined in the main article. The axial EEI uses half-angle and factored phase-matrix formulas, not a difference of two order-one matrices divided by a tiny step squared. Both CS and FB call these same subflow routines.

For spin 2, let \(\mathcal A\) be the singlet-pairing matrix of the main article, \(A_{00}=u^T\mathcal A u\), and \(v=\mathcal A\overline u\). Then \(\|v\|^2=\rho/5\) and \(u^*v=\overline{A_{00}}\). Near saturation, use the nonnegative Gram identity
\[
 \chi^2=\frac15\left(\|u\|^2\|v\|^2-|u^*v|^2\right)
 =\frac15\sum_{j<k}|u_jv_k-u_kv_j|^2.
\]
It equals \((\rho/5)^2-|A_{00}|^2/5\), but avoids subtracting those nearly equal quantities. This refinement is included in the current code release; it does not alter the spin-1 benchmarks or their previously generated endpoints. Small-grid cross-language tests exercise all three spin classes. Stable local formulas do not prevent accumulated FFT roundoff or loss of relative accuracy in the difference of nearly identical terminal states.

\FloatBarrier

\section{Symmetry levels, legal merging, and call counts}\label{sec:calls}

A primitive linear call is one evaluation of a prescribed first-order map \(L_{F,h}\) or its method adjoint \(L_{F,h}^\dagger(t)=L_{F,h}(-t)^{-1}\), not one FFT. The two symmetric second-order constructions compared analytically in the main article are
\begin{align*}
L_{{\rm sym},h}(t)&=L_{F,h}^\dagger(t/2)\circ L_{F,h}(t/2),\\
S_{2,h}^{\rm lin}(\tau)&=L_{{\rm sym},h}(\tau/2)\circ N_h(\tau)\circ L_{{\rm sym},h}(\tau/2),\\
S_{2,h}^{\rm FB}(\tau)&=L_{F,h}^\dagger(\tau/2)\circ N_h(\tau)\circ L_{F,h}(\tau/2).
\end{align*}
The rightmost map acts first. Both are consistent and self-adjoint; neither claims exact linear evolution. Their isolated-step counts are respectively four and two primitive linear calls, with one nonlinear call in each. The main article proves that three total calls are minimal in the stated map-independent black-box class and characterizes the two FB orientations. The result uses classical lowest-order composition conditions~\cite{HairerLubichWanner2006,blanes2008splitting}; it is not a lower bound on all high-order integrators, internal FFTs, or equal-error work.

\paragraph{Composition of complete blocks}
For a symmetric composition with \(r_p\) blocks, \(r_2=1\), \(r_4=3\), \(r_6=7\), linear-level repair uses \(4r_p\) linear and \(r_p\) nonlinear calls, whereas FB uses \(2r_p\) and \(r_p\). Both have designed order \(p\) under the common smoothness and exact-coefficient assumptions. Their leading error constants may differ. Complete approximate linear maps cannot be merged merely because they are self-adjoint: neither symmetry nor first-order consistency gives the group identity \(L(s)L(t)=L(s+t)\).

\paragraph{Factor-level palindromes and the separate shear alternative}
Suppose \(L_{F,h}(t)=\Phi_m(t)\circ\cdots\circ\Phi_1(t)\), with each \(\Phi_j\) individually self-adjoint. A linear-level palindrome \(L_{{\rm sym},h}\) has \(2m\) factor calls, or \(2m-1\) if the two central \(\Phi_m(t/2)\) factors belong to an exact one-parameter group and are merged. Thus repairing both linear halves separately uses \(4m-2\) factor calls under this fusion, while a complete FB block uses \(2m\). For the fully active eight-module ordering in the main article, this is \(30\) versus \(16\). Only the merged pair requires a group property; self-adjoint RSDA rotation modules are not assumed to be exact rotation groups. An independent palindromic shear propagator has its own factor list and cost. It is a valid symmetric alternative, not four calls to an EEI black box, and is not covered by a blanket factor-of-two FFT claim.

\paragraph{Alternative nonlinear--linear--nonlinear arrangement}
The symmetric block
$S_{2,h}^{\rm NLN}(\tau)=N_h(\tau/2)\circ
L_{{\rm sym},h}(\tau)\circ N_h(\tau/2)$
uses two primitive linear calls and two nonlinear half-calls in
isolation. In a sequence of $J$ signed blocks, adjacent exact
nonlinear half-flows can be merged, leaving $2J$ linear and at most
$J+1$ nonlinear calls, compared with $2J$ and $J$ for FB.
If both endpoint factors of the linear palindrome obey exact group
laws, additional fusion can give comparable leading factor counts.
Thus the isolated-step black-box minimum in the main article is
not an equal-work or all-methods lower bound.

\paragraph{Compact CS and FB costs}
For the three-coefficient fourth-order composition, compact CS4 uses four EEI and three nonlinear calls; FB-S4 uses six and three. At order six the corresponding counts are eight versus fourteen EEI calls and seven nonlinear calls. Linear-level repair instead requires twelve and twenty-eight EEI calls at these orders. These comparisons answer different questions: FB avoids redundant work relative to nested repair, but costs more per macrostep than the compact formulas whose defect is analyzed.

The scalar mixed-space implementation uses three directional FFT/inverse-FFT pass pairs per EEI call. The axial spinor \emph{unfused historical production code} evaluates its mixed Fourier multipliers separately and uses five pairs. The fused spinor implementation retains the same six algebraic factors while cancelling adjacent inverse/forward transforms, so it uses three pairs, also for the method adjoint. A directional pair is a transform plus inverse transform along all corresponding grid lines and spin components. Thus the fused scalar and axial spinor CS4/FB-S4 counts are \(12/18\) pairs and the CS6/FB-S6 counts \(24/42\). The historical unfused spinor CS4/FB-S4 and CS6/FB-S6 costs were \(20/30\) and \(40/70\) pairs, respectively. The linear-level construction has \(36/84\) scalar or fused-spinor pairs at orders four/six, versus \(60/140\) for unfused spinor. These are internal EEI-only counts: nonzero SOC modules require additional unchanged transforms, and no cross-macrostep full-EEI group identity is invoked. Historical convergence CSVs were obtained with the unfused code and have not been replaced. The linear-level construction remains an analytic benchmark, not a separate numerical dataset.

\Needspace{17\baselineskip}
\paragraph{A symmetric merged palindrome need not be fourth order}
One might instead merge the adjacent linear stages in the fourth-order triple jump and restore a palindrome by using method adjoints. Consider the four-linear-call arrangement
\[
\mathcal P_h(\tau)=L_h^\dagger(c_1\tau)\,N_h(\gamma\tau)\,
L_h^\dagger(c_2\tau)\,N_h(\gamma_0\tau)\,
L_h(c_2\tau)\,N_h(\gamma\tau)\,L_h(c_1\tau),
\]
where \(\gamma=1/(2-2^{1/3})\), \(\gamma_0=1-2\gamma\), \(c_1=\gamma/2\), and \(c_2=(1-\gamma)/2\). The arrangement is self-adjoint and first-order consistent. It is \emph{not} generically fourth order: take the admissible commuting skew-Hermitian matrices \(A,P,K\) with \(P\ne0\), \(K\ne0\), and set \(L_h(t)=\exp(tA+t^2P+t^3K)\) and \(N_h(t)=I\). Then \(L_h^\dagger(t)=\exp(tA-t^2P+t^3K)\), and direct multiplication gives
\[
\mathcal P_h(\tau)=\exp\!\left(\tau A+2(c_1^3+c_2^3)\tau^3K\right),\qquad
2(c_1^3+c_2^3)=\frac{\gamma^3+(1-\gamma)^3}{4}\approx0.60591526.
\]
Thus this symmetric four-call palindrome is not generically fourth order. This does not exclude other EEI-specific constructions or establish a general call-count lower bound.

For comparison, composing three complete FB blocks gives a common third-order modified term multiplied by \(2\gamma^3+\gamma_0^3=0\), which the compact four-call counterexample does not reproduce. The example changes the high-order stage assembly. It does not apply to composing either valid second-order block \(S_{2,h}^{\rm lin}\) or \(S_{2,h}^{\rm FB}\) with the exact fourth-order coefficients. It explains why a symmetric rearrangement alone is insufficient, not why properly composed linear-level repair would fail, and it gives no general high-order call-count lower bound.

\FloatBarrier

\section{Additional spinor diagnostics and model coverage}\label{sec:spinor-complete}

\subsection{Initial-state EEI diagnostics}\label{subsec:eei-structure}

The initial-state comparisons use the main article's spin-1 model with \(\bm\Omega=(0,0,0.2)\), \(c_0=10\), \(c_1=1\), \(M=A_j=0\), and harmonic trap. On \([-8,8)^3\), the polynomial--Gaussian state is \(\Psi_{\rm pg}^0=C_{\rm pg}e^{-|\bm x|^2/2}(x^2,y^2,z^2)^T\). The perturbed Gaussian \(\Psi_{\rm pert}^0\) has the three components defined in the main article, with perturbation size \(0.1\) and carriers \((3,4,2)\), \((-3,4,-2)\). Each state is normalized on its own grid. The residuals \(E_{\rm raw}\) and \(E_{\rm adj}\) are those defined in~\eqref{eq:supp-reversal-residual-definitions}; the same norm conventions apply to these spinor tests.

\begin{table}[tbp]
\centering
\caption{Raw negative-step and method-adjoint EEI reversal residuals for the two spin-1 initial states at \(h=1/2\). Forward steps preserve mass to roundoff.}
\label{tab:eei-linear-reversibility}
{\footnotesize
\begin{tabular}{ccccc}
\toprule
& \multicolumn{2}{c}{Polynomial--Gaussian}
& \multicolumn{2}{c}{Perturbed Gaussian}\\
\cmidrule(lr){2-3}\cmidrule(l){4-5}
\(\tau\) & \(E_{\rm raw}\) & \(E_{\rm adj}\) & \(E_{\rm raw}\) & \(E_{\rm adj}\)\\
\midrule
\(4.0\times10^{-2}\) & \(3.7904\times10^{-10}\) & \(6.5872\times10^{-16}\) & \(4.7053\times10^{-5}\) & \(6.3684\times10^{-16}\)\\
\(2.0\times10^{-2}\) & \(9.4855\times10^{-11}\) & \(6.7596\times10^{-16}\) & \(1.1822\times10^{-5}\) & \(6.3907\times10^{-16}\)\\
\(1.0\times10^{-2}\) & \(2.3721\times10^{-11}\) & \(6.6559\times10^{-16}\) & \(2.9619\times10^{-6}\) & \(6.5649\times10^{-16}\)\\
\(5.0\times10^{-3}\) & \(5.9308\times10^{-12}\) & \(6.6880\times10^{-16}\) & \(7.4121\times10^{-7}\) & \(6.4406\times10^{-16}\)\\
\midrule
last rate & \(2.0000\) & -- & \(1.9985\) & --\\
\bottomrule
\end{tabular}
}
\end{table}

Direct evaluation of the axial defect formula in the main article gives \(\|D_{2,h}\Psi^0\|_h/\|\Psi^0\|_h=1.18623\times10^{-7}\) for \(\Psi_{\rm pg}^0\) and \(1.48385\times10^{-2}\) for \(\Psi_{\rm pert}^0\). By the reversal-residual expansion in the main article, the predicted residuals at \(\tau=0.005\) are \(5.9312\times10^{-12}\) and \(7.4193\times10^{-7}\), respectively, agreeing with Table~\ref{tab:eei-linear-reversibility} to about \(0.007\%\) and \(0.10\%\). Roundoff \(E_{\rm adj}\) values check the inverse identity.

An additional double-precision NumPy implementation of the axial EEI defined in the main article reproduces the raw residuals in Table~\ref{tab:eei-linear-reversibility} to the displayed precision. Table~\ref{tab:eei-fixed-data-refinement} keeps \(\Psi_{\rm pert}^0\), \(\varepsilon=0.1\), carrier \((3,4,2)\), domain, and \(\tau=0.005\) fixed. Intermediate grids resolve the decrease in \(d_h=\|D_{2,h}\Psi_{\rm pert}^0\|_h/\|\Psi_{\rm pert}^0\|_h\) before the reversal residual reaches roundoff.

\begin{table}[htbp]
\centering
\caption{Fixed-data refinement on \([-8,8)^3\) at \(\tau=0.005\): defect coefficient and reversal residual. See text for precision limits.}
\label{tab:eei-fixed-data-refinement}
\begin{tabular}{@{}lcccc@{}}
\toprule
\(h\) & \(1/2\) & \(0.4\) & \(1/3\) & \(1/4\)\\
\midrule
\(d_h\) & \(1.4839\times10^{-2}\) & \(1.6968\times10^{-4}\) & \(1.5823\times10^{-7}\) & \(9.8\times10^{-13}\)\\
\(E_{\rm raw}\) & \(7.4121\times10^{-7}\) & \(8.4792\times10^{-9}\) & \(7.9090\times10^{-12}\) & \(8.3\times10^{-16}\)\\
\bottomrule
\end{tabular}
\end{table}

Direct evaluation of \(D_{2,h}\) involves cancellation. On this \(\ell=8\) box, extended precision confirms \(d_{1/4}\approx9.8\times10^{-13}\); \(2\tau^2d_{1/4}\approx4.9\times10^{-17}\) is below reversal noise. The unresolved initial coefficients in the following large-box refinement test instead concern \(\ell=16\); their computed digits are not comparable across boxes. At \(h=1/8\), \(E_{\rm raw}=9.5\times10^{-16}\) remains at roundoff. This shows suppression on the initial state, not exact grid reversibility or persistence during evolution; Section~\ref{subsec:eei-evolving-defect} tests the latter directly.

\subsection{Evolution of the defect and Fourier tail}\label{subsec:eei-evolving-defect}

For this study the box is \([-16,16)^3\), the grid spacing is \(h=1/4\), and \(T=0.02\), with the same perturbed state and physical parameters. We write \(d_h(t)=\|D_{2,h}U(t)\|_h/\|U(t)\|_h\), evaluating the coefficient on the numerical FB trajectory. The hat in the following indicator denotes the grid DFT; \(\boldsymbol k\) ranges over its resolved frequencies.

The archived paired implementation also tracks the defect and a spectral-tail indicator on the same FB-S6 trajectory, using \(\tau=0.0025\). Define
\[
\eta_y(t)=\left(\frac{\sum_{|k_y|\ge0.8\pi/h}\|\widehat U_{\boldsymbol k}(t)\|_{\mathbb C^3}^2}
{\sum_{\boldsymbol k}\|\widehat U_{\boldsymbol k}(t)\|_{\mathbb C^3}^2}\right)^{1/2}.
\]
The saved trajectory shows simultaneous growth of this tail and
$d_h(t)$, as plotted in Figure~\ref{fig:spinor-defect-growth}:
between $t=0.005$ and $0.020$, $d_h$ rises from
$3.247\times10^{-6}$ to $1.088\times10^{-5}$ while $\eta_y$
rises from $8.113\times10^{-6}$ to $2.959\times10^{-5}$.
The initial defect is below the cancellation floor. Correlation
is not a bandwise decomposition or a causal attribution.

\begin{figure}[tbp]
\centering
\includegraphics[width=0.78\linewidth]{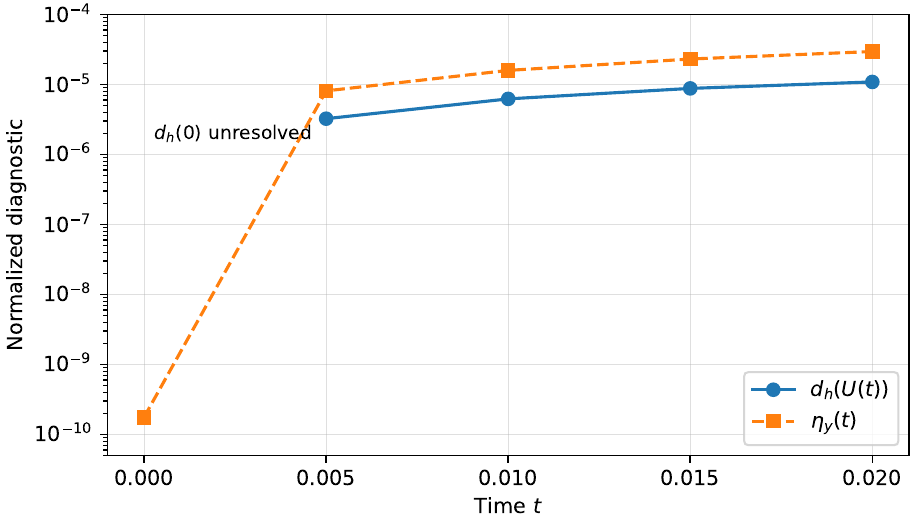}
\caption{Defect and outer-band norm on the archived FB-S6 trajectory. Lines join recorded times, not continuous measurements. The initial defect is unresolved and is omitted from its curve. These are correlated statewise diagnostics, not a bandwise decomposition of the full defect.}
\label{fig:spinor-defect-growth}
\end{figure}

\begin{table}[htbp]
\centering\small
\caption{Archived statewise refinement on \([-16,16)^3\), at \(T=0.02\), along FB trajectories. Initial evaluations below the cancellation floor are not reported as zeros. These are not new compact-CS spatial-error measurements.}
\label{tab:dynamic-defect-refinement}
\begin{tabular}{@{}llrrr@{}}\toprule
\(h\)&Endpoint method&Endpoint step&\(d_h(0)\)&\(d_h(T)\)\\\midrule
\(1/2\)&FB-S6&\(1/1600\)&\(1.43\times10^{-2}\)&\(1.41\times10^{-2}\)\\
\(0.4\)&FB-S4&\(1/1600\)&\(1.57\times10^{-4}\)&\(5.62\times10^{-4}\)\\
\(1/3\)&FB-S4&\(1/1600\)&\(1.41\times10^{-7}\)&\(2.46\times10^{-4}\)\\
\(1/4\)&FB-S6&\(1/400\)&unresolved&\(1.09\times10^{-5}\)\\
\(1/5\)&FB-S4&\(1/800\)&unresolved&\(3.50\times10^{-7}\)\\
\(1/7\)&FB-S6&\(1/400\)&unresolved&\(4.76\times10^{-10}\)\\\bottomrule
\end{tabular}
\end{table}
The endpoint methods and step sizes are listed explicitly; these are retained diagnostic trajectories, not a time-refinement comparison across different spatial grids. The \(h=1/4\) five-time trajectory above specifically uses \(\tau=0.0025\). At \(h=1/7\), the endpoint raw reversal residual at probe \(\delta=0.005\) is \(2.39\times10^{-14}\), close to \(2\delta^2d_h(T)=2.38\times10^{-14}\), while the adjoint-return residual is \(7.32\times10^{-16}\). Thus the small statewise coefficient is resolved on that endpoint even though the corresponding initial evaluation is not.

\subsection{Conventions for the model-coverage tables}

Unless a caption specifies otherwise, these tests use \([-8,8)^3\), \(V=(x^2+y^2+z^2)/2\), uniform spacing \(h\), and the eight-factor RSDA-based FB ordering of the main article, not the EEI module. The errors are phase-aligned relative errors against same-family FB-S6 references. A full-step mass-rescaling safeguard was used in these convergence tables. Archived preliminary comparisons reported no change at the displayed significant digits on omitting the safeguard, with unrescaled drift about \(10^{-13}\); this is a statement about those checks, not a new measurement for the scalar runs. Separate reversal and adaptive-ODE checks do not apply the safeguard.

\subsection{Convergence across spin classes}\label{subsec:temporal-convergence}

These tests activate arbitrary three-dimensional rotation and all three SOC directions. They cover density coupling with an internal Hamiltonian, spin exchange, and spin-2 singlet pairing. Let \(g_{\rm A}=e^{-(x^2+0.8y^2+1.2z^2)/2}\) and \(g_{\rm B}=e^{-(x^2+y^2+z^2)/2}\). Normalize each raw state by \(\Psi^0=\Psi^0_{\rm raw}/\|\Psi^0_{\rm raw}\|_h\). The spin-\(1/2\) and spin-1 raw states are
\[
\Psi^0_{1/2,\rm raw}=g_{\rm A}\begin{pmatrix}
1+0.15\cos(0.7x)\sin(0.5y)\\
0.85e^{i(0.20x-0.15y+0.10z)}
\end{pmatrix},
\]
\[
\Psi^0_{1,\rm raw}=\begin{pmatrix}
g_{\rm B}[1+\sin(\pi z)+i\cos(\pi(x+y+z))]\\1\\g_{\rm B}
\end{pmatrix}.
\]

The spin-2 state is \(g_{\rm B}(q_1,\ldots,q_5)^T\), where
\[
\begin{aligned}
q_1&=1+\sin(\pi z)+i\cos(\pi(x+y+z)),&q_2&=0.7+0.2x+0.3iy,\\
q_3&=1.2+0.1\cos(\pi(x+y)),&q_4&=0.8+0.1\sin(\pi x)-0.2i\cos(\pi y),\\
q_5&=0.6+0.3\cos(\pi z)+0.1i\sin(\pi(x-y)).
\end{aligned}
\]
Components follow the spin ordering defined in the main article; the parameters and reference steps are specified next.

The three tests use $h=1/8$ and the following configurations
(with the harmonic trap $V=|\bm x|^2/2$):
\begin{itemize}
\item Spin $1/2$: $T=0.2$, $\bm\Omega=(0.8,-0.6,0.7)$,
$A_x=2\sigma_x$, $A_y=1.3\sigma_y$, $A_z=0.7\sigma_z$,
$M=0.25\sigma_z+0.15\sigma_x+0.07\sigma_y$,
$\beta=\left(\begin{smallmatrix}1.2&0.8\\0.8&1\end{smallmatrix}\right)$,
reference step $4\times10^{-4}$.
\item Spin 1: $T=0.1$, $\bm\Omega=(-0.45,0.85,0.6)$,
$(A_x,A_y,A_z)=(0.8f_x,1.5f_y,2.4f_z)$,
$(c_0,c_1)=(400,-30)$,
$M=\left(\begin{smallmatrix}0.45&0.30&0\\0.30&0&0.20i\\0&-0.20i&0.25\end{smallmatrix}\right)$,
reference step $2.5\times10^{-4}$.
\item Spin 2: $T=0.1$, $\bm\Omega=(0.4,0.9,-0.6)$,
$(A_x,A_y,A_z)=(f_x,2f_y,0.7f_z)$,
$(c_0,c_1,c_2)=(400,40,-30)$, $M=0$,
reference step $2.5\times10^{-4}$.
\end{itemize}
The spin-1 driver adjusts its coarsest requested step $0.008$
to $1/130$ so that $T/\tau$ is integral; order calculations use the
actual step ratio. The spin-2 initial state has
$\|A_{00}\|_h/\|\rho\|_h=0.19068$ and therefore exercises pairing.

\begin{table}[htbp]
\centering\small
\caption{Representative converged-error pairs $({\rm error},p_{\rm obs})$
for FB-S2/S4/S6 against same-family FB-S6 references. All three
columns use the listed $\tau$. The selected rows are within the resolved range.}
\label{tab:spin-class-summary}
\begin{tabular}{@{}llccc@{}}\toprule
Spin & $\tau$ & FB-S2 & FB-S4 & FB-S6\\\midrule
$1/2$ & $0.025$ & $(1.5359\!\times\!10^{-4},2.0020)$ & $(5.5964\!\times\!10^{-7},3.9908)$ & $(1.9692\!\times\!10^{-10},5.9868)$\\
1 & $0.001$ & $(9.4017\!\times\!10^{-6},1.9991)$ & $(5.0496\!\times\!10^{-8},3.9453)$ & $(1.3554\!\times\!10^{-10},5.9403)$\\
2 & $0.00125$ & $(7.9137\!\times\!10^{-5},2.0042)$ & $(1.0806\!\times\!10^{-6},4.0200)$ & $(3.2858\!\times\!10^{-9},5.9571)$\\\bottomrule
\end{tabular}
\end{table}

Table~\ref{tab:spin-class-summary} demonstrates second-, fourth-
and sixth-order refinement across the three spin classes without relying
on a vanishing spin-2 singlet interaction. A separate unrescaled spin-2
check gives mass drift below $3.7\times10^{-13}$ and reversal defects
below $1.6\times10^{-14}$ for the three FB orders on $[0,0.1]$.

\paragraph{Independent small-grid verification}
Adaptive DOP853 checks the complete semidiscrete equations on \([-3,3)^3\) with \(6^3\) nodes, \(T=0.08\), and splitting steps \(0.08,0.04,0.02\), without stepwise rescaling. All rotation and SOC directions and the internal coupling are active; spin 2 has \(c_2=0.9\). The archive configurations \nolinkurl{cases/full_spin0p5.json}, \nolinkurl{cases/full_spin1.json}, and \nolinkurl{cases/full_spin2.json} specify the full parameters used by \nolinkurl{python/validate.py}. Table~\ref{tab:independent-ode} reports unaligned S6 errors and rates against the tighter ODE reference; the archive also records orders two and four and phase-aligned checks. These tests include active pairing and do not replace large-grid references.

\begin{table}[htbp]
\centering
\caption{Independent ODE check: raw relative S6 error at \(\tau=0.02\), rate from \(0.04\) to \(0.02\), and relative difference between DOP853 references with \((\mathrm{rtol},\mathrm{atol})=(5\times10^{-13},2\times10^{-15})\) and \((2.3\times10^{-14},2\times10^{-16})\).}
\label{tab:independent-ode}
\begin{tabular*}{\textwidth}{@{\extracolsep{\fill}}lccc@{}}
\toprule
Spin & Raw error & S6 rate & Reference difference\\
\midrule
\(1/2\) & \(1.453\times10^{-11}\) & 5.9931 & \(6.344\times10^{-15}\)\\
1 & \(2.443\times10^{-11}\) & 5.9928 & \(8.929\times10^{-15}\)\\
2 & \(9.174\times10^{-11}\) & 5.9890 & \(1.121\times10^{-14}\)\\
\bottomrule
\end{tabular*}
\end{table}

\subsection{Noncommuting Hermitian matrices}\label{subsec:generic-hermitian-soc}

To test mixed SOC and internal couplings beyond individual Pauli or spin matrices, use
\begin{equation}\label{eq:num-generic-soc-matrices}
\begin{aligned}
A_x&=\sigma_x+0.30\sigma_z+0.20\sigma_y,&
A_y&=0.70\sigma_y+0.20\sigma_x-0.10\sigma_z,\\
A_z&=0.50\sigma_z+0.40\sigma_y+0.15\sigma_x,&
M&=0.60\sigma_x+0.80\sigma_z+0.25\sigma_y.
\end{aligned}
\end{equation}
All six pairwise commutators are nonzero, with Frobenius norms between \(0.9442\) and \(2.0543\), excluding simultaneous diagonalization.

Use the preceding spin-\(1/2\) initial state on the grid \(h=1/4\), with final time \(0.16\), rotation \(\bm\Omega=(0.55,-0.75,0.65)\), and \(\beta=\left(\begin{smallmatrix}1.1&0.7\\0.7&0.9\end{smallmatrix}\right)\). The FB-S6 reference step is \(\tau_{\rm ref}=2.0\times10^{-3}\).

\begin{table}[htbp]
\centering
\caption{Temporal convergence for the spin-\(1/2\) problem with the generic Hermitian SOC and internal matrices in \eqref{eq:num-generic-soc-matrices}.}
\label{tab:generic_soc_temporal_convergence}
{\small
\setlength{\tabcolsep}{4pt}
\begin{tabular}{@{}lcccccc@{}}
\toprule
\(\tau\)
& \multicolumn{2}{c}{FB-S2}
& \multicolumn{2}{c}{FB-S4}
& \multicolumn{2}{c}{FB-S6}\\
\cmidrule(lr){2-3}\cmidrule(lr){4-5}\cmidrule(l){6-7}
& error & order & error & order & error & order\\
\midrule
\(8.0\times10^{-2}\)
& \(6.1439\times10^{-4}\) & --
& \(1.2670\times10^{-5}\) & --
& \(2.2384\times10^{-8}\) & --\\
\(4.0\times10^{-2}\)
& \(1.5333\times10^{-4}\) & \(2.0026\)
& \(7.9759\times10^{-7}\) & \(3.9896\)
& \(3.5332\times10^{-10}\) & \(5.9853\)\\
\(2.0\times10^{-2}\)
& \(3.8314\times10^{-5}\) & \(2.0006\)
& \(4.9941\times10^{-8}\) & \(3.9974\)
& \(5.5408\times10^{-12}\) & \(5.9947\)\\
\(1.0\times10^{-2}\)
& \(9.5776\times10^{-6}\) & \(2.0002\)
& \(3.1227\times10^{-9}\) & \(3.9993\)
& \(1.8544\times10^{-13}\) & \(4.9011\)\\
\bottomrule
\end{tabular}
}
\end{table}

Table~\ref{tab:generic_soc_temporal_convergence} verifies the designed orders without commutativity assumptions. The first two FB-S6 refinements show sixth order; reference and roundoff errors affect the final rate near \(10^{-13}\).

\FloatBarrier
\Needspace{10\baselineskip}
\section{Fused axial spinor EEI: implementation and performance records}\label{sec:fusion-records}

\paragraph{Six computational factors, three transform pairs}
The axial reduction of the seven-factor EEI has five nonidentity operator factors (two phases, two shears and a Fourier multiplier), distinct from the shear-only EEI5 of Liu--Wei--Zhang. For implementation, we split the Fourier multiplier into commuting planar and \(k_z\) parts. The resulting six computational factors, denoted $R_1,\ldots,R_6$ to distinguish them from the seven operator factors $Q_{j,h}$ in the main article, are physical phase \(R_1\), \((x,k_y,z)\) shear \(R_2\), \((k_x,k_y,z)\) planar multiplier \(R_3\), \(k_z\) kinetic multiplier \(R_4\), \((k_x,y,z)\) shear \(R_5\), and final physical phase \(R_6\). Fourier-space fusion reduces directional FFT/inverse-FFT pairs from five to three without changing the grid map.

For a \emph{forward} call at \(t\), the execution is
\[
 R_1\ \xrightarrow{\mathcal F_y}\ R_2\ \xrightarrow{\mathcal F_x}\ R_3\ \xrightarrow{\mathcal F_z}\ R_4\ \xrightarrow{\mathcal F_z^{-1}}\ \xrightarrow{\mathcal F_y^{-1}}\ R_5\ \xrightarrow{\mathcal F_x^{-1}}\ R_6.
\]
The method adjoint \(L_{B,h}(t)=L_{F,h}(-t)^{-1}\) uses coefficients at \(-t\), conjugated in reverse factor order \(R_6,\ldots,R_1\), with transforms \(\mathcal F_x,\mathcal F_z,\mathcal F_z^{-1},\mathcal F_y,\mathcal F_x^{-1},\mathcal F_y^{-1}\). Cancellation occurs only between inverse transforms on the same grid; no full-EEI group law is assumed. Signed composition coefficients, high-level MATLAB actions, spin updates, and SOC multipliers are unchanged.

\paragraph{MATLAB regression and timing protocol}
Regression suites passed 486 and 1458 checks with maximum
relative differences $1.543\times10^{-14}$ and
$2.649\times10^{-14}$ (tolerance $5\times10^{-12}$), covering
signed forward/adjoint maps, spin $1/2$, 1 and 2, CS and FB,
and selected SOC/RSDA settings. Timings used CPU double precision,
MATLAB R2023b Update 8, and a MacBook Pro (dual-core 2.9 GHz
Intel i5, 16 GB RAM, macOS 12.7.6) with default thread settings.

At $N=128$ and $t=0.035$, the No SOC case used $\ell=8$,
$\Omega_z=0.2$, $(c_0,c_1)=(10,1)$, and the perturbed spin-1
state. The SOC + internal case used $\ell=4$, $\Omega_z=0.7$,
$(c_0,c_1)=(2,0.75)$ and nonzero coupling matrices.
Setup and I/O were excluded. Each operation was warmed up and
measured in three alternating single repetitions; only their
medians were retained, so no timing variability estimate is claimed.

\begin{table}[H]
\centering\small
\caption{MATLAB \(128^3\) benchmark medians (seconds). Rel. difference is \(\|U_{\rm fused}-U_{\rm legacy}\|/\max(1,\|U_{\rm legacy}\|)\) for the timed operation. Forward/adjoint denote complete linear calls; the SOC + internal rows include the unchanged coupling operations. \(N=128\), three alternating trials each, one repetition per trial.}
\label{tab:fusion-all-operations}
\begin{tabular}{@{}llrrrc@{}}\toprule
Case & Operation & Legacy (s)& Fused (s)& Ratio& Rel. difference\\\midrule
No SOC&Forward&0.4299&0.2873&0.668&\(5.69\times10^{-16}\)\\
&Adjoint&0.3772&0.3349&0.888&\(5.35\times10^{-16}\)\\
&FB-S4&6.0489&5.4083&0.894&\(1.49\times10^{-15}\)\\
&FB-S6&14.6977&13.3451&0.908&\(2.41\times10^{-15}\)\\\midrule
SOC + internal&Forward&1.2057&0.9604&0.797&\(7.84\times10^{-16}\)\\
&Adjoint&1.3629&1.1075&0.813&\(5.13\times10^{-16}\)\\
&FB-S4&10.1133&9.0260&0.892&\(2.03\times10^{-15}\)\\
&FB-S6&24.5882&21.2424&0.864&\(3.21\times10^{-15}\)\\\bottomrule
\end{tabular}
\end{table}
The full-step reductions are \(10.6\%\), \(9.2\%\), \(10.8\%\), and \(13.6\%\). They compare the same method before and after transform fusion, not FB against CS at equal solution error. The optimization is available to both families. These isolated-operation results do not establish identical long-time floating-point trajectories; the convergence endpoints in Sections~S1 and S5 were generated with the unfused implementation.

\FloatBarrier

\section{Reproducibility and scope of verification}\label{sec:reproducibility}
The reproducibility archive at \url{https://doi.org/10.5281/zenodo.23138002} contains the scalar and spinor solvers, configurations, numerical CSV/MAT records, and scripts for rebuilding the comparison tables and figures from unrounded data. Its source map records the code versions, dataset associations, and file checksums. The scalar CS and FB endpoints were computed in separate runs with matching recorded configurations; the spinor comparisons likewise use a common model, grid and initial state. Historical convergence data are retained with the unfused implementation. The default axial EEI now uses the equivalent mixed-Fourier evaluation documented in Section~S6, whose floating-point operation order differs. No historical endpoint is relabeled as a fused-code run.

Scalar statewise postprocessing used one retained FB-S6 terminal array. No compact-CS terminal array was used for this analysis, so the resulting defect norms and Fourier-tail statistics are not attributed to CS. Regeneration instructions and the identification of the large terminal arrays are given in the archive README.

The stable EEI and spin-1 nonlinear formulas used for the convergence comparisons are retained. The spin-2 code additionally uses the equivalent nonnegative Gram-minor formula of Section~S3 for the singlet frequency near cancellation; this does not affect the spin-1 results. MATLAB verification comprised 19 stable-formula checks, 73 operation-level MATLAB/Python comparisons and 16 compact-stage array comparisons, all passing. The independent small-grid ODE tests verify the additional RSDA-based spinor implementation, not the large-grid EEI compact comparisons. Section~S5 retains the additional FB model-coverage and trajectory data with their own configurations and references. These checks support the stated fixed-grid time-integration results; they do not certify spatial or long-time physical accuracy.

\FloatBarrier
\section{The discrete nonlocal operator in the scalar benchmark}\label{sec:nonlocal-reproduction}

This appendix fixes the grid functional approximating \(\Phi_{\rm P}\) in the scalar dipolar model of the main article. Extend each side length and grid size by a factor of two while keeping the original mesh width. The padded grid has \(2N_\alpha\) nodes on \([-2\ell_\alpha,2\ell_\alpha)\), with coordinates denoted by \(\bm x^p_{\mathbf j}\) and signed frequencies \(k^p\). Let \(\mathcal Z\) insert the density into its centered \(N_x\times N_y\times N_z\) block, filling the rest with zero, and let \(\mathcal C\) crop that same block.

Write \(\varepsilon=\varepsilon_{\rm FSA}>0\) for the smoothing length. Then sample the far-field kernel, with quadrature weight:
\[
 g_{\varepsilon,\mathbf j}=\frac{h_xh_yh_z}{4\pi}
 \begin{cases}
 \operatorname{erf}(|\bm x^p_{\mathbf j}|/\varepsilon)/|\bm x^p_{\mathbf j}|,
       &\bm x^p_{\mathbf j}\ne0,\\
 2/(\sqrt\pi\varepsilon),&\bm x^p_{\mathbf j}=0.
 \end{cases}
\]
Here \(\operatorname{erf}\) is the error function. Let \(\mathcal F_p\) be the unnormalized padded-grid DFT with paired inverse, and \(\widehat g_\varepsilon\) the transform of this kernel after shifting the spatial origin to index zero. Define on the same grid
\[
 b_\varepsilon(k^p)=
 \begin{cases}
 (1-e^{-\varepsilon^2|k^p|^2/4})/|k^p|^2,&k^p\ne0,\\
 \varepsilon^2/4,&k^p=0.
 \end{cases}
\]
For the directional derivative only, replace the Nyquist entry in each component's signed frequency vector by zero, and denote the resulting vector by \(k^{p,(1)}\). This is a convention of the nonlocal FSA multiplier only; it is not the signed-Nyquist derivative matrix \(D_{\alpha,h}\) used in the EEI analysis. Using the dipole direction \(\bm d\) $\bm d=(1/2,\sqrt3/2,0)^T$ used in the main scalar benchmark, the stored multiplier is
\begin{equation}
 M_{\rm P}(k^p)=\operatorname{Re}\left[
 (\widehat g_\varepsilon(k^p)+b_\varepsilon(k^p))
       \bigl(-(\bm d\cdot k^{p,(1)})^2\bigr)\right].
 \label{eq:FSA_multiplier}
\end{equation}
Here \(\widehat g_\varepsilon\) represents the smoothly sampled far-field contribution, while \(b_\varepsilon\) is the analytic correction associated with the kernel split. The actual grid potential is
\begin{equation}
 \Phi_{{\rm P},h}(\rho)=
 \operatorname{Re}\,\mathcal C\,\mathcal F_p^{-1}
 \left(M_{\rm P}\,\mathcal F_p(\mathcal Z\rho)\right).
 \label{eq:FSA_grid_potential}
\end{equation}
This is a real functional of the density, as required by the exact scalar nonlinear subflow. The paired FFT normalization, quadrature factor, centering, crop, and zeroed Nyquist derivative are all part of its definition. In particular, replacing it by a periodic multiplier for \(U_{\rm dd}\) would define a different experiment. Both CS and FB use exactly this nonlocal calculation.

\section{Proof of the exact spin-2 flow and its reversibility}\label{sec:spin2-reversibility}

For completeness, we give the proof of the spin-2 reversibility proposition in the main article, including the endpoint coefficient update needed to justify the method-adjoint identity.

\begin{proof}
Fix a grid point and write \(a=V+(c_0+c_2/5)\rho\), \(H_{\rm sp}=c_1\bm F\cdot\bm f\), \(\Phi_{\rm ph}(\tau)=e^{-i\tau a}I_5\), and \(Q(\tau)=e^{-i\tau H_{\rm sp}}\). The spin identities give \(Q^T\mathcal A Q=\mathcal A\) and \(\mathcal A\overline{Q\Psi}=Q\mathcal A\overline\Psi\). Thus spin rotation preserves \(A_{00}\) and respects the pairing direction.

For the gauged pairing map \(P_{\rm sp}(\tau)\Psi=b\), put \(W=\mathcal A\overline\Psi\). Writing \(\Psi\) for the gauged spinor after extracting \(\Phi_{\rm ph}\) and \(Q\), its equation is
\(\dot\Psi=ic_2\bigl(\alpha\Psi-A_{00}\mathcal A\overline\Psi\bigr),\qquad \alpha=\rho/5.\)
This flow preserves density, since \(\partial_t(\Psi^*\Psi)=2\operatorname{Re}[ic_2(\alpha\rho-|A_{00}|^2)]=0\). Moreover, \(\dot A_{00}=2ic_2(\alpha A_{00}-\rho A_{00}/5)=0\). Thus both coefficients are constant in the gauged system. Using \(\mathcal A^2=I_5/5\) to differentiate \(W\) gives
\begin{equation}\label{eq:spin2-pairing-matrix-sup}
\frac{\dd}{\dd t}\begin{pmatrix}\Psi\\W\end{pmatrix}
=ic_2K\begin{pmatrix}\Psi\\W\end{pmatrix},\qquad
K=\begin{pmatrix}\alpha&-A_{00}\\\overline{A_{00}}/5&-\alpha\end{pmatrix}.
\end{equation}
Since \(K^2=\chi^2I_2\), its exponential is \(e^{ic_2\tau K}=\cos(c_2\chi\tau)I_2+ic_2\tau\operatorname{sinc}(c_2\chi\tau)K\); the first component is the singlet-pairing update in the main article. Hence \(P_{\rm sp}(-\tau)=P_{\rm sp}(\tau)^{-1}\).

It remains to check the coefficients on a return step. For a scalar phase \(z\), \(|z|=1\), the input change \(\Psi\mapsto z\Psi\) gives \(A_{00}\mapsto z^2A_{00}\) and \(\mathcal A\overline\Psi\mapsto z^{-1}\mathcal A\overline\Psi\). Therefore
\(P_{\rm sp}(\tau)(z\Psi)=zP_{\rm sp}(\tau)\Psi.\)
The spin identities similarly give \(P_{\rm sp}(\tau)(Q\Psi)=QP_{\rm sp}(\tau)\Psi\). The pairing flow also preserves \(\bm F\), by the same spin identities used above. Hence \(\Phi_{\rm ph},Q,P_{\rm sp}\) commute as state-dependent flow maps, and their product is the exact flow of the spin-2 local equation in the main article.

At the endpoint, \(A_{00}\) becomes \(e^{-2i\tau a}A_{00}\). Equivariance accounts for this change in the recomputed negative step. Since \(\rho,\bm F,\chi\) are unchanged, that step inverts all three factors. Unitarity of \(\Phi_{\rm ph},Q\) and density preservation of \(P_{\rm sp}\) prove the claim.
\end{proof}

\begingroup
\footnotesize
\setlength{\bibsep}{0pt}
\bibliographystyle{plainnat}
\bibliography{BEC_refs}
\endgroup